\documentclass[12pt,twoside]{amsart}
\usepackage[margin=3cm]{geometry}
\usepackage[colorlinks=false]{hyperref}
\usepackage[english]{babel}
\usepackage{graphicx,titling}
\usepackage{float}
\usepackage{amsmath,amsfonts,amssymb,amsthm}
\usepackage{lipsum}
\usepackage[T1]{fontenc}
\usepackage{fourier}
\usepackage{color}
\usepackage[latin1]{inputenc}
\usepackage{esint}
\usepackage{caption}
\usepackage{ccicons}

\makeatletter
\def\blfootnote{\xdef\@thefnmark{}\@footnotetext}
\makeatother

\newcommand\ccnote{
    \blfootnote{\copyright\,\, Herbert Koch, Angkana R\"uland, and Mikko Salo}
    \blfootnote{\ccLogo\, \ccAttribution\,\, Licensed under a \href{https://creativecommons.org/licenses/by/4.0/}{Creative Commons Attribution License (CC-BY)}.}
}

\usepackage[export]{adjustbox}
\numberwithin{equation}{section}
\usepackage{setspace}
\renewcommand{\le}{\leqslant}
\renewcommand{\leq}{\leqslant}

\renewcommand{\geq}{\geqslant}
\renewcommand{\mathbb}{\varmathbb}
\usepackage{fancyhdr}
\newtheorem{theorem}{Theorem}[section]
\newtheorem{lemma}[theorem]{Lemma}

\newtheorem{proposition}[theorem]{Proposition}
\newtheorem{definition}[theorem]{Definition}
\newtheorem{remark}[theorem]{Remark}
\newtheorem{example}[theorem]{Example}

\newcommand {\R} {\mathbb{R}} \newcommand {\Z} {\mathbb{Z}}
\newcommand {\T} {\mathbb{T}} \newcommand {\N} {\mathbb{N}}
\newcommand {\C} {\mathbb{C}} 
\newcommand {\p} {\partial}

\newcommand {\D} {\Delta}

\DeclareMathOperator{\tr}{tr}

\DeclareMathOperator {\dist} {dist}
\DeclareMathOperator {\Ree} {Re}

\DeclareMathOperator{\spa} {span}

\DeclareMathOperator{\F} {\mathcal{F}}

\newcommand{\mR}{\mathbb{R}}                    % Formatting for R
\newcommand{\mC}{\mathbb{C}}                    % Formatting for C
\newcommand{\mZ}{\mathbb{Z}}                    % Formatting for Z
\newcommand{\abs}[1]{\lvert #1 \rvert}          % Formatting for the absolute value
\newcommand{\norm}[1]{\lVert #1 \rVert}         % Formatting for the norm
\newcommand{\br}[1]{\langle #1 \rangle}         % Formatting for the inner product

\newcommand{\ol}[1]{\overline{#1}}

\newcommand{\eps}{\varepsilon}

\newcommand{\supp}{\mathrm{supp}}

\address{Herbert Koch, Mathematisches Institut, Universit\"at Bonn, Endenicher Allee 60, 53115 Bonn, Germany}
\email{koch@math.uni-bonn.de}
\medskip
\address{Angkana R\"uland, Institut f\"ur Angewandte Mathematik, Ruprecht-Karls-Universit\"at Heidelberg, Im Neuenheimer Feld 205, 69120 Heidelberg, Germany; Current address: Institut für Angewandte Mathematik, Universit\"at Bonn, Endenicher Allee 60, 53115 Bonn, Germany}
\email{Angkana.Rueland@uni-heidelberg.de}
\medskip
\address{Mikko Salo, University of Jyvaskyla, Department of Mathematics and Statistics, PO Box 35, 40014 University of Jyvaskyla, Finland}
\email{mikko.j.salo@jyu.fi}

\begin{document}

\thispagestyle{empty}

\begin{minipage}{0.28\textwidth}
\begin{figure}[H]
%\centering
\includegraphics[width=2.5cm,height=2.5cm,left]{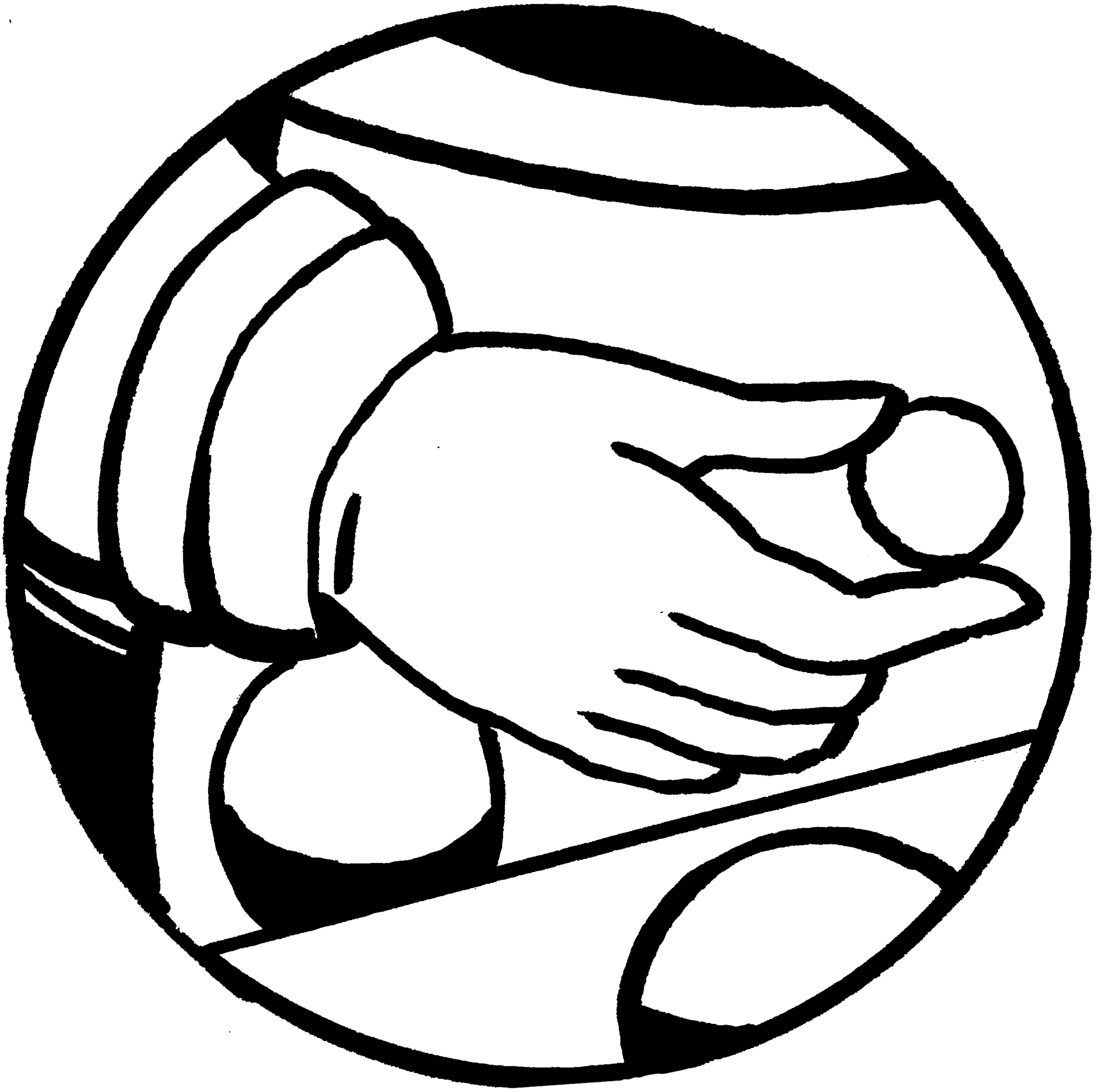}
\end{figure}
\end{minipage}
\begin{minipage}{0.7\textwidth} 
\begin{flushright}
%% The following metadata, in particular
%% the Paper No. and the DOI will be inserted by the journal
Ars Inveniendi Analytica (2025), Revised Paper No. 1, 94 pp.
\\
DOI 10.15781/kvf0-ds86
\\
ISSN: 2769-8505
\end{flushright}
\end{minipage}

\ccnote

\vspace{1cm}

%%      -------------------------------------------------------------------------------
%%      -------------------------- TITLE ----------------------------
%%      -------------------------------------------------------------------------------
%% Authors, please put here the full title of the article

\begin{center}
\begin{huge}
\textit{On  instability mechanisms\\for inverse problems}
\end{huge}
\end{center}

\vspace{1cm}

%%      -------------------------------------------------------------------------------
%%      -------------------------- AUTHORS AND AFFILIATIONS ----------------------------
%%      -------------------------------------------------------------------------------
%% Authors, please put here your full names and affiliations

\begin{minipage}[t]{.28\textwidth}
\begin{center}
{\large{\bf{Herbert Koch}}} \\
\vskip0.15cm
\footnotesize{Universit\"at Bonn}
\end{center}
\end{minipage}
\hfill
\noindent
\begin{minipage}[t]{.28\textwidth}
\begin{center}
{\large{\bf{Angkana R\"uland}}} \\
\vskip0.15cm
\footnotesize{Ruprecht-Karls-Universit\"at Heidelberg}\\
\footnotesize{Current address: Universit\"at Bonn}
\end{center}
\end{minipage}
\hfill
\noindent
\begin{minipage}[t]{.28\textwidth}
\begin{center}
{\large{\bf{Mikko Salo}}} \\
\vskip0.15cm
\footnotesize{University of Jyvaskyla}
\end{center}
\end{minipage}

\vspace{1cm}

%%% Please replace "James Mustard" below 
%%% with the name of the managing editor for your submission.
%%% If you are unsure about their identity
%%% please ask an editor-in-chief about.

\begin{center}
\noindent \em{Revision communicated by Francesco Maggi}
\end{center}
\vspace{1cm}

%%      -------------------------------------------------------------------------------
%%      -------------------------- BEGIN ABSTRACT ----------------------------
%%      -------------------------------------------------------------------------------
%% Authors, please put here the ABSTRACT and KEYBOARDS

\noindent \textbf{Abstract.} \textit{In this article we present three robust instability mechanisms for linear and nonlinear inverse problems. All of these are based on strong compression properties (in the sense of singular value or entropy number bounds) which we deduce through either strong global smoothing, only weak global smoothing or microlocal smoothing for the corresponding forward operators, respectively. As applications we for instance present new instability arguments for unique continuation, for the backward heat equation and for linear and nonlinear Calder\'on type problems in general geometries, possibly in the presence of rough coefficients. 
Our instability mechanisms could also be of interest in the context of control theory, providing estimates on the cost of (approximate) controllability in rather general settings. This is a revised version of the article ``On instability mechanisms for inverse problems'' Ars Inveniendi Analytica (2021), Paper No. 7, 93 pp by the same authors.}
\vskip0.3cm

\vspace{0.5cm}

%%      -------------------------------------------------------------------------------
%%      -------------------------- BEGIN ARTICLE ----------------------------
%%      -------------------------------------------------------------------------------
%% Authors, copy the body of your paper here

\section{Introduction} \label{sec_introduction}

The purpose of this article is to revisit and to systematically extend instability arguments for certain classes of inverse problems. In such problems the forward operator does not have a continuous inverse, and these problems are thus automatically unstable. In fact instability, or ill-posedness, lies at the heart of inverse problems research and much of the related literature addresses various aspects of it. We mention very briefly a few selected aspects:

\begin{itemize}
\item 
Inverting the two-dimensional Radon transform, or the geodesic X-ray transform in geometries without conjugate points, is a mildly ill-posed inverse problem \cite{Natterer_book, AS20}. Variants involving limited data or geometries with conjugate points may be severely ill-posed \cite{Natterer_book, SU04, MSU15}.
\item 
For the boundary rigidity problem a stability theory has been developed in \cite{SU04, SU08}. A main feature is that suitable stability for the linearized problem implies local uniqueness and stability for the corresponding nonlinear problem. Corresponding abstract results are given in \cite{SU09}.
\item 
In the Calder\'on problem one has logarithmic stability \cite{Alessandrini} and this is optimal in general \cite{Mandache}. Stability improves if the unknowns are restricted to a finite dimensional space \cite{AlessandriniVessella}. For certain partial data or anisotropic variants only double logarithmic stability is known \cite{CaroSalo, CDR}.
\item 
For the wave equation, the problem of determining coefficients from the hyperbolic Dirichlet-to-Neumann map is often quite stable in geometries without conjugate points \cite{SU05, Montalto}. In the presence of conjugate points only weak stability results are known \cite{AKKLT, BKL}.
\item 
Even in the presence of ill-posedness, regularization can be used to reduce the effect of measurement errors and to obtain more stable approximate reconstructions \cite{EHN}.
\end{itemize}

Many of the above results focus on showing that the inverse problem in question has at least a certain (weak) degree of stability. In this article we consider corresponding \emph{instability} results, showing that a certain type of instability is inherent to the inverse problem and cannot be improved. In the literature several arguments deducing the degree of instability are known also in cases where direct singular value computations are not immediately feasible. Prominent examples for linear inverse problems are based on (microlocal) regularization for the forward operator e.g.\ in the setting of the geodesic X-ray transform and related problems, see \cite{Quinto, Quinto1, SU09, StefanovUhlmann, SU13, MSU15, HU18}. For nonlinear Calder\'on type inverse problems there is a method due to Mandache \cite{Mandache} based on the construction of certain exponentially decaying bases and entropy arguments, see also \cite{DiCristoRondi,Isaev,IsaevI,IsaevII, A07, ZZ19}. However, the latter results rely on strong assumptions on the operators (e.g.\ constant coefficients in the principal part) or on the geometry in which the problem is phrased (e.g.\ in the form of symmetry of the underlying domain such that separation of variables arguments can be used). 
Even for linear inverse problems the available results showing that logarithmic stability is optimal often rely on explicit computations for special operators and geometries, see e.g.\ \cite{AnderssonBoman} or the classical works \cite{Hadamard, John} for instability of the unique continuation principle. We further remark that, in general, there is no direct relation between the instability of a nonlinear inverse problem and its linearization \cite[Appendix]{EKN89}.

As a key novelty of this article, we present three robust instability mechanisms which can be applied to many different inverse problems involving general variable coefficient operators in general domains. A prototypical setting to which this applies are Calder\'on type problems (both linearized and nonlinear) in nonsmooth domains with coefficients having low, scaling critical regularity. As other examples we consider instability in general settings for the unique continuation principle for elliptic and hyperbolic equations, instability for the backward heat equation, the cost of approximate controllability in various problems from control theory (of elliptic, parabolic and hyperbolic type), and examples involving microlocal smoothing such as Radon transforms with limited data. An important aspect of these instability mechanisms is that they allow us to read off the degree of instability of an inverse problem from directly accessible properties of the forward operator (which may now lack symmetry or high regularity). 
In spite of their wide applicability, a central feature of our arguments which is shared by all three mechanisms is that while a stability result for an inverse problem is ``hard'' (one needs a quantitative version of a uniqueness result), the corresponding instability result is often rather ``soft''. 

Before turning to the explicit mechanisms, let us formulate the abstract set-up of the problem we are interested in. To this end, consider a (possibly nonlinear) map 
\begin{align*}
F:X \rightarrow Y,
\end{align*}
where $X,Y$ are metric spaces (often Banach spaces) and $F$ is continuous injective map modelling the forward operator of the inverse problem under consideration. We then seek to investigate the stability properties of the associated inverse problem: given $y \in F(X)$, find $x \in X$ so that 
\[
F(x) = y.
\]
In typical inverse problems the map $F: X \to F(X)$ does not have a continuous inverse. Yet, it is well-known that when restricted to a compact set $K\subset X$ the operator $F|_{K}:K\mapsto F(K)$ is a homeomorphism whose inverse has some modulus of continuity $\omega$ (see Lemma \ref{lemma_modulus_abstract}). It is the purpose of this article to identify general properties of the forward operator $F$ yielding information on the possible moduli $\omega$ in dependence of the choice of $X,Y,K$.

In the sequel, we discuss the three instability mechanisms and address prominent examples to which these apply. We however emphasise that the methods which we introduce in this article are applicable to a much larger class of problems and that we have only selected a number of expository examples to illustrate these.

\subsection{Instability by strong (analytic) smoothing properties of the forward operator}

\label{sec:instI}

In our first instability mechanism we are mainly motivated by inverse problems for elliptic or parabolic equations. Here one of the most prominent examples which had been studied by Mandache \cite{Mandache} and Di Cristo-Rondi \cite{DiCristoRondi} is the instability of the Calder\'on problem. In this problem which goes back to work of Calder\'on \cite{Calderon} one is interested in recovering a positive conductivity function $\gamma$ in the equation
\begin{align}
\label{eq:Calderon}
\begin{split}
\nabla \cdot \gamma \nabla u &= 0 \mbox{ in } \Omega \subset \R^n,\\
u & = f \mbox{ on } \partial \Omega,
\end{split}
\end{align}
from the knowledge of the Dirichlet-to-Neumann map 
\begin{align*}
\Lambda_{\gamma}: f \mapsto \gamma \p_{\nu} u|_{\partial \Omega},
\end{align*}
where $\p_{\nu} u|_{\partial \Omega}$ denotes the normal derivative of $u$ on the boundary (viewed in an appropriate weak sense). This problem is further closely related to the problem of recovering the potential $q$ in the Schr\"odinger equation 
\begin{align}
\label{eq:Gelfand}
\begin{split}
-\D u + q u &= 0 \mbox{ in } \Omega, \\
u & = f \mbox{ on } \partial \Omega,
\end{split}
\end{align}
from the knowledge of the Dirichlet-to-Neumann operator
\begin{align*}
\Lambda_q: f \mapsto \partial_{\nu }u|_{\partial \Omega}.
\end{align*}
The Calder\'on problem with a sufficiently regular scalar conductivity can be reduced to this problem by a Liouville transform. For simplicity, we assume that $q$ is such that zero is not a Dirichlet eigenvalue of the equation \eqref{eq:Gelfand}.
Both problems are well-studied non-linear inverse problems with (for $n\geq 3$) uniqueness proofs at different regularities \cite{SylvesterUhlmann, Novikov, Chanillo, LavineNachman, Haberman_conductivity, Haberman_magnetic}, stability estimates \cite{Alessandrini} and reconstruction algorithms \cite{Nachman} available (we refer to \cite{Uhlmann_survey} for an overview of the known results and further references). We are mainly interested in the stability properties of these problems. By virtue of the work of Alessandrini \cite{Alessandrini} it is known that both problems (in suitable regularity scales) enjoy logarithmic stability estimates under mild a priori assumptions on the data, e.g.\ if $q_1, q_2\in L^{\infty}(\Omega)$ with $\|q_j\|_{L^{\infty}(\Omega)} \leq M$, there exists a constant $C = C(\Omega, n, M)>0$ such that
\begin{align*}
\|q_1-q_2\|_{H^{-1}(\Omega)}
 \leq \omega(\|\Lambda_{q_1}-\Lambda_{q_2}\|_{H^{\frac{1}{2}}(\partial \Omega) \rightarrow H^{-\frac{1}{2}}(\p \Omega)}),
\end{align*}
where $\omega(t) \leq C \abs{\log\,t}^{-\frac{2}{n+2}}$ for $t$ small. A natural question then concerns the optimality of these estimates. This was studied by Mandache \cite{Mandache} who proved the following necessity of the exponential instability:

\begin{theorem}[\cite{Mandache}]
\label{thm:Mandache_Cal}
Let $B_1(0) \subset \R^n$ for $n\geq 2$ denote the unit ball. Let $q_0 \in L^{\infty}(B_1)\cap C^m(\overline{B}_1) $ with $\supp(q_0) \subset B_{1/2}(0)$. Then there is $\eps_0>0$ such that for any $\eps\in (0,\eps_0)$ and $m \geq 1$ there exist potentials  $q_1, q_2 \in C^m(\ol{B}_1)\cap L^{\infty}(B_1)$ with $\supp(q_j)\subset B_{1/2}(0)$ and constants $C>0$, $\gamma>0$ such that we have
\begin{align*}
\|\Lambda_{q_1}- \Lambda_{q_2}\|_{H^{\frac{1}{2}}(\partial B_1) \rightarrow H^{-\frac{1}{2}}(\partial B_1)} 
&\leq  \exp(-C\eps^{-\gamma}),\\
\|q_1 -q_2\|_{L^{\infty}(B_1)}& = \eps,\\
\|q_j - q_0\|_{L^{\infty}(B_1)} & \leq \eps, \ j \in \{1,2\},\\
\|q_j - q_0\|_{C^m(\ol{B}_1)} & \leq 1.
\end{align*}
\end{theorem}

We emphasize that bounds as in Theorem \ref{thm:Mandache_Cal} directly imply conditions on possible moduli of continuity for the Calder\'on problem showing that for the function spaces from the theorem the modulus of continuity \emph{cannot} be better than of logarithmic type. More precisely, if for any potentials $q$ with $\|q\|_{C^m(\ol{B}_1)}\leq 1$ and $r>0$ an estimate of the form 
\begin{align*}
\|q_1- q_2\|_{L^{\infty}(B_1)} \leq \omega(\|\Lambda_{q_1}-\Lambda_{q_2}\|_{H^{\frac{1}{2}}(\partial \Omega) \rightarrow H^{-\frac{1}{2}}(\p \Omega)})
\end{align*}
holds for $ \Vert q_j-q \Vert_{L^\infty(B_1)} < r$, $ \Vert q_j-q \Vert_{C^m(\overline{B_1})}\le 1$, then necessarily $\omega(t)\geq C|\log(t)|^{-\frac{1}{\gamma}}$.

Mandache's argument relied on two key observations which were later systematically investigated and extended by Di Cristo-Rondi \cite{DiCristoRondi}:
\begin{itemize}
\item[(i)] For the setting of the unit ball $B_{1}(0)$ it is possible to explicitly construct a basis $(f_j)_{j=1}^{\infty} \subset H^{\frac{1}{2}}(\partial B_1(0))$ such that
\begin{align*}
|((\Lambda_{q}-\Lambda_0)f_{j_1}, f_{j_2})_{L^2(\partial B_1)}| \leq C e^{-c\max\{j_1, j_2\}} ,
\end{align*}
for some constants $c,C>0$.
\item[(ii)] Using (i), on the one hand, Mandache proved an entropy estimate showing that the map $q \mapsto \Gamma(q) := \Lambda_{q}-\Lambda_0$ is highly compressing (with respect to suitable Hilbert-Schmidt type norms). This means that the image of a ball in $L^{\infty}(B_1)$ under the operator $\Gamma$ can be covered by a relatively small $\delta$-net. On the other hand, the metric space of admissible potentials has a large $\eps$-discrete set, which corresponds to a capacity estimate (see Section \ref{sec:abstract} for the corresponding definitions). Relying on a pigeonhole argument, the comparison of the entropy vs the capacity estimate thus leads to the result of Theorem \ref{thm:Mandache_Cal}.
\end{itemize}
While the argument in (ii) is rather general (and mainly based on entropy and capacity estimates), (i) is a rather problem specific ingredient in Mandache's original work. It highly uses the interplay of the geometry of $B_1(0)$ and the constant coefficient Laplacian in separating variables and working with spherical harmonics. 

Thus, natural questions which arise are:
\begin{itemize}
\item[(Q1)]
Does the exponential instability remain true in more general geometries which enjoy less symmetry than the unit ball? 
\item[(Q2)] Do the instability arguments persist for more general elliptic operators with variable coefficients (i.e., with $L_0 = \p_i a^{ij} \p_j + \p_j b_j^1 + b_j^2 \p_j + c $)?
\end{itemize}

Our first instability mechanism gives a positive answer to these questions in the presence of \emph{strong (analytic) smoothing properties}: We prove that if the forward operator is analytically smoothing, then exponential instability in the associated inverse problem is unavoidable. This in particular allows us, for instance, to reprove Mandache type instability results for Calder\'on type inverse problems in rather general (analytic) geometries with analytic background metrics. One of the novelties here is that we obtain instability results for general geometries and coefficients under real-analyticity assumptions, whereas most earlier results have been restricted to special geometries that allow explicit computations based on separation of variables. It is possible to establish a rather general, abstract framework for this (see Section \ref{sec:abstract}).

As an example of this abstract scheme, we present a new proof of the exponential instability of the Calder\'on problem relying on abstract smoothing arguments (see Section \ref{sec:examples1} for the details and proofs).

\begin{theorem}
\label{thm:Cald1}
Let $(M,g)$ be a compact Riemannian $n$-manifold with smooth boundary $\partial M$. Let $M' \Subset M^{\mathrm{int}}$.
Let $\lambda_1>0$ denote the first Dirichlet eigenvalue of the Laplace-Beltrami operator on $(M,g)$ and let $C>0$.
Fix $s > \frac{n}{2}$ and $\delta > 0$, and suppose that 
\[
\norm{q_1-q_2}_{H^s(M)} \leq \omega(\norm{\Lambda_{q_1}-\Lambda_{q_2}}_{H^{1/2}(\p M) \to H^{-1/2}(\p M)}),
\]
whenever $\norm{q_j}_{L^{\infty}(M)} \leq \lambda_1/2$, $\norm{q_j}_{H^{s+\delta}(M)} \leq C$ and $\supp(q_j) \subset M'$. Then $\omega(t)$ cannot be a H\"older modulus of continuity. Moreover, if $M$, $g$ and $\p M$ are real-analytic, then $\omega(t) \gtrsim \abs{\log t}^{-\frac{\delta(2n-1)}{n}}$ for $t$ small.
\end{theorem}

\begin{remark}
We highlight that while the formulation of the instability statement in Theorem \ref{thm:Mandache_Cal} may appear to be more quantitative at first sight than the formulation of our Theorem \ref{thm:Cald1}, under very mild conditions both types instability formulations are indeed equivalent. We refer to Remark \ref{rmk:equiv_Mandache} for a detailed discussion of this.
\end{remark}

In contrast to the argument of Mandache, a key feature of this first instability mechanism is the robustness of our arguments with respect to (analytic) domain or coefficient perturbations. Instead of constructing explicit singular value bases by hand as in the first step of Mandache's argument, we use certain analytic smoothing properties and abstract entropy/capacity number arguments. These allow us to conclude the existence of corresponding singular value bases for which the operator has exponential decay.
We remark that this is \emph{not} restricted to linear problems and that similar arguments apply to any inverse problem where the forward operator has suitable analytic smoothing properties.

\subsection{Instability through an iterated small regularity gain} 
\label{sec:instII}

While our first instability mechanism from Section \ref{sec:instI} allows us to generalize the Mandache type exponential instability results to more general geometries, it still relies on very strong (real-analytic) smoothing properties of the forward operators and can only be used if the underlying structures (coefficients and boundaries) are real-analytic. A further natural question thus addresses the necessary regularity of the forward operator. 

\begin{itemize}
\item[(Q3)] Are the analytic smoothing properties of the forward operator necessary for exponential instability of the inverse problem? Can one find operators with possibly irregular coefficients so that there are stronger stability results?
\end{itemize}

In the case of the Calder\'on problem there are particularly three points which have an influence on the smoothing properties of the forward operator: These are the regularity of the underlying domain, the regularity of the coefficients of the operator $L_0$ and the fact that the potential $q_0$ is still chosen to vanish near $\p \Omega$. It is thus natural to wonder whether it is possible to weaken these assumptions and whether this has an effect on the instability properties of the inverse problem.

As our second mechanism we show that (in mainly elliptic and parabolic inverse problems) it is \emph{not} necessary to have a high degree of smoothing for instability. In these cases it is not possible to construct operators at lower regularity for which better stability estimates hold.
On the contrary, we show that by a suitable iteration already a very small amount of regularity in the equation suffices to deduce (exponential) instability in the inverse problem. As an example of this type of instability mechanism, we present instability results for the Calder\'on problem with scaling critical low regularity metrics and potentials which are non-smooth up to the boundary and prove that even in this framework (in which e.g.\ the unique continuation principle in general fails) one can prove logarithmic instability.

\begin{theorem}
\label{thm:thm_Calderon}
Let $\Omega \subset \R^n$ with $n\geq 2$ be a bounded Lipschitz domain and let 
\[
L = - \p_j a^{jk} \p_k + b_j \p_j + \p_j c_j + q_0
\]
with $(a^{jk}) \in L^{\infty}(\Omega, \R^{n\times n})$ symmetric and uniformly elliptic, $b_j, c_j \in L^{n}(\Omega)$ and $q_0 \in L^{\frac{n}{2}}(\Omega)$. Assume that zero is not a Dirichlet eigenvalue of $L$. Denote the associated Dirichlet-to-Neumann map by $\Lambda_{L}$. 
Then, for any $\delta>0$ there are $\eps_0 >0$ and $C>0$ such that for any $\eps \in (0,\eps_0)$, there exist potentials $q_1, q_2 \in  W^{\delta,\frac{n}{2}}(\Omega)$ with 
\begin{align*}
\|\Lambda_{L + q_1}- \Lambda_{L + q_2}\|_{H^{\frac{1}{2}}(\partial \Omega) \rightarrow H^{-\frac{1}{2}}(\partial \Omega)} 
&\leq  \exp(-C\eps^{-\frac{n}{\delta(2n+1)}}),\\
\|q_1 -q_2\|_{L^{\frac{n}{2}}(\Omega)}&\geq \eps,\\
\|q_j \|_{W^{\delta,\frac{n}{2}}(\Omega)} & \leq 1, \ j \in \{1,2\}.
\end{align*}
\end{theorem}

In the setting of the Calder\'on problem this gives a complete answer to the questions (Q1)-(Q3) showing that even in settings in which the unique continuation property fails, one can at best hope for logarithmic stability in the inverse problem. 

In contrast to the arguments in the previous section, where we deduced the decay of the singular values of the relevant maps through a high degree of regularity, we here deduce the decay of the singular values through an iteration of only a very small amount of regularity. This improvement allows us to treat equations and systems with rough coefficients in (relatively) rough domains. Related observations which however build on the separability of Green's functions have been used in the numerical analysis literature to prove (quantitative) approximation properties by means of hierarchical matrices \cite{BH03, EZ18}. Also, in these works only minimal regularity -- encoded in the validity of Caccioppoli inequalities -- is required.

While the Calder\'on problem is a prototypical example of a nonlinear inverse problem to which these ideas apply, we emphasize that they are not restricted to this specific problem but could also be used to explain the instability of related inverse problems such as for instance the fractional Calder\'on problem \cite{RulandSalo_instability}, inverse inclusion or scattering problems as treated in \cite{DiCristoRondi} and the deterioration of Lipschitz stability estimates for finite dimensional problems depending on the available degrees of freedom \cite{R06}. 
As further applications, we employ these ideas to prove that the unique continuation property for elliptic equations up to the boundary can at best have a logarithmic modulus of continuity (see Theorem \ref{thm:UCP_instab}). This shows that the classical instability result of Hadamard \cite{Hadamard} (see also \cite{AlessandriniRondiRossetVessella}), involving the constant coefficient Laplacian in a flat domain, remains valid for quite general coefficients and geometries.
The instability of unique continuation is also closely related to estimates on the \emph{cost of approximation} in control theoretic problems. For instance, in Section \ref{sec:control} we present a sample result of this for lower bounds on the cost of controllability for the heat equation under low regularity assumptions, see Theorem \ref{thm:controllability_heat_cost}.

\subsection{Instability through microlocal smoothing properties}

\label{sec:instabIII}

As a last instability mechanism, we discuss microlocal smoothing properties.

In contrast to the previous two mechanisms, this type of instability mechanism does not require the forward operator to be \emph{globally} smoothing. Instabilities of different type can already be deduced from microlocal smoothing properties. This includes for instance instability properties of the limited data Radon transform \cite{Quinto, Quinto1} as well as the super-polynomial instability of the geodesic X-ray transform \cite{StefanovUhlmann, MSU15, HU18}. We will discuss these examples and give an abstract framework for this in Section \ref{sec:micro_smooth_1}.

A further example to which this applies, and which might be of interest also to the control theory community, is the logarithmic instability of the unique continuation property for the wave equation in the absence of the geometric control condition of Bardos-Lebeau-Rauch \cite{BardosLebeauRauch}:

\begin{theorem}
Let $(M,g)$ be an analytic, closed Riemannian manifold. Let $\Omega \subset M$ be an open set. Assume that $T > 2\max\{\dist(x,\Omega), \ x \in M\}$ and assume that there exists a ray of geometric optics which does not intersect $\Omega \times [0,T]$.
Consider the solution to
\begin{align*}
\p_t^2 u - \D_g u & = 0 \mbox{ in } M \times (0,T),\\
(u, \p_t u) & =(u_0,u_1) \mbox{ on } M \times \{0\}.
\end{align*}
Suppose that there is an inequality of the form
\begin{align*}
\|(u_0,u_1)\|_{L^2(M)\times H^{-1}(M)}
\leq \omega\left( \frac{\|u\|_{L^2((0,T),H^1(\Omega))}}{\|(u_0,u_1)\|_{H^1(M)\times L^2(M)}} \right) \|(u_0,u_1)\|_{H^1(M)\times L^2(M)}
\end{align*}
for some modulus of continuity $\omega$. Then for any $\eps>0$ there exists $C_{\eps}>0$ such that we have $\omega(t)\geq C_{\eps} |\log(t)|^{- \frac{n+2}{n}-\eps}$ for $t$ small.
\end{theorem}

Although this is a known result in the control theory community (see for instance \cite{Lebeau} or the recent article \cite{LaurentLeautaud} providing corresponding stability estimates), our argument presents a simple and conceptually very clear way of obtaining this instability property. It shows that in particular our third mechanism applies in non-elliptic and non-parabolic contexts.

\medskip

All in all, the three outlined mechanisms provide methods of deducing instability in inverse and control theory problems in a ``soft'' but robust way.
In particular our second instability mechanism shows that high order regularity is not the only mechanism leading to instability. Instability should rather be viewed as relying on a \emph{strong compression mechanism}.

\subsection{Outline of the article}

The remainder of the article is organised as follows. Section \ref{sec:pre} includes some useful facts related to abstract inverse problems, singular values, Weyl asymptotics, and Sobolev and Gevrey/real-analytic spaces. In Section \ref{sec:abstract} we formulate the abstract instability framework based on entropy and capacity estimates. Entropy numbers for various embeddings between spaces of functions or operators are also discussed. Next, in Section \ref{sec:examples1}, we present our first robust instability mechanism, the ``analytic smoothing implies exponential instability'' argument outlined in Section \ref{sec:instI}. As examples, we discuss the instability of unique continuation and the linearized and nonlinear Calder\'on problems for $C^{\infty}$ or analytic coefficients. In Section \ref{sec:instab_low_reg} we then address the instability arguments in the presence of low regularity as outlined in Section \ref{sec:instII}. In addition to the low regularity Calder\'on problem we also investigate here linear problems such as the backward heat equation or unique continuation at low regularity. In Section \ref{sec:micro_smooth_1} we discuss our last instability mechanism based on microlocal smoothing properties, and apply this to the limited data Radon transform, geodesic X-ray transform and an inverse problem for the transport equation.

The article includes four appendices. Appendix \ref{sec:abstr} presents a framework for instability properties of linear inverse problems based on singular value arguments. In the linear case this gives an alternative, more direct approach than the entropy/capacity approach in Section \ref{sec:abstract}. Appendix \ref{sec_appendix_nonlinear} includes some (technical) proofs from Section \ref{sec:pre}. Appendix \ref{sec:Carl} discusses optimal stability for interior unique continuation, and shows how Carleman estimates can lead to complementary lower bounds on the singular values of the relevant operators. Finally, Appendix \ref{sec:Gevrey} includes some facts related to Gevrey/analytic pseudodifferential operators.

\subsection*{Notation}
We will write $A \lesssim B$ (resp.\ $A \gtrsim B$) if $A \leq C B$ (resp.\ $A \geq c B$) for some constants $C, c > 0$ which do not depend on asymptotic parameters. We also write $A \sim B$ if $c A \leq B \leq C A$ for such constants $C, c > 0$. For example, the singular value estimate $\sigma_k \sim k^{-s}$ means that $c k^{-s} \leq \sigma_k \leq C k^{-s}$ uniformly over $k \geq 1$, and the modulus of continuity estimate $\omega(t) \gtrsim t^{\alpha}$ for small $t$ means that $\omega(t) \geq c t^{\alpha}$ when $0 \leq t \leq t_0$ for some small $t_0 > 0$. We will also use the Einstein summation convention where a repeated index in upper and lower position is summed, e.g.\ $\p_j(a^{jk} \p_k u)$ denotes $\sum_{j,k=1}^n \p_j(a^{jk} \p_k u)$ if $u$ is a function in a domain in $\mR^n$.

\subsection*{Acknowledgements}

The authors are deeply grateful to Luca Rondi, who has made important contributions to this article. Conversations between the last named author and Rondi initiated this research, and the presentation has improved considerably thanks to several helpful comments from Rondi.  The authors are also grateful to Thorsten Hohage, Barbara Kaltenbacher, Richard Nickl and Hongkai Zhao for providing relevant references, and to the anonymous referees for helpful suggestions. The authors would further like to thank Leonard Busch and Lauri Oksanen for pointing out various misprints in the first published version of this article.
H.K. was partially supported by the Deutsche Forschungsgemeinschaft (DFG, German Research Foundation) through the Hausdorff Center for Mathematics under Germany's Excellence Strategy - EXC-2047/1 - 390685813 and through CRC 1060 - project number 211504053. 
A.R.\ was supported by the Deutsche Forschungsgemeinschaft (DFG, German Research Foundation) under Germany's Excellence Strategy EXC-2181/1 - 390900948 (the Heidelberg STRUCTURES Cluster of Excellence). She would also like to thank the MPI MIS where part of this work was carried out.
M.S.\ was supported by the Academy of Finland (Finnish Centre of Excellence in Inverse Modelling and Imaging, grant numbers 312121 and 309963) and by the European Research Council under Horizon 2020 (ERC CoG 770924).

\section{Preliminaries}
\label{sec:pre}

\subsection{Abstract setup for stability}

Let $F: X \to Y$ be a continuous map (linear or nonlinear) between two metric spaces. We also call $F$ the \emph{forward operator}. We consider the basic inverse problem for $F$: given $y \in Y$, find $x \in X$ with 
\[
F(x) = y.
\]
In this section, we will make the following simplifying assumptions (which will be true in most of our examples): 
\begin{itemize}
\item 
$y \in F(X)$, so that there exists at least one $x \in X$ with $F(x) = y$ (more generally, one could consider approximate solutions such as regularized or least squares solutions);
\item 
$F$ is injective, so that there exists at most one $x \in X$ with $F(x) = y$ (more generally, one could consider equivalence classes of solutions such as those related by some gauge invariance, or minimal norm solutions).
\end{itemize}
We are interested in cases where the inverse problem for $F$ is \emph{ill-posed}, meaning that the continuous bijective map $F: X \to F(X)$ does not have a continuous inverse.

It is well known that one can restore some stability in ill-posed problems when the unknowns are restricted to a compact subset of $X$ (see for instance \cite{Tikhonov, John}). This corresponds to imposing \emph{a priori bounds} on the unknowns, and leads to \emph{conditional stability results} which are often formulated in terms of a \emph{modulus of continuity}, i.e.\ an increasing function $\omega: [0,\infty] \to [0,\infty]$ which is continuous at $0$ with $\omega(0) = 0$. The basis for this is the following simple topological result.

\begin{lemma} \label{lemma_modulus_abstract}
Let $F: X \to Y$ be a continuous injective map between two metric spaces. If $K$ is any compact subset of $X$, the map $F|_K: K \to F(K)$ is a homeomorphism and there is a modulus of continuity $\omega$ such that 
\[
d_X(x_1, x_2) \leq \omega(d_Y(F(x_1), F(x_2))), \qquad x_1, x_2 \in K.
\]
\end{lemma}
\begin{proof}
The map $F|_K: K \to F(K)$ is a continuous bijective map from a compact space to a Hausdorff space, hence it is a homeomorphism. Its inverse $G: F(K) \to K$ is uniformly continuous since $F(K)$ is compact, and thus 
\[
d_X(G(y_1), G(y_2)) \leq \omega(d_Y(y_1, y_2)), \qquad y_j \in F(K),
\]
for some modulus of continuity $\omega$. The result follows since $y_j = F(x_j)$ for a unique $x_j \in K$.
\end{proof}

With the previous observation in hand, we thus define the modulus of stability for an inverse problem given the spaces $X,Y$ and the compact set $K$:

\begin{definition}
\label{def:stab_mod}
Let $F: X \to Y$ be a continuous map between metric spaces, let $K$ be a subset of $X$ (usually compact), and let $\omega$ be a modulus of continuity. We say that the \emph{inverse problem for $F$ is $\omega$-stable in $K$} if 
\[
d_X(x_1, x_2) \leq \omega(d_Y(F(x_1), F(x_2))), \qquad x_1, x_2 \in K.
\]
\end{definition}

We note that the modulus of continuity $\omega$ in the above definition depends on 
\begin{itemize}
\item 
the forward operator $F$;
\item 
the spaces $X$ and $Y$ and their topologies (there could be several reasonable choices); and 
\item 
the set $K$ (again there could be many possible choices).
\end{itemize}

In the following sections, given $F$, $X$, $Y$, $K$, we seek to find necessary conditions for $\omega$ (thus providing bounds on the instability of an inverse problem).

In many places of this article, we will consider a less general set-up than the one outlined in Definition \ref{def:stab_mod}: Often the relevant operator $A:X\rightarrow Y$ will be linear with $X$, $Y$ normed spaces (or even separable Hilbert spaces) and $K:=\{x\in X: \ \|x\|_{X_1} \leq 1\}$ where $X_1$ is a normed space (or even separable Hilbert space) which embeds into $X$ compactly. In this setting we will consider two closely related notions of $\omega$-stability:
\begin{itemize}
\item[(a)] $\|x\|_{X}\leq \omega(\|A x\|_{Y})$ for all $x\in X$ with $\|x\|_{X_1}\leq 1$,
\item[(b)] $\|x\|_{X} \leq \omega\left( \frac{\|Ax\|_{Y}}{\|x\|_{X_1}} \right) \|x\|_{X_1}$ for all $x \in X_1 \setminus \{0\}$.
\end{itemize}
We note that, on the one hand, by linearity of $A$ and scaling, the stability condition (a) always  implies the condition (b). If, on the other hand, (b) is satisfied and if $\omega(t)$ is such that ``a concavity condition at zero'' holds, i.e.
\begin{align}
\label{eq:modulus_cond}
r \omega \left(\frac{t}{r} \right) \leq \omega(t) \mbox{ whenever } 0 < r \leq 1,
\end{align}
then also the condition (a) holds. 

We remark that the condition \eqref{eq:modulus_cond} does not pose any essential restrictions, since any continuous map on a compact metric space admits a concave modulus of continuity satisfying \eqref{eq:modulus_cond}. Under these conditions one thus has the equivalence of (a) and (b).

\subsection{Singular value decomposition}

If the forward map $F$ is a linear, compact operator between separable Hilbert spaces (in which case we often write $F = A$), then it has a singular value decomposition and the decay of singular values plays a crucial role in the instability properties. We recall the properties of singular values that will be useful for our purposes.

\begin{proposition} \label{prop_singular_value_decomposition}
Let $A: X \to Y$ be a compact linear operator between separable Hilbert spaces with $X$ infinite dimensional. Let $\lambda_1 \geq \lambda_2 \geq \ldots \geq 0$ be the eigenvalues of $A^* A$, and let $(\varphi_j)_{j=1}^{\infty}$ be a corresponding orthonormal basis of $X$ consisting of eigenfunctions. Let $\sigma_j := \sqrt{\lambda_j}$ for $j \geq 1$, and let $\psi_j := \frac{1}{\sigma_j} A \varphi_j$ when $\sigma_j \neq 0$. Then $\psi_j$ are orthonormal in $Y$, and $A$ has the singular value decomposition 
\[
Au = \sum_{\sigma_j \neq 0} \sigma_j (u, \varphi_j)_X \psi_j, \qquad u \in X.
\]
Here $(\varphi_j)_{j=1}^{\infty}$ is called a singular value basis for $A$. In particular one has 
\[
A\varphi_j = \sigma_j \psi_j, \ \  A^* \psi_j = \sigma_j \varphi_j \ \ \text{ when $\sigma_j \neq 0$}.
\]

The singular values $\sigma_j = \sigma_j(A) = \sigma_j(A^*)$ have the following properties:
\begin{itemize}
\item[(a)] $\norm{A} = \sigma_1(A) \geq \sigma_2(A)\geq \ldots \geq 0$.
\item[(b)] If $A$ is injective, then $\sigma_j(A) > 0$ for all $j \geq 1$.
\item[(c)] If $B: X \to Y$ is compact, then for all $j,k \geq 1$ one has 
\begin{align*}
\sigma_{j+k-1}(A+B) \leq \sigma_j(A) + \sigma_k(B).
\end{align*}
\item[(d)] If $C: X_1 \to X$ is compact, then for all $j,k \geq 1$ one has 
\begin{align*}
\sigma_{j+k-1}(A \circ C) \leq \sigma_j(A) \sigma_k(C).
\end{align*}
\item[(e)] If $\Phi: X_1 \to X$ and $\Psi: Y \to Y_1$ are bounded linear operators, one has 
\[
\sigma_j(A \circ \Phi) \leq \sigma_j(A) \norm{\Phi}, \qquad \sigma_j(\Psi \circ A) \leq \norm{\Psi} \sigma_j(A).
\]
Moreover, if $\Phi$ and $\Psi$ are isometries, then 
\[
\sigma_j(A \circ \Phi) = \sigma_j(\Psi \circ A) = \sigma_j(A).
\]
\end{itemize}
\end{proposition}
\begin{proof}
Since $A^* A: X \to X$ is compact and self-adjoint, by the spectral theorem there is an orthonormal basis $(\varphi_j)_{j=1}^{\infty}$ of $X$ and eigenvalues $\lambda_1 \geq \lambda_2 \geq \ldots \geq 0$ with $\lambda_j \to 0$ so that $A^* A \varphi_j = \lambda_j \varphi_j$. The expression for $Au$ and the equations $A \varphi_j = \sigma_j \psi_j$ and $A^* \psi_j = \sigma_j \varphi_j$ follow immediately, and one also has $\sigma_j(A) = \sigma_j(A^*)$.

For (a) it is enough to note that $\sigma_1(A)^2 = \lambda_1(A^* A) = \norm{A^* A} = \norm{A}^2$, and (b) is clear. The standard Weyl inequality in (c) follows from the Courant minimax principle 
\[
\sigma_{j+k-1}(A+B) = \lambda_{j+k-1}((A+B)^*(A+B))^{1/2} = \min_S \max_{u \perp S, \norm{u} = 1} \norm{Au + Bu}
\]
where the minimum is over all subspaces $S \subset X$ with $\dim(S) \leq j+k-2$. Applying the minimax principle to $A^* A$ and $B^* B$, there are subspaces $U$ and $V$ of $X$ of dimension $j-1$ and $k-1$, respectively, so that 
\[
\norm{Au} \leq \sigma_j(A) \norm{u}, \qquad \norm{Bv} \leq \sigma_k(B) \norm{v}
\]
when $u \perp U$ and $v \perp V$. Let $S$ be a subspace of $X$ containing $U$ and $V$ with $\dim(S) = j+k-2$. Then 
\[
\sigma_{j+k-1}(A+B) \leq \max_{u \perp S, \norm{u} = 1} \norm{Au + Bu} \leq \sigma_j(A) + \sigma_k(B).
\]

For (d) we argue as in (c) and find subspaces $U$ and $V$ of $X$ having dimension $j-1$ and $k-1$, respectively, so that when $u \perp U$ and $v \perp V$ one has 
\[
\norm{Au} \leq \sigma_j(A) \norm{u}, \qquad \norm{Cv} \leq \sigma_k(C) \norm{v}.
\]
Let $S$ be a subspace of $X$ having dimension $j+k-2$ and containing both $C^* U$ and $V$. If $w \perp S$, one has $Cw \perp U$ and 
\[
\norm{ACw} \leq \sigma_j(A) \norm{Cw} \leq \sigma_j(A) \sigma_k(C) \norm{w}.
\]
The minimax principle proves (d). Finally, for (e) note that 
\begin{align*}
\sigma_j(\Psi \circ A) = \min_S \max_{u \perp S, \norm{u} = 1} \norm{\Psi(Au)} \leq \norm{\Psi} \sigma_j(A)
\end{align*}
where the minimum is over subspaces $S$ with dimension $j-1$. It also follows that $\sigma_j(A \circ \Phi) = \sigma_j(\Phi^* \circ A^*) \leq \norm{\Phi^*} \sigma_j(A^*) = \norm{\Phi} \sigma_j(A)$. The last part of (e) follows by looking at $A^* A$ and using that eigenvalues are preserved under isometries.
\end{proof}

\subsection{Weyl asymptotics}

Most of the instability results proved in this article involve some form of Weyl asymptotics for the eigenvalues or singular values of certain operators. We will state three such results. The first one is just the usual Weyl law for the Laplace-Beltrami operator on a compact manifold without boundary.

Let $(M,g)$ be a closed (i.e.\ compact with no boundary) oriented smooth Riemannian manifold with $\dim(M) = n$. Then there is a natural $L^2$ space $L^2(M)$ equipped with the measure induced by the volume form on $(M,g)$. The Laplace-Beltrami operator $\Delta_g$ of $(M,g)$ is given in local coordinates by 
\[
\Delta_g = \abs{g}^{-1/2} \p_j( \abs{g}^{1/2} g^{jk} \p_k u)
\]
where $(g_{jk})$ is the metric $g$ in local coordinates, $(g^{jk})$ is the inverse matrix of $(g_{jk})$, and $\abs{g} = \det(g_{jk})$. The operator $-\Delta_g$ (with domain $C^{\infty}(M)$) is an unbounded self-adjoint operator on $L^2(M)$ with compact resolvent and hence has a discrete spectrum. The eigenvalues have the following Weyl asymptotics (see e.g.\ \cite[Section 8.3 in vol. II]{Taylor}).

\begin{theorem} \label{thm_weyl1}
Let $(M,g)$ be a closed smooth $n$-manifold, and let $0 = \lambda_1 \leq \lambda_2 \leq \ldots$ be the eigenvalues of $-\Delta_g$ on $L^2(M)$. The eigenvalues of $-\Delta_g$ satisfy 
\[
\lim_{\lambda \to \infty} \frac{\#\{ j \,;\, \lambda_j \leq \lambda \}}{\lambda^{n/2}} = \frac{\mathrm{Vol}(M)}{\Gamma(\frac{n}{2}+1) (4\pi)^{n/2}}.
\]
In particular, the nonzero eigenvalues satisfy 
\[
\lambda_j \sim j^{\frac{2}{n}}.
\]
\end{theorem}

The next result is a Weyl law for the Dirichlet Laplacian with bounded measurable coefficients \cite{BirmanSolomyak_nonsmooth}. It is sufficient for us to state this result in a bounded Euclidean domain with smooth boundary, but various generalizations are available (see e.g. \cite{BNR}).

\begin{theorem} \label{thm_weyl2}
Let $\Omega \subset \mR^n$ be a bounded domain with smooth boundary, and let $g$ be a symmetric matrix function with $g \in L^{\infty}(\Omega, \mR^{n \times n})$ and $g^{jk}(x) \xi_j \xi_k \geq c \abs{\xi}^2$ for a.e.\ $x \in \Omega$ where $c > 0$. If $0 < \lambda_1 \leq \lambda_2 \leq \ldots$ are the eigenvalues of the Dirichlet Laplacian $-\Delta_g$ on $\Omega$, then 
\[
\lambda_j \sim j^{2/n}.
\]
\end{theorem}

Finally we state a Weyl law for classical pseudodifferential operators of negative order on a closed manifold $(M,g)$ that have nonvanishing principal symbol (i.e.\ are \emph{elliptic} or \emph{noncharacteristic}) at some point of $T^* M$. These operators are compact on $L^2(M)$, and the behaviour of their singular values will be useful when studying instability in the presence of microlocal smoothing effects. The result essentially follows from \cite{BS_psdo1, BS_psdo2, Karol} but for completeness we give a proof in Appendix \ref{sec_appendix_nonlinear}. We refer to \cite[Chapter 18]{Hoermander} for the notation and basic facts related to pseudodifferential operators ($\Psi$DOs).

\begin{theorem} \label{thm_weyl3}
Let $(M,g)$ be a closed smooth $n$-manifold and let $A \in \Psi^{-m}_{\mathrm{cl}}(M)$ have nonvanishing principal symbol at some $(x_0, \xi_0) \in T^* M \setminus \{0\}$, where $m > 0$. The compact operator $A: L^2(M) \to L^2(M)$ satisfies 
\[
\sigma_j(A) \sim j^{-m/n}.
\]
\end{theorem}

\subsection{Sobolev spaces}

In many inverse problems the forward operator acts between Sobolev type function spaces. Here we will set up certain spaces that will be relevant for this article, mostly in the $L^2$ setting on $C^{\infty}$ manifolds (more general $L^p$ based spaces will be considered in the low regularity setting in Section \ref{sec:instab_low_reg}).

Let first $(M,g)$ be a closed smooth $n$-manifold. For any $s \in \mR$ one can define the $L^2$ based Sobolev space $H^s(M)$ using local coordinates and the space $H^s(\mR^n)$, see \cite[Section 4.3 in vol.\ I]{Taylor}. This is a Hilbert space and there are many equivalent definitions. For example, if $k \geq 0$ is an integer one has the equivalent norm (see \cite[Appendix B]{Besse})
\begin{equation} \label{eq:hkm_norm_equivalent}
\norm{u}_{H^k(M)}^2 = \sum_{j=0}^k \norm{\nabla^j u}_{L^2(M)}^2
\end{equation}
where $\nabla$ is the total covariant derivative induced by the Levi-Civita connection on $(M,g)$, and the $L^2$ norms on the right are norms on tensor fields induced by the metric $g$. Moreover, for any $s \in \mR$ we may consider the Bessel potential 
\[
J^s u = (1-\Delta_g)^{s/2} u.
\]
Here $J^s$ is an elliptic pseudodifferential operator of order $s$ on $M$, and we may use the equivalent norm $\norm{u}_{H^s(M)} = \norm{J^s u}_{L^2(M)}$ on $H^s(M)$ (see e.g.\ \cite{Shubin}).

It will be particularly convenient for us that $H^s(M)$ can be identified with a sequence space with polynomial weights. To see this, let $(\varphi_j)_{j=1}^{\infty}$ be an orthonormal basis of $L^2(M)$ consisting of eigenfunctions of the Laplace-Beltrami operator $-\Delta_g$, corresponding to eigenvalues $0 = \lambda_1 \leq \lambda_2 \leq \ldots \to \infty$. Then 
\[
J^s u = \sum_{j=1}^{\infty} (1+\lambda_j)^{s/2} (u, \varphi_j)_{L^2(M)} \varphi_j.
\]
By Theorem \ref{thm_weyl1} one has the Weyl law $1 + \lambda_j \sim j^{2/n}$ for $j \geq 1$. This gives the following equivalent norm on $H^s(M)$:

\begin{proposition} \label{prop_sobolev_sequence_space}
Let $(M,g)$ be a closed smooth $n$-manifold, and let $(\varphi_j)_{j=1}^{\infty}$ be an orthonormal basis of $L^2(M)$ consisting of eigenfunctions of $-\Delta_g$ as above. Then 
\[
\norm{u}_{H^s(M)}^{2} \sim \sum_{j=1}^{\infty} j^{\frac{2s}{n}} \abs{(u,\varphi_j)}^2.
\]
\end{proposition}

The previous result shows that indeed $H^s(M)$ is isomorphic to a sequence space. We will also need such a space in an abstract setting.

\begin{definition} \label{def_hs_sequence}
Let $X$ be a separable Hilbert space and let $\varphi = (\varphi_j)_{j=1}^{\infty}$ be an orthonormal basis of $X$. Let also $n \geq 1$. For any $s \in \mR$ define 
\[
J^s u =  \sum_{j=1}^{\infty} j^{\frac{s}{n}} (u, \varphi_j) \varphi_j
\]
and the norm 
\[
\norm{u}_{h^s} = \norm{J^s u}_{\ell^2} = \left( \sum_{j=1}^{\infty} j^{\frac{2s}{n}} \abs{(u, \varphi_j)}^2 \right)^{1/2}.
\]
For $s \geq 0$ let $h^s = h^s_{n,X,\varphi}$ be the subspace of $X$ consisting of elements with finite $h^s$ norm, and for $s < 0$ let $h^s$ be the completion of $X$ under the $h^s$ norm.
\end{definition}

Thus Proposition \ref{prop_sobolev_sequence_space} states that $H^s(M) = h^s$ with equivalent norms when $n = \dim(M)$, $X = L^2(M)$, and $\varphi = (\varphi_j)_{j=1}^{\infty}$ is an orthonormal basis of $L^2(M)$ consisting of eigenfunctions of the Laplace-Beltrami operator $-\Delta_g$. Note that we slightly abuse notation and write $J^s$ both for the Bessel potential $(1-\Delta_g)^{s/2}$, which is convenient when we want to use $\Psi$DO properties, and for the sequence space operator in Definition \ref{def_hs_sequence}.

We will also need Sobolev spaces $H^s(M)$ when $(M,g)$ is a compact manifold with smooth boundary. In this case the $H^k(M)$ spaces can be defined using the norm \eqref{eq:hkm_norm_equivalent}, and more generally $H^s(M)$ is defined as in \cite[Section 4.4 in vol.\ I]{Taylor}. In the low regularity setting we will use the spaces $H^s(\Omega)$ when $\Omega \subset \mR^n$ is a bounded open set with Lipschitz boundary and also the spaces $H^t(\p \Omega)$ for $-1 \leq t \leq 1$, see e.g.\ \cite{McLean}.

For $M$ a smooth, not necessarily compact manifold, we define ``localized'' Sobolev spaces:

\begin{definition}
\label{def:loc_Hs}
For $M$ a smooth manifold, the space $H^s_L(M)$ with $L \subset M^{\text{int}}$ compact will be identified with the space $H^s_L(M_1):=\{u\in H^s(M_1) \,;\, \supp(u)\subset L\}$, where $M_1$ is a (fixed) closed manifold containing an open neighbourhood of $L$ in $M$.
\end{definition}

\subsection{Normal derivatives of solutions} \label{subseq_nd}

We record here a standard fact regarding weak normal derivatives of solutions of elliptic equations, which will be used several times later. Let $\Omega \subset \mR^n$ be a bounded open set with Lipschitz boundary, and consider the operator 
\[
L = - \p_j a^{jk} \p_k + b_j \p_j + \p_j c_j + q_0
\]
with $(a^{jk}) \in L^{\infty}(\Omega, \R^{n\times n})$ symmetric and uniformly elliptic, $b_j, c_j \in L^{n}(\Omega)$ and $q_0 \in L^{\frac{n}{2}}(\Omega)$. In terms of scaling this is the roughest possible framework in which the Dirichlet problem for $L$ is well-posed outside its spectrum and satisfies the Fredholm alternative. This also includes the case where $q_0 \in W^{-1,n}(\Omega)$ (see Remark \ref{rmk:gen_Cald}).

If $u \in H^1(\Omega)$ is a weak solution of $Lu = 0$ in $\Omega$, the normal derivative $\p_{\nu}^L u|_{\p \Omega}$ is defined weakly as an element of $H^{-1/2}(\p \Omega)$ by
\begin{align} 
\label{eq:bilinear}
\langle \p_{\nu}^L u|_{\p \Omega}, h \rangle_{\partial \Omega}:=
\int\limits_{\Omega} (a^{jk} \p_j u \p_k v + v b_j \p_j u - u c_j \p_j v + q_0 u v) \,dx, \ h \in H^{\frac{1}{2}}(\partial \Omega),
\end{align}
where $v \in H^1(\Omega)$ is a function such that $v|_{\partial \Omega} = h$. We remark that the right hand side of \eqref{eq:bilinear} is well-defined by the H{\"o}lder inequality and Sobolev embedding and that the weak definition of the normal derivative is independent of the choice of the extension $v \in H^{1}(\Omega)$ of $h \in H^{\frac{1}{2}}(\partial \Omega)$. Indeed, the latter follows from the fact that $u$ is a weak solution of the equation $Lu = 0$, which means that the right hand side of \eqref{eq:bilinear} vanishes if $v \in H^1_0(\Omega)$.

The definition \eqref{eq:bilinear} is justified by the fact that when $u$ and the coefficients are sufficiently regular, an integration by parts shows that $\p_{\nu}^{L}u = (a^{jk} \p_j u - c_k u)\nu_k$. Similar considerations are valid for second order elliptic operators on a compact Riemannian manifold $(M,g)$ with smooth boundary.

\subsection{Gevrey spaces}
\label{sec:Gev_def}

In cases where the forward operator is analytic smoothing, one needs to consider spaces of real analytic functions. It is convenient to work more generally with Gevrey functions. We recall the necessary definitions following \cite{Gramtchev, Rodino_book, HuaRodino, Treves} and \cite[Chapter 8.4]{Hoermander} (and use the convention that $0^0 =1$). 

\begin{definition}
\label{defi:Gsigma}
Let $U \subset \R^n$ be an open set, $\sigma\geq 1$ and $f:U\rightarrow \C$. We say that $f$ is \emph{Gevrey-$\sigma$ regular}, or $f\in G^{\sigma}(U)$, iff $f\in C^{\infty}(U)$ and for each compact set $K\subset U$ there exists a constant $C_K>1$ such that
\begin{align*}
\sup\limits_{x\in K}|\partial^{\alpha}f(x)|\leq C_K^{|\alpha|+1} |\alpha|^{\sigma |\alpha|} \mbox{ with } \alpha \in (\N \cup \{0\})^{n}.
\end{align*}
Further, $G^{\sigma}_c(U):=G^{\sigma}(U)\cap C^{\infty}_c(U)$. \end{definition}

We note that the space $G^1(U)$ corresponds to the real analytic functions and thus $G_c^1(U)= \emptyset$. This leads to a number of technical difficulties, since partitions of unity are not available and one needs to work with sequences of cutoff functions that are ``analytic up to a finite order'' (see Lemma \ref{lem:spatial_cutoff}). However, for $\sigma >1$, we have that $G_c^{\sigma}(U)\neq \emptyset$ (even with $G_c^{\sigma}(U)$ dense in function spaces like $C_c^{\infty}(U)$ or in $L^1_{\mathrm{loc}}(U)$, see \cite[Section 1.4]{Rodino_book}). In the sequel, for simplicity, we will in certain applications restrict our attention to the case $\sigma > 1$. We remark that Gevrey spaces for $\sigma > 1$ allow for the construction of partitions of unity, implying that all local definitions can also be transferred to manifolds with $G^{\sigma}$ atlases.

Next we consider certain Hilbert spaces of Gevrey functions on closed manifolds, defined in terms of Fourier coefficients. We first define an abstract sequence space. Here it is natural to consider sequences over $j \geq 0$ instead of $j \geq 1$.

\begin{definition} \label{def_asigmarho_sequence}
Let $X$ be a separable Hilbert space and let $\varphi = (\varphi_j)_{j=0}^{\infty}$ be an orthonormal basis of $X$. Let also $n \geq 1$. For $1 \leq \sigma < \infty$ and $\rho > 0$, define 
\[
\norm{u}_{a^{\sigma,\rho}} = \left( \sum_{j=0}^{\infty} e^{2\rho j^{\frac{1}{n \sigma}}} \abs{(u, \varphi_j)}^2 \right)^{1/2}.
\]
Let $a^{\sigma,\rho} = a^{\sigma,\rho}_{n,X,\varphi}$ be the subspace of $X$ consisting of elements with finite $a^{\sigma,\rho}$ norm.
\end{definition}

Now consider a closed smooth $n$-manifold $(M,g)$. Let $\varphi = (\varphi_j)_{j=0}^{\infty}$ be an orthonormal basis of eigenfunctions for $-\Delta_g$ with eigenvalues $0 = \lambda_0 \leq  \lambda_1 \leq \lambda_2 \leq \ldots$. By Weyl asymptotics, $\lambda_j \sim j^{2/n}$ for $j \geq 1$. Given $1 \leq \sigma < \infty$ and $\rho > 0$, consider the Hilbert space 
\begin{align}
\label{eq:Asr}
A^{\sigma,\rho}(M) = a^{\sigma,\rho}_{n,L^2(M), \varphi}
\end{align}
with the norm 
\[
\norm{u}_{A^{\sigma,\rho}} := (\sum_{j=0}^{\infty} e^{2\rho j^{\frac{1}{n \sigma}}} \abs{(u,\varphi_j)}^2 )^{1/2}.
\]
Clearly $A^{\sigma, \rho}(M) \subset C^{\infty}(M)$. For $(M,g)$ analytic, we also define the Gevrey space (with the convention $0^0 = 1$) 
\[
G^{\sigma}(M) = \{ u \in C^{\infty}(M) \,;\, \text{there is $C > 0$ with $\norm{\nabla^k u}_{L^{\infty}(M)} \leq C^{k+1} k^{\sigma k}$ for $k \geq 0$} \}.
\]
Then $G^1(M)$ is the space of real-analytic functions.

The connection of $G^{\sigma}(M)$ with the Hilbert spaces $A^{\sigma,\rho}(M)$ is given in Lemma \ref{lemma_arho_comega}, which states that 
\[
G^{\sigma}(M) = \cup_{\rho > 0} A^{\sigma,\rho}(M).
\]
The point is roughly that functions satisfying $\abs{\nabla^k u} \leq C R^k k^{\sigma k}$ for \emph{fixed} $R > 0$ form a Banach space (where $C > 0$ corresponds to the norm), and the Fourier coefficients of these functions satisfy $\abs{(u,\varphi_j)} \leq C' e^{-c j^{\frac{1}{n\sigma}}}$ for some \emph{fixed} $c > 0$ leading to the spaces $A^{\sigma,\rho}(M)$ above. We refer to Appendix \ref{sec_appendix_nonlinear} for more details.

Similarly as in Definition \ref{def:loc_Hs} we also use the following localized $A^{\sigma,\rho}$ spaces:

\begin{definition}
\label{defi:Gev_loc}
Let $(M,g)$ be a smooth manifold and $L\subset M^{\text{int}}$ compact. We define $$A^{\sigma, \rho}_L(M):=\{u\in A^{\sigma,\rho}(M_1) \,;\, \supp(u)\subset L\}\,.$$ Here $M_1$ is a (fixed) closed manifold containing a neighbourhood of $L$. 
\end{definition}

We remark that by virtue of Lemma \ref{lemma_arho_comega} and the connection to the Gevrey spaces, the choice of $M_1$ has an influence on the value of $\rho$ but not the one of $\sigma$.

\section{Abstract framework for instability}
\label{sec:abstract}

In this section, following ideas introduced by Mandache \cite{Mandache} in the context of the Calder\'on problem (see also \cite{SV93} for an even earlier application of these ideas in the specific context of analytic continuation), we exploit the observation that the instability of a map can be encoded in entropy and capacity estimates in suitable function spaces. These quantities measure complexity in an information-theoretic sense (see the Appendix II in \cite{KolmogorovTikhomirov}). These and related concepts have also been used in the context of statistical minimax theory (see \cite[Chapter 6.3 and Theorem 6.3.2]{GN21} and the references therein) and, more recently, in machine learning contexts \cite{BGKP19}. While using the same notions and objects as in these works, in our context and applications, we have to make problem-dependent choices and develop the theory for these (e.g.~including general complexity estimates on certain operator spaces, see Section \ref{sec:ops}).
We stress that this approach works equally well for linear and nonlinear inverse problems.

We recall that the argument of Mandache relied on two main steps:
\begin{itemize}
\item[(i)] the construction of an explicit orthonormal basis with (exponential) decay,
\item[(ii)] general entropy and capacity estimates.
\end{itemize}
While (ii) consisted of a very general argument, Mandache strongly exploited symmetries of the domain and operator to infer (i).

In the sequel, we argue that step (i) is in fact not required: instability can be deduced from pure singular value, entropy and capacity considerations. In particular, this allows us to discard the strong symmetry assumptions that have been used in Mandache's argument and all adaptations of it (see for instance \cite{DiCristoRondi,Isaev,IsaevI,IsaevII}). Examples for applications of this in inverse problems will be discussed in Section \ref{sec:examples1}.

In Section \ref{sec:instab_low_reg} we extend this idea even further and prove that it is also applicable if only a minimal amount of regularization is available.

\subsection{General principle}

First we address step (ii) in the argument of Mandache and formulate this as a general principle. To this end, we begin by recalling the notions of $\eps$-discrete sets and $\delta$-nets.  

\begin{definition}
\label{eq:discrete_net}
Let $(X,d_X)$ denote a metric space and let $\delta >0$. A set $X_1 \subset X$ is a $\delta$-net for $X$, if for every $x\in X$ there exists $\tilde{x}\in X_1$ such that $d_X(x,\tilde{x}) \leq \delta$. 

Let $\eps>0$. A subset $X_2 \subset X$ is an $\eps$-discrete subset of $X$, if for each pair $x,\tilde{x} \in X_2$ it holds $d_X(x,\tilde{x})\geq \eps$.    
\end{definition}

Informally speaking, on the one hand, $\eps$-discrete sets measure how large a function space is, by providing lower bounds on its ``extendedness''. On the other hand, $\delta$-nets measure how compact a space is by yielding upper bounds on its size. In order to deduce instability results for a forward map $F: X \to Y$ restricted to a compact set $K \subset X$, it is enough to show that the image $F(K)$ in $Y$ is ``compressed'' (meaning that it can be covered by relatively few $\delta$-balls), while $K$ is sufficiently ``extended'' (meaning that it has many points at least $\eps$ apart).

The two notions are closely related:

\begin{lemma}
\label{lem:nets_discrete} 
Let $X$ be a metric space and let $\delta>0$.
\begin{itemize}
\item[(a)]  Let $A$ be a $\delta$-net and $B$ be $\eps$-discrete where $\eps > 2\delta$. Then
  $\#B \le \#A$.
\item[(b)]  A maximal $\delta$-discrete set is a $\delta$-net. 
\item[(c)]  If $ f : X \to Y$ is continuous with modulus of continuity $\mu(r)$, then a $\delta$-net of $X$ is mapped to a $\mu(\delta)$-net of $f(X)$. If $\rho$ is a monotonically increasing function and 
  \[ d(x,y) \le \rho( d(f(x),f(y))), \] 
  the image of a $\delta$-discrete set in $X$ is $\rho^{-1} (\delta)$-discrete in $f(X)$.
\end{itemize}
\end{lemma}  

\begin{proof}
For (a), it suffices to observe that for any point $x \in A$ there is at most one point $y \in B$ with $d(x,y) \leq \delta$. 
For (b) we
  note that if $S$ is a maximal $\delta$-discrete set, then $S \cup \{ x \}$ cannot be $\delta$-discrete for any $x \notin S$, which means that for any $x \notin S$ there is $y \in S$ with $d(x,y) < \delta$. Thus $S$ is a $\delta$-net.
  The claims in (c)
  also follow directly from the definition. The claim on the
  $\delta$-nets follows by the definition of continuity. The claim on
  the $\delta$-discrete sets follow from the fact that if $x,y$ are
  elements in a $\delta$-discrete set, then
\begin{align*}
\delta \leq d(x,y) \leq \rho(d(f(x),f(y))).
\end{align*}
Inverting $\rho$ yields the claim.
\end{proof}

The closely related concepts of \emph{entropy} and \emph{capacity} of a metric space are discussed in \cite{KolmogorovTikhomirov, ET08}.

With the notions of $\delta$-nets and $\eps$-discrete sets in hand, we formulate the second step (ii) of the instability argument of Mandache \cite{Mandache} in a general framework. This framework has the advantage that one can consider mappings between general metric spaces (it is not necessary to work with Sobolev type spaces).

\begin{theorem}[\cite{Mandache}]
\label{thm:Mandache}
Let $X$ and $Y$ be metric spaces, let $K \subset X$ be compact, and let $F: K \rightarrow Y$ be an injective map. Let $f, g$ be strictly decreasing continuous functions on $\mR_+$, with $f(0+) = g(0+) = +\infty$, so that for any sufficiently small $\delta, \eps > 0$, 
\begin{itemize}
\item[(i)] there is a $\delta$-net $Y_{\delta}\subset Y$ for $F(K)$ with $\leq f(\delta)$ elements,
\item[(ii)] there is a $\eps$-discrete set $X_{\eps} \subset K$ with $> g(\eps)$ elements.
\end{itemize}
Then the following statements hold:
\begin{enumerate}
\item[(a)]
For any small $\eps > 0$ there are $x_1, x_2$ such that 
\begin{align}
\label{eq:elements}
x_1, x_2 \in K, \qquad d_X(x_1, x_2) \geq \eps, \qquad d_Y(F(x_1), F(x_2)) \leq 2 f^{-1}(g(\eps)).
\end{align}
\item[(b)]
If $\omega$ is a modulus of continuity such that 
\begin{equation} \label{distance_k_stability}
d_X(x_1, x_2) \leq \omega(d_Y(F(x_1), F(x_2))), \qquad x_1, x_2 \in K,
\end{equation}
then 
$\omega(t) \geq g^{-1}(f(t/2))$ for $t$ small.
\end{enumerate}
\end{theorem}

We deduce the result as a consequence of the pigeonhole principle.

\begin{proof}
Let $\delta > 0$ be small, and let $Y_{\delta}\subset Y$ be a $\delta$-net for $F(K)$ with $\abs{Y_{\delta}} \leq f(\delta)$. Let $\eps > 0$ be such that $f(\delta) = g(\eps)$, i.e.\ $\eps = g^{-1}(f(\delta))$. Then there is a $\eps$-discrete set $X_{\eps} \subset K$ with $\abs{X_{\eps}} > g(\eps)$. It follows that 
\begin{align}
\label{eq:pidgeon}
\abs{Y_{\delta}} \leq f(\delta) = g(\eps) < \abs{X_{\eps}}.
\end{align}
By the pigeonhole principle, there exist $x_1, x_2 \in X_{\eps}$ such that 
\begin{align*}
d_Y(F(x_1), F(x_2)) \leq 2 \delta, \mbox{ but } d_X(x_1, x_2) \geq \eps.
\end{align*} 
This proves (a) since $\delta = f^{-1}(g(\eps))$. If \eqref{distance_k_stability} holds, by the monotonicity of $\omega$, one has 
\[
\eps \leq \omega(2\delta).
\]
Thus $\omega(t) \geq g^{-1}(f(t/2))$ for $t > 0$ small, proving (b).
\end{proof}

\begin{remark}
\label{rmk:equiv_Mandache}
It is easy to see that (a) and (b) above are equivalent, if $F$ is required to be continuous. Indeed, assume that (a) is valid. If \eqref{distance_k_stability} holds, then choosing $\eps > 0$ small and $x_1, x_2 \in K$ as in (a) gives that $\eps \leq \omega(2 f^{-1}(g(\eps))$, which implies (b). Conversely, assume that $F$ is continuous and (b) is valid. By Lemma \ref{lemma_modulus_abstract} there is a modulus of continuity $\omega$ such that 
\[
d_X(x_1, x_2) \leq \omega(d_Y(F(x_1), F(x_2))), \qquad x_1, x_2 \in K.
\]
Here we may replace $\omega$ by the corresponding minimal modulus of continuity. Then (b) gives that $\omega(t) \geq g^{-1}(f(t/2))$. On the other hand, the minimal modulus satisfies 
\[
\omega(t) = \sup \{ d_X(x_1, x_2) \,;\, d_Y(F(x_1), F(x_2)) \leq t, \ x_j \in K \}.
\]
By compactness the supremum is in fact a maximum, and hence given $t > 0$ there are $x_1, x_2 \in K$ with $d_Y(F(x_1), F(x_2)) \leq t$ and $d_X(x_1, x_2) \geq \omega(t)$. Since $\omega(t) \geq g^{-1}(f(t/2))$, the claim in (a) follows.

We emphasise that the assumption on the continuity of $F$ is natural and is always true in our applications (see for instance Section \ref{sec:pre} where our general setting is described). Thus it makes no difference if the instability results are formulated in the form of (a) or in the form of (b) above.
\end{remark}

\subsection{Entropy numbers}

Our first instability mechanism is based on global smoothing properties of the forward map $F$. Here we assume that $F$ maps into a Banach space $Y$. If $F(K)$ is contained in a ``smooth'' or ``compressed'' subspace $Y_1$ of $Y$, we may write 
\[
F|_K = i \circ \tilde{F}
\]
where $\tilde{F}: K \to Y_1$ satisfies $\tilde{F}(x) = F(x)$, and $i$ is the inclusion $Y_1 \to Y$. If $\tilde{F}$ is continuous, then $F(K) \subset i(B)$ where $B$ is a ball in $Y_1$. To show that $F(K)$ is compressed, it is enough to show that $i(B)$ can be covered by relatively few $\delta$-balls in $Y$.

The notion of \emph{entropy numbers} (see \cite{CS90, ET08}), which measure the compactness of a linear operator, is ideally suited to the setup described above. We also introduce the related \emph{capacity numbers} \cite{CS90}.

\begin{definition}
\label{defi:entropy_numb}
Let $X,Y$ be Banach spaces and let $A: X \to Y$ be a bounded linear operator. Put $U_X:=\{x\in X: \|x\|_X \leq 1 \}$. Then for any $k \geq 1$, the $k$th entropy number $e_k(A)$ of $A$ is defined by
\begin{align*}
  e_k(A)&:= \inf\left\{ \eps > 0: A(U_X) \subset \bigcup\limits_{j=1}^{2^{k-1}}(y_j + \eps U_Y) \mbox{ for some } y_1,\dots, y_{2^{k-1}} \in Y  \right\}\\
 &= \inf \{ \eps>0: \text{ there exists an $\varepsilon$-net for $A(U_X)$ of cardinality $2^{k-1}$}\} 
    . \end{align*}

The $k$th capacity number $c_k(A)$ of $A$ is 
\begin{align*}
c_k(A) &:= \sup \{ \eps > 0 : \ \text{there are $x_1, \ldots, x_N \in U_X$ with $N > 2^{k-1}$ and} \\
&  \qquad \qquad \qquad \qquad \text{$d_Y(Ax_j, Ax_k) \geq 2 \eps$ for $j \neq k$} \}
 \\ &  =   \sup \{ \eps>0:  \text{ there is a $2\eps$-discrete set for $A(U_X)$ of cardinality $2^{k-1}+1$}\} 
.
\end{align*}
\end{definition}

The entropy and capacity numbers are very similar (see also \cite[Section 1.1]{CS90}):

\begin{lemma}
One has 
\begin{equation} \label{entropy_capacity_relation}
\frac{1}{2} e_k(A) \leq c_k(A) \leq e_k(A).
\end{equation}
\end{lemma}
\begin{proof}
Let $\eps > e_k(A)$, so that $A(U_X) \subset \cup_{j=1}^{2^{k-1}} \ol{B(y_j, \eps)}$. Let also $\rho < c_k(A)$, so that there is a $2\rho$-discrete set $S$ in $A(U_X)$ with cardinality $2^{k-1}+1$. Then some two points $z, w \in S$ must lie in some ball $\ol{B}(y_j,\eps)$. One has $2\rho \leq d(z, w) \leq 2\eps$, which proves that $c_k(A) \leq e_k(A)$.

On the other hand, if $c_k(A) < \rho$, then there is a maximal $2\rho$-discrete set $S$ in $A(U_X)$ with cardinality $\leq 2^{k-1}$. By Lemma \ref{lem:nets_discrete} $S$ is also a $2\rho$-net, showing that $e_k(A) \leq 2\rho$. Thus $e_k(A) \leq 2 c_k(A)$.
\end{proof}

By \eqref{entropy_capacity_relation} it will be sufficient to focus on entropy number estimates. The following simple result shows how typical entropy number bounds could be used for showing instability based on Theorem \ref{thm:Mandache}.

\begin{lemma} \label{lemma_entropy_typical}
Let $X, Y$ be Banach spaces, and let $X_1 \subset X$ and $Y_1 \subset Y$ be closed subspaces so that the inclusions $i: X_1 \to X$ and $j: Y_1 \to Y$ are compact. Let $K = \{ x \in X \,;\, \norm{x}_{X_1} \leq r \}$ for some $r > 0$, and assume that $F: K \to Y$ is a map such that $F(K) \subset \{ y \in Y_1 \,;\, \norm{y}_{Y_1} \leq R \}$ for some $R > 0$. Suppose that, for some function $\omega: \mR_+ \to \mR_+$, 
\[
\norm{x_1 - x_2}_X \leq \omega(\norm{F(x_1) - F(x_2)}_Y), \qquad x_1, x_2 \in K.
\]
\begin{enumerate}
\item[(a)] 
If $e_k(i) \gtrsim k^{-m}$ and $e_k(j) \lesssim k^{-s}$ for some $m, s > 0$, then for $t$ small 
\[
\label{eq:mod_a}
\omega(t) \gtrsim t^{m/s}.
\]
\item[(b)] 
If $e_k(i) \gtrsim k^{-m}$ and $e_k(j) \lesssim e^{-ck^{\alpha}}$ for some $m, c, \alpha > 0$, then for $t$ small 
\[
\omega(t) \gtrsim \abs{\log t}^{-m/\alpha}.
\]
\item[(c)] 
If $e_k(i) \gtrsim e^{-dk^{\beta}}$ and $e_k(j) \lesssim e^{-ck^{\alpha}}$ for some $c, d, \alpha, \beta > 0$, then for $t$ small 
\[
\omega(t) \gtrsim e^{-c_1 \abs{\log t}^{\beta/\alpha}}.
\]
\end{enumerate}
\end{lemma}
\begin{proof}
The main point is to use the entropy bounds to estimate the functions $f(\delta)$ and $g(\eps)$ in Theorem \ref{thm:Mandache}.

(a) If $e_k(i) \geq c_0 k^{-m}$, by \eqref{entropy_capacity_relation} one has $c_k(i) \geq \frac{c_0}{2} k^{-m}$, and thus for any $\eps < \frac{c_0}{2} r k^{-m}$ there is a set $X_{\eps} \subset K$ so that $\abs{X_{\eps}} > 2^{k-1}$ and $X_{\eps}$ is $\eps$-discrete in $X$. Now 
\[
\eps = \frac{c_0}{2} r k^{-m} \quad \Longleftrightarrow \quad 2^{k-1} = \frac{1}{2} e^{( (\frac{c_0 r}{2} )^{1/m} \log 2) \eps^{-1/m}}.
\]
Thus one may choose 
\[
g(\eps) = e^{c^{1/m} \eps^{-1/m}}, \qquad c = \frac{c_0 r}{2^{m+1}}.
\]
The inverse function is $g^{-1}(\eta) = c (\log \eta)^{-m}$.

If $e_k(j) < C_0 k^{-s}$, by the bound on $F(K)$, the set $F(K)$ can be covered with $2^{k-1}$ balls of radius $R C_0 k^{-s}$. Now 
\[
\delta = R C_0 k^{-s} \quad \Longleftrightarrow \quad 2^{k-1} = \frac{1}{2} e^{((C_0 R)^{1/s} \log 2) \delta^{-1/s}}.
\]
Hence, if $\delta \geq R C_0 k^{-s}$, there is a $\delta$-net $Y_{\delta}$ for $F(K)$ with $\abs{Y_{\delta}} \leq e^{(C_0 R)^{1/s} \delta^{-1/s}}$. This proves that one may choose 
\[
f(\delta) = e^{C\delta^{-1/s}}, \qquad C = (C_0 R)^{1/s}.
\]

By Theorem \ref{thm:Mandache}, for $t$ small one has 
\begin{align*}
\omega(t) &\geq g^{-1}(f(t/2)) = c (C(t/2)^{-1/s})^{-m}
\gtrsim t^{m/s}.
\end{align*}

(b) We can use the same $g(\eps)$ as in part (b). If $e_k(j) \leq C_0 e^{-ck^{\alpha}}$ for some $C_0, c, \alpha > 0$, then a computation shows that one may choose for $\delta$ small 
\[
f(\delta) = e^{C \abs{\log \delta}^{1/\alpha}}.
\]
It follows that $\omega(t) \geq g^{-1}(f(t/2)) \gtrsim \abs{\log t}^{-m/\alpha}$.

(c) We can take $f(\delta)$ as in part (b), and a computation shows that one may take 
\[
g(\eps) = e^{C_1 \abs{\log \eps}^{1/\beta}}.
\]
The inverse function is $g^{-1}(\eta) = e^{-c_0 (\log \eta)^{\beta}}$, and $\omega(t) \geq g^{-1}(f(t/2)) \gtrsim e^{-c_1 \abs{\log t}^{\beta/\alpha}}$.
\end{proof}

\subsection{Properties of entropy numbers}
In the remainder of this section we will study estimates for entropy numbers. We note that entropy numbers have the following properties (cf.\ Chapter 1.3 in \cite{ET08}).

\begin{lemma}
\label{lem:prop_entropy}
Let $X,Y,Z$ be Banach. For bounded linear operators $A, B: X \to Y$ and $C: Y \to Z$ with $A \not\equiv 0$, we have the following properties:
\begin{itemize}
\item[(i)] $\|A\| = e_1(A) \geq e_2(A)\geq \ldots > 0$.
\item[(ii)] For all $j,k \geq 1$ one has 
\begin{align*}
e_{j+k-1}(C \circ A) \leq e_j(C) e_k(A).
\end{align*}
\item[(iii)] For all $j,k \geq 1$ one has 
\begin{align*}
e_{j+k-1}(A+B) \leq e_j(A) + e_k(B).
\end{align*}
\item[(iv)] $A$ is compact if and only if $e_k(A) \to 0$ as $k \to \infty$.
\item[(v)] $A$ has finite rank if and only if $e_k(A) \lesssim e^{-ck}$ for some $c > 0$.
\item[(vi)] If $X$ and $Y$ are Hilbert spaces, then $e_k(A) = e_k(A^*) = e_k(\sqrt{A^* A})$.
\end{itemize}
\end{lemma}

The above result shows that entropy numbers are indeed related to compactness. Singular values also encode compactness properties of linear operators. There are various relations between these two notions (see \cite{CS90, ET08}). For us, the following result stating the equivalence of typical decay rates will be sufficient.

\begin{lemma} \label{lemma_singular_entropy_relation}
Let $A: X \to Y$ be a compact operator between separable Hilbert spaces. For any $s > 0$, one has 
\begin{align*}
\sigma_k(A) \lesssim k^{-s} \quad &\Longleftrightarrow \quad e_k(A) \lesssim k^{-s}, \\
\sigma_k(A) \gtrsim k^{-s} \quad &\hspace{-1.5pt}\implies \hspace{-2pt} \quad e_k(A) \gtrsim k^{-s}. \\
\intertext{Moreover, if $\mu > 0$, then}
\sigma_k(A) \lesssim e^{-ck^{\mu}} \text{ for some $c > 0$} \quad &\Longleftrightarrow \quad e_k(A) \lesssim e^{-\tilde{c} k^{\frac{\mu}{1+\mu}}} \text{ for some $\tilde{c} > 0$}, \\
\sigma_k(A) \gtrsim e^{-ck^{\mu}} \text{ for some $c > 0$} \quad &\hspace{-1.5pt}\implies \hspace{-2pt} \quad e_k(A) \gtrsim e^{-\tilde{c} k^{\frac{\mu}{1+\mu}}} \text{ for some $\tilde{c} > 0$}.
\end{align*}
\end{lemma}
\begin{proof}
The argument is based on \cite[Proposition 1.3.2]{CS90}, which states that for any $N \geq 1$, 
\begin{equation} \label{carl_ineq_strong}
\sup_{k \geq 1} 2^{-\frac{N-1}{2k}} (\sigma_1(A) \cdots \sigma_k(A))^{1/k} \leq e_N(A) \leq 6 \,\sup_{k \geq 1} 2^{-\frac{N-1}{2k}} (\sigma_1(A) \cdots \sigma_k(A))^{1/k}.
\end{equation}
In fact this is stated in \cite{CS90} for diagonal operators $\ell^2 \to \ell^2$, but $\sqrt{A^* A}$ is unitarily equivalent to such an operator $D$ and unitary equivalence does not affect singular values or entropy numbers. Thus the result holds first for $D$, and then also for $A$ since $\sigma_j(\sqrt{A^* A}) = \sigma_j(A)$ and $e_j(\sqrt{A^* A}) = e_j(A)$. We rewrite \eqref{carl_ineq_strong} in the more convenient form, with $c_1 > 0$ an absolute constant, 
\begin{equation} \label{carl_ineq_strong_second}
\sup_{k \geq 1}  e^{-\frac{N}{2k}}  (\sigma_1 \cdots \sigma_k)^{1/k} \leq e_N \lesssim \,\sup_{k \geq 1} e^{-c_1 N/k} (\sigma_1 \cdots \sigma_k)^{1/k}.
\end{equation}

If $\sigma_j \lesssim j^{-s}$, then $(\sigma_1 \cdots \sigma_k)^{1/k} \lesssim (k!)^{-s/k} \sim k^{-s}$ by Stirling's formula. Now optimizing the function $f(k):= ck^{-s} e^{-c_1 N/k}$ in $k$ yields that $k \sim c_s N$ and thus \eqref{carl_ineq_strong_second} gives $e_N \lesssim N^{-s}$. Conversely, if $e_N \lesssim N^{-s}$, then choosing $k=N$ on the left of \eqref{carl_ineq_strong_second} and using $\sigma_1 \geq \sigma_2 \geq \ldots$ gives $\sigma_N \lesssim N^{-s}$. The statement that $\sigma_j \gtrsim j^{-s}$ implies $e_N \gtrsim N^{-s}$ follows from the Stirling formula as above.

Now if $\sigma_k \lesssim e^{-ck^{\mu}}$, then evaluating a Riemann sum gives 
\[
(\sigma_1 \cdots \sigma_k)^{1/k} \lesssim e^{-\frac{c}{\mu+1} k^{\mu}}.
\]
Given $N$, the expression $e^{-c_1 N/k} e^{-\frac{c}{\mu+1} k^{\mu}}$ is maximal when $k \sim N^{\frac{1}{1+\mu}}$. Choosing this value of $k$ on the right of \eqref{carl_ineq_strong_second} gives $e_N \lesssim e^{-\tilde{c} N^{\frac{\mu}{1+\mu}}}$. Conversely, assume $e_N \lesssim e^{-\tilde{c} N^{\frac{\mu}{1+\mu}}}$. Fixing $j \geq 1$, choosing $k=j$ on the left of \eqref{carl_ineq_strong_second} and using that $\sigma_1 \geq \sigma_2 \geq \ldots$  gives 
\[
\sigma_j \lesssim e^{c_0 N/j -\tilde{c} N^{\frac{\mu}{1+\mu}}}, \qquad N \geq 1.
\]
The right hand side is minimal when $N \sim j^{1+\mu}$, and this choice gives $\sigma_j \lesssim e^{-c j^{\mu}}$ for some $c > 0$ as required. Finally, if $\sigma_j \gtrsim e^{-cj^{\mu}}$, then a Riemann sum argument gives $(\sigma_1 \cdots \sigma_k)^{1/k} \gtrsim e^{-\frac{c}{\mu+1} \frac{(k+1)^{\mu+1}}{k}} \gtrsim e^{-\tilde{c} k^{\mu}}$. By \eqref{carl_ineq_strong_second} one has $e_N \gtrsim \sup_{k \geq 1} e^{-c_0 N/k} e^{-\tilde{c} k^{\mu}}$, and choosing $k \sim N^{\frac{1}{1+\mu}}$ gives $e_N \gtrsim e^{-\tilde{c} N^{\frac{\mu}{1+\mu}}}$.
\end{proof}

\subsection{Embeddings between Sobolev or Gevrey spaces}

Recall from the beginning of this subsection that we are interested in entropy numbers of inclusions $i: Y_1 \to Y$, where $Y_1$ is a subspace of a Banach space $Y$. The case of Sobolev type spaces, at least on domains in $\mR^n$, is well understood \cite{ET08}. The estimates on manifolds are similar, and we give a proof for completeness.

\begin{theorem} \label{thm_sobolev_inclusion}
Let $M$ be a compact smooth $n$-manifold with or without smooth boundary, and let $s_1, s_2 \in \mR$ with $s_1 > s_2$. The embedding $i: H^{s_1}(M) \to H^{s_2}(M)$ satisfies 
\begin{gather*}
\sigma_k(i) \sim k^{-(s_1-s_2)/n}, \\
e_k(i) \sim k^{-(s_1-s_2)/n}.
\end{gather*}
The same bounds hold for the embedding $i: h^{s_1}_{n,X,\varphi} \to h^{s_2}_{n,X,\varphi}$ (see Definition \ref{def_hs_sequence}).
\end{theorem}

\begin{proof}
By Lemma \ref{lemma_singular_entropy_relation}, it is enough to prove the statement about singular values. Assume first that $M$ has no boundary, and consider the norm $\norm{f}_{H^s} = \norm{J^s f}_{L^2}$ where $$J^s f = \sum_{j=1}^{\infty} j^{\frac{s}{n}} (f,\varphi_j)_{L^2} \varphi_j$$ as in Definition \ref{def_hs_sequence}. Note that 
\[
(i^* i(f), h)_{H^{s_1}} = (f, h)_{H^{s_2}} = (J^{-2(s_1-s_2)} f, h)_{H^{s_1}}.
\]
Thus $i^* i = J^{-2(s_1-s_2)}$ on $H^{s_1}(M)$, and using the isometries $J^{\pm s_1}$ we have 
\begin{align*}
\sigma_k(i)^2 &= \sigma_k(i^* i) = \sigma_k(J^{s_1} \circ i^* i \circ J^{-s_1}) = \sigma_k( J^{-2(s_1-s_2)}: L^2(M) \to L^2(M) ) \\
 &= k^{-\frac{2(s_1-s_2)}{n}}.
\end{align*}
In the last step we used that $J^{-2(s_1-s_2)}: L^2(M) \to L^2(M)$ is a diagonal operator.

Next assume that $M$ is compact with smooth boundary. Let $N$ be a closed manifold containing $M$, let $E: H^{s_1}(M) \to H^{s_1}(N)$ be a bounded extension operator, and let $R: H^{s_2}(N) \to H^{s_2}(M), Ru = u|_M$ be the restriction operator. Then $i = R \circ \iota \circ E$, where $\iota$ is the inclusion $H^{s_1}(N) \to H^{s_2}(N)$. Consequently we have the upper bound 
\[
\sigma_k(i) \leq \norm{R} \sigma_k(\iota) \norm{E} \lesssim  k^{-(s_1-s_2)/n}.
\]
For the lower bound on $\sigma_k(i)$, we choose a compact set $K \subset M^{\mathrm{int}}$ and a cutoff function $\chi \in C^{\infty}(N)$ with $\mathrm{supp}(\chi) \subset K$ and $\chi \neq 0$ somewhere. Consider the operator 
\[
A: H^{s_1}(N) \to H^{s_2}(N), \ \ A = E_0 \circ i \circ R \circ m_{\chi}
\]
where $E_0$ denotes extension by zero, and $m_{\chi}$ denotes multiplication by $\chi$. Then 
\[
\sigma_k(A) \leq \norm{E_0} \sigma_k(i) \norm{R \circ m_{\chi}} \lesssim \sigma_k(i).
\]
It is enough to prove a lower bound for $\sigma_k(A)$. We compute $A^* A$ as 
\[
(A^* A f, h)_{H^{s_1}} = (\chi f, \chi h)_{H^{s_2}} = (\chi J^{2s_2} \chi f, h)_{L^2} = (J^{-2s_1} \chi J^{2s_2} \chi f, h)_{H^{s_1}}
\]
where we now choose $J^s = (1-\Delta_g)^{s/2}$. Thus $A^* A f = J^{-2s_1} \chi J^{2s_2} \chi f$ on $H^{s_1}(N)$. Using the isometries $J^{\pm s_1}$ yields  
\[
\sigma_k(A)^2 = \sigma_k(A^* A: H^{s_1} \to H^{s_1}) = \sigma_k(J^{-s_1} \chi J^{2s_2} \chi J^{-s_1}: L^2 \to L^2).
\]
Now $J^{-s_1} \chi J^{2s_2} \chi J^{-s_1}$ is a classical pseudodifferential operator of order $-2(s_1-s_2)$ on $N$ and it has nonvanishing principal symbol at points where $\chi \neq 0$. Thus by Theorem \ref{thm_weyl3} its singular values on $L^2(N)$ satisfy $\sigma_k \gtrsim k^{-\frac{2(s_1-s_2)}{n}}$. This concludes the proof.
\end{proof}

By using similar localization arguments as in Theorem \ref{thm_sobolev_inclusion}, one obtains the following localized bounds which will be used in Section \ref{sec:micro_abstract}.

\begin{proposition}
\label{prop:entropy_localized}
Let $N$ be a smooth $n$-dimensional manifold. Let $Q\subset N$ be compact and let $t\in \R$, $m>0$. Then for $H^{t+m}_Q(N) :=\{f\in H^{t+m}(N): \ \supp(f)\subset Q\}$ 
and for the embedding $i: H^{t+m}_Q(N) \to H^t_Q(N)$ we have 
\begin{align*}
e_k(i) \sim k^{-m/n}.
\end{align*}
\end{proposition}

For later purposes we note that similar estimates also hold under weaker assumptions on the domain (no high smoothness necessary), see \cite[Theorem 23.2]{Triebel_fractals}.

\begin{proposition} \label{prop_ek_nonsmooth}
Let $\Omega$ be a bounded Lipschitz domain in $\mR^n$. If $s_1 > s_2$, then the embedding $i: H^{s_1}(\Omega) \to H^{s_2}(\Omega)$ has entropy numbers satisfying 
\[
e_k(i) \sim k^{-\frac{s_1-s_2}{n}}, \qquad k \geq 1.
\]
\end{proposition}

Next we consider spaces of real-analytic or more generally Gevrey functions. Let $(M,g)$ be a compact connected smooth $n$-manifold without boundary, and recall the spaces $A^{\sigma,\rho}(M)$ of Gevrey functions introduced in \eqref{eq:Asr} and below in Section \ref{sec:Gev_def}. These Hilbert spaces have the following entropy properties.

\begin{theorem} \label{thm_analytic_inclusion}
Let $(M,g)$ be a closed smooth $n$-manifold, let $1 \leq \sigma < \infty$, $\rho > 0$ and $s \in \mR$. The embedding $i: A^{\sigma,\rho}(M) \to H^s(M)$ satisfies for some $\tilde{\rho} > 0$ 
\begin{align*}
\sigma_k(i) &\sim k^{s/n} \exp(-\rho (k-1)^{\frac{1}{n \sigma}}), \\
e_k(i) &\lesssim \exp(-\tilde{\rho} k^{\frac{1}{n \sigma + 1}}).
\end{align*}
The same bounds hold for the embedding $i: a^{\sigma,\rho}_{n,X,\varphi} \to h^s_{n,X,\varphi}$ (see Definitions \ref{def_hs_sequence} and \ref{def_asigmarho_sequence}).
\end{theorem}
\begin{proof}
By Lemma \ref{lemma_singular_entropy_relation}, it is enough to prove the statement about singular values. Since $M$ is closed, $H^s(M)$ is isomorphic to the space $h^s$ of sequences $x = (x_0, x_1, \ldots)$ with norm $\norm{x}_{h^s} = ( \sum_{j=0}^{\infty} (1+j)^{\frac{2s}{n}} \abs{x_j}^2 )^{1/2}$ (see Proposition \ref{prop_sobolev_sequence_space}), and $A^{\sigma,\rho}(M)$ is isomorphic to the space 
\[
a^{\sigma,\rho} = \{ x = (x_0, x_1, \ldots) \in \ell^2 \,;\, \norm{x}_{a^{\sigma,\rho}} := (\sum_{j=0}^{\infty} e^{2\rho j^{\frac{1}{n \sigma}}} \abs{x_j}^2)^{1/2} < \infty \}.
\]
Then $\sigma_k(i) \sim \sigma_k(\iota)$ where $\iota$ is the inclusion $a^{\sigma,\rho} \to h^s$. We compute 
\[
\sum_{j=0}^{\infty} e^{2\rho j^{\frac{1}{n \sigma}}} (\iota^* \iota(x))_j \ol{y}_j = (\iota^* \iota(x), y)_{a^{\sigma,\rho}} = (\iota(x), \iota(y))_{h^s} = \sum_{j=0}^{\infty} (1+j)^{2s/n} x_j \ol{y}_j.
\]
Thus $\iota^* \iota$ is the diagonal operator with $(\iota^* \iota(x))_j = (1+j)^{2s/n} e^{-2\rho j^{\frac{1}{n \sigma}}} x_j$ for $j \geq 0$, which shows that $\sigma_k(\iota) = k^{s/n}e^{-\rho (k-1)^{\frac{1}{n \sigma}}}$ for $k$ large.
\end{proof}

The above results immediately yield decay rates for singular values of smoothing operators. Further results of this type may be found in \cite{BS77}.

\begin{theorem}
\label{thm:eigenval_bds}
Let $M$ and $N$ be compact $C^{\infty}$ manifolds, having dimensions $n_M$ and $n_N$.
Let $A: H^{s_2}(N) \to H^{s_1}(M)$ be a compact linear operator with singular values $\sigma_j(A)$.
\begin{enumerate}
\item[(i)] 
If $A$ extends as a bounded operator $A_1: H^{s_2-m_2}(N) \to H^{s_1}(M)$ such that $$A_1(H^{s_2-m_2}(N)) \subset H^{s_1+m_1}(M)$$ where $m_1, m_2 \geq 0$ with $m_1+m_2 > 0$, then 
\[
\sigma_j(A) \lesssim j^{-m_1/n_M-m_2/n_N}.
\]
In particular, if $M$ is a closed $n$-manifold and $B \in \Psi^{-s}(M)$, where $B$ is considered as a compact operator $H^t(M) \to H^t(M)$ for some $t \in \mR$, then 
\[
\sigma_j(B) \lesssim j^{-s/n}.
\]
\item[(ii)] 
If $(M,g)$ is a closed analytic $n$-manifold and if $A(H^{s_2}(N)) \subset A^{\sigma,\rho}(M)$, then for some $c > 0$ 
\[
\sigma_j(A) \lesssim \exp(-cj^{\frac{1}{n \sigma}}).
\]
In particular if $B \in \Psi^{-\infty}_{\sigma}(M)$, i.e.\ $B$ is Gevrey $G^{\sigma}$-smoothing on $M$, and $B$ is considered as an operator $L^2(M) \to L^2(M)$, then 
\[
\sigma_j(B) \lesssim \exp(-cj^{\frac{1}{n \sigma}}).
\]
\end{enumerate}
\end{theorem}

We refer to Appendix \ref{sec:analytic} for definitions and results on Gevrey smoothing operators.

\begin{proof}[Proof of Theorem \ref{thm:eigenval_bds}]
(i) We write $A = i_1 \circ \tilde{A} \circ i_2$, where $\tilde{A}: H^{s_2-m_2}(N) \to H^{s_1+m_1}(M)$ is bounded (this holds by the closed graph theorem since $A_1(H^{s_2-m_2}(N)) \subset H^{s_1+m_1}(M)$) and $i_1: H^{s_1+m_1}(M) \to H^{s_1}(M)$, $i_2: H^{s_2}(N) \to H^{s_2-m_2}(N)$ are the natural inclusions. Write $j + 1 = 2 \ell + r$ where $\ell \geq 1$ and $r \in \{ 0, 1 \}$. Then 
\begin{align*}
\sigma_j(A) = \sigma_{2\ell+r-1}(i_1 \circ \tilde{A} \circ i_2) \leq \sigma_{\ell}(i_1) \sigma_{\ell+r}(\tilde{A} \circ i_2) \leq \sigma_{\ell}(i_1) \norm{\tilde{A}} \sigma_{\ell}(i_2).
\end{align*}
Next we use that 
\begin{gather*}
\sigma_{\ell}(i_1) \lesssim \ell^{-m_1/n_M}, \\
\sigma_{\ell}(i_2) \lesssim \ell^{-m_2/n_N}.
\end{gather*}
Thus, since $\ell \geq j/2$, it follows that 
\[
\sigma_j(A) \lesssim \ell^{-m_1/n_M-m_2/n_N} \lesssim j^{-m_1/n_M-m_2/n_N}.
\]
The result for $B$ follows immediately.

(ii) Write $A = i \circ \tilde{A}$ where $\tilde{A}: H^{s_2}(N) \to A^{\sigma,\rho}(M)$ is bounded (by the closed graph theorem since $A(H^{s_2}(N)) \subset A^{\sigma,\rho}(M)$) and $i: A^{\sigma,\rho}(M) \to H^{s_1}(M)$ is the natural inclusion. Then 
\[
\sigma_j(A) \leq \sigma_j(i) \norm{\tilde{A}} \lesssim \norm{\tilde{A}} e^{-cj^{\frac{1}{n\sigma}}}.
\]
Now if $B$ is $G^{\sigma}$ smoothing on $M$, then $B$ has a $G^{\sigma}$ integral kernel $K_B(x,y)$ so that 
\[
Bu(x) = \int_M K_B(x,y) u(y) \,dV_g(y).
\]
The Hilbert-Schmidt bound gives 
\[
\norm{(-\Delta_g)^k Bu}_{L^2(M)} \leq \norm{(-\Delta_{g,x})^k K_B}_{L^2(M \times M)} \norm{u}_{L^2(M)}.
\]
Since $K_B$ is a $G^{\sigma}$ function, there are $C, R > 0$ so that 
\[
\norm{(-\Delta_g)^k Bu}_{L^2(M)} \leq C R^{2k} (2k)^{2k\sigma} \norm{u}_{L^2(M)}.
\]
Lemma \ref{lemma_arho_comega}(c) implies that for $\rho = c_0 R^{-1/\sigma}$, one has 
\[
\norm{Bu}_{A^{\sigma,\rho}(M)} \leq c_1 C \norm{u}_{L^2(M)}.
\]
Thus $B$ induces a bounded map $L^2(M) \to A^{\sigma,\rho}(M)$, and the result follows from the first part of (ii).
\end{proof}

\subsection{Embeddings between spaces of operators}
\label{sec:ops}

Finally, in order to deal with inverse problems where the measurement is a Dirichlet-to-Neumann type operator, we discuss entropy bounds between spaces of operators. In this case the range of the forward map lies in the Banach space $B(H^s(M),H^{-s}(M))$ of bounded linear operators from a Sobolev space $H^s(M)$ to its dual $H^{-s}(M)$, where $M$ is a closed $n$-manifold.

We first consider the case $s=0$ and prove the following entropy bound when embedding smoothing operators of order $m$ into the space of bounded operators on $L^2$.

\begin{theorem} \label{thm_entropy_bounded_operators_initial}
  Let $(M,g)$ be a closed smooth $n$-manifold. If $m>0$, then the embedding
$ i: B(H^{-m}, L^2) \cap B(L^2, H^m) \to B(L^2,L^2)$ satisfies
\begin{equation} \label{b_initial_entropy_bound}
e_k(i) \lesssim k^{-\frac{m}{2n}+ \delta}
\end{equation}
for any $\delta > 0$.
\end{theorem}

In the proof we will work with Hilbert spaces instead of the Banach space $B(X, Y)$. This amounts to replacing $B(X,Y)$ by the space $HS(X,Y)$ of Hilbert-Schmidt operators (i.e.\ replacing the Schatten class $S_{\infty}$ by $S_2$). We first recall some definitions. Let $X$ and $Y$ be separable Hilbert spaces. The space $HS(X,Y)$ consists of those compact linear operators $T: X \to Y$ for which the quantity
\begin{equation} \label{hsxy_equivalent_norms}
\norm{T}_{HS(X,Y)}^2 = \sum_{j=1}^{\infty} \sigma_j(T)^2 = \mathrm{tr}(T^* T) = \sum_{j=1}^{\infty} \norm{T\varphi_j}_Y^2 = \sum_{j,k=1}^{\infty} \abs{(T\varphi_j, \psi_k)_Y}^2
\end{equation}
is finite and independent of the choice of orthonormal bases $(\varphi_j)$ and $(\psi_k)$ of $X$ and $Y$, respectively. This is a Hilbert space with inner product 
\[
(S, T)_{HS(X,Y)} = \tr(T^* S) = \sum_{j=1}^{\infty} (S \varphi_j, T \varphi_j)_Y.
\]
Recall also that when $X = Y = L^2(M)$, any $T \in HS(L^2, L^2)$ has a Schwartz kernel $K_T \in L^2(M \times M)$ and 
\begin{equation} \label{hs_ltwo_kernel}
\norm{T}_{HS(L^2, L^2)} = \norm{K_T}_{L^2(M \times M)}.
\end{equation}

\begin{proof}[Proof of Theorem \ref{thm_entropy_bounded_operators_initial}]
We argue in two steps, first proving an initial (suboptimal) entropy bound when $m > n/2$, and then improving on this in the second step by using a decomposition of operators into smooth and small parts. In the proof we fix an orthonormal basis $(\varphi_j)_{j=1}^{\infty}$ of $L^2(M)$ consisting of eigenfunctions of $-\Delta_g$, and use the Sobolev norm $\norm{f}_{H^s(M)} = \norm{J^s f}_{L^2(M)}$ where $J^s f = \sum_{j=1}^{\infty} j^{\frac{s}{n}} (f, \varphi_j) \varphi_j$. We will also write $\norm{T}_* = \norm{T}_{B(H^{-m},L^2)} + \norm{T}_{B(L^2, H^m)}$.

\emph{Step 1: Initial entropy bound.}
Assuming $m > n/2$, we will prove that whenever $0 \leq r < m-n/2$ one has 
\begin{equation} \label{ek_initial_mntwo}
e_k(i: B(H^{-m}, L^2) \cap B(L^2, H^m) \to B(L^2, L^2)) \lesssim k^{-\frac{r}{2n}}.
\end{equation}

We first claim that there is a continuous embedding 
\begin{equation} \label{b_first_estimate}
B(H^{-m}, L^2) \cap B(L^2, H^m) \subset HS(H^{-r}, L^2) \cap HS(L^2, H^r).
\end{equation}
In fact, by \eqref{hsxy_equivalent_norms} and interpolation we have 
\begin{align*}
\norm{T}_{HS(H^{-r}, L^2)}^2 &= \sum \norm{T J^r \varphi_j}_{L^2}^2 \leq \sum \norm{T}_*^2 \norm{J^r \varphi_j}_{H^{-m}}^2 \leq \norm{T}_{*}^2 \sum j^{-\frac{2(m-r)}{n}}, \\
\norm{T}_{HS(L^2, H^r)}^2 &= \sum \norm{T \varphi_j}_{H^r}^2 \leq \sum \norm{T \varphi_j}_{L^2}^{\frac{2(m-r)}{m}} \norm{T \varphi_j}_{H^m}^{\frac{2r}{m}} \leq \norm{T}_{*}^2 \sum j^{-\frac{2(m-r)}{n}}.
\end{align*}
Since $m-r > n/2$, this proves \eqref{b_first_estimate}.

The next step is to show that the embedding $\iota: HS(H^{-r}, L^2) \cap HS(L^2, H^r) \rightarrow HS(L^2, L^2)$ satisfies 
\[
e_k(\iota)    \lesssim k^{-\frac{r}{2n} }  .
\]
Using \eqref{hsxy_equivalent_norms} and \eqref{hs_ltwo_kernel}, it is easy to check that
\begin{align*}
\norm{T}_{HS(H^{-r}, L^2) \cap HS(L^2, H^r))}
& =  \norm{ TJ^r }_{HS(L^2, L^2)} + \norm{J^{r} T }_{HS(L^2, L^2)} \\
&  = \norm{  J_y^r K_T }_{L^2(M \times M)} + \norm{J_x^{r}  K_T}_{L^2(M \times M)}.
\end{align*}
Since $\norm{K_T}_{H^{r}} \leq \norm{ J_y^{r} K_T}_{L^2}+\norm{ J_x^{r} K_T}_{L^2} $, by Theorem \ref{thm_sobolev_inclusion}
one has 
\[
e_k(\iota) \leq e_k(i: H^{r}(M \times M) \to L^2(M \times M)) \lesssim k^{-\frac{r}{2n}}.
\]

We have proved that one has continuous embeddings 
\begin{align}
\label{eq:z0}
B(H^{-m}, L^2) \cap B(L^2, H^m) \subset HS(H^{-r}, L^2) \cap HS(L^2, H^r) \subset HS(L^2, L^2) \subset B(L^2, L^2)
\end{align}
where the middle embedding has entropy numbers  $e_k \lesssim k^{-\frac{r}{2n}}$. This implies \eqref{ek_initial_mntwo}.

\emph{Step 2: Improvement and derivation of the claimed entropy bound.}
Now we assume $m > 0$ and improve the bound \eqref{ek_initial_mntwo} to the claimed bound $e_k(i) \lesssim k^{-\frac{m}{2n}+\delta}$. 
  
Our strategy is to fix $N \geq 1$ and split the map $i$ as 
\[
i(T) = \alpha_N(T) + \beta_N(T)
\]
where, using the projection $P_N f = \sum_{j=1}^N (f, \varphi_j)_{L^2} \varphi_j$, one has 
\[
\alpha_N(T) = P_N T P_N, \qquad \beta_N(T) = T - P_N T P_N.
\]
  
Note that for any $s \geq t$ 
\[
\norm{P_N f}_{H^s} \leq N^{\frac{s-t}{n}} \norm{f}_{H^t}, \qquad \norm{(I - P_N) f}_{H^t} \leq N^{-\frac{s-t}{n}} \norm{f}_{H^s}.
\]
These estimates imply that 
\[
\norm{\beta_N(T) f}_{L^2} \leq \norm{(I - P_N) T f}_{L^2} + \norm{P_N T(I - P_N) f}_{L^2} \leq N^{-\frac{m}{n}} \norm{T}_{*} \norm{f}_{L^2}.
\]
On the other hand, for any $s > \max(m, n/2+1)$ one has 
\begin{align*}
\norm{\alpha_N(T) f}_{H^s} &\leq N^{\frac{s-m}{n}} \norm{T P_N f}_{H^m} \leq N^{\frac{s-m}{n}} \norm{T}_* \norm{f}_{L^2}, \\
\norm{\alpha_N(T) f}_{L^2} &\leq \norm{T}_* \norm{P_N f}_{H^{-m}} \leq N^{\frac{s-m}{n}} \norm{T}_* \norm{f}_{H^{-s}}.
\end{align*}
Thus $\alpha_N$ maps $B(H^{-m}, L^2) \cap B(L^2, H^m)$ to $B(H^{-s},L^2) \cap B(L^2, H^s)$ with norm $\leq N^{\frac{s-m}{n}}$. Combining this with \eqref{ek_initial_mntwo} using the choice $r = s - n/2 - 1/2$, we have 
\[
e_k(\alpha_N: B(H^{-m}, L^2) \cap B(L^2, H^m) \to B(L^2, L^2)) \lesssim N^{\frac{s-m}{n}} k^{-\frac{r}{2n}}.
\]
On the other hand, since $e_1(\beta_N) \leq \norm{\beta_N}$ one has 
\begin{align*}
e_1(\beta_N: B(H^{-m}, L^2) \cap B(L^2, H^m) \to B(L^2, L^2)) \leq N^{-\frac{m}{n}}.
\end{align*}
It follows from the above estimates and Lemma \ref{lem:prop_entropy}(iii) that 
\[
e_k(i) \leq e_k(\alpha_N) + e_1(\beta_N) \lesssim N^{\frac{s-m}{n}} k^{-\frac{r}{2n}} + N^{-\frac{m}{n}}.
\]
Choosing $N \sim k^{\frac{r}{2s}}$ yields that $e_k(i) \lesssim k^{-\frac{m}{2n} \frac{r}{s}} = k^{-\frac{m}{2n} (1-\frac{n+1}{2s})}$. This proves the bound $e_k(i) \lesssim k^{-\frac{m}{2n}+\delta}$ for any $\delta > 0$ by taking $s$ large.
\end{proof}

We next give a more general version of Theorem \ref{thm_entropy_bounded_operators_initial}. It is related to embedding the space of operators which are smoothing of order $m$ (resp.\ Gevrey smoothing), and whose adjoints are also smoothing, into $B(H^s, H^{-s})$. Recall that any $T \in B(H^s, H^{-s})$ has a formal adjoint $T' \in B(H^s, H^{-s})$ defined by 
\[
(T'u, v) = (u, Tv), \qquad u, v \in H^s,
\]
using the standard distributional (or $L^2$) pairing.

\begin{theorem} \label{thm_entropy_bounded_operators}
Let $M$ be a closed smooth $n$-manifold, let $s \in \mR$, and let $H^t = H^t(M)$. 
\begin{enumerate}
\item[(a)] 
Let $m > 0$, and define 
\[
Z^m(H^s,H^{-s}) := \{ T \in B(H^s, H^{-s}) \,;\, T(H^s) \subset H^{-s+m} \text{ and } T'(H^s) \subset H^{-s+m} \}.
\]
Then $Z^m(H^s,H^{-s})$ is a Banach space with norm 
\[
T \mapsto \max(\sup_{\norm{u}_{H^s}=1} \norm{Tu}_{H^{-s+m}}, \sup_{\norm{u}_{H^s}=1} \norm{T'u}_{H^{-s+m}}).
\]
If $m > 0$, the embedding $i: Z^m(H^s,H^{-s}) \to B(H^s, H^{-s})$ satisfies for any $\delta > 0$ 
\[
e_k(i) \lesssim k^{-\frac{m}{2n} + \delta}.
\]
\item[(b)] 
Let $\rho > 0$ and $1 \leq \sigma < \infty$, and define 
\[
W^{\sigma,\rho}(H^s, H^{-s}) := \{ T \in B(H^s, H^{-s}) \,;\, T(H^s) \subset A^{\sigma,\rho} \text{ and } T'(H^s) \subset A^{\sigma,\rho} \}.
\]
Then $W^{\sigma,\rho}(H^s, H^{-s})$ is a Banach space with norm 
\[
T \mapsto \max(\sup_{\norm{u}_{H^s}=1} \norm{Tu}_{A^{\sigma,\rho}}, \sup_{\norm{u}_{H^s}=1} \norm{T'u}_{A^{\sigma,\rho}}).
\]
For some $c > 0$, the embedding $i: W^{\sigma,\rho}(H^s, H^{-s}) \to B(H^s, H^{-s})$ satisfies 
\[
e_k(i) \lesssim e^{-c k^{\frac{1}{2n\sigma+1}}}.
\]
\end{enumerate}
\end{theorem}

\begin{remark}
In Theorem \ref{thm_entropy_bounded_operators} one needs to consider operators that are smoothing and also their adjoints are smoothing. In fact, if one does not consider the adjoints, the result fails since the embedding 
\[
B(H^s, H^{-s+m}) \to B(H^s, H^{-s})
\]
is not even compact. Then by Lemma \ref{lem:prop_entropy} its entropy numbers cannot go to zero. To see this, let $(\varphi_j)$ and $(\psi_k)$ be orthonormal bases of $H^s$ and $H^{-s}$, respectively, so that $H^{-s+m}$ has equivalent norm 
\[
\norm{u}_{H^{-s+m}} = ( \sum_{k=1}^{\infty} k^{2m} \abs{(u, \psi_k)_{H^{-s}}}^2 )^{1/2}.
\]
Let $T_l \in B(H^s, H^{-s+m})$ with $T_l(u) = (u,\varphi_l)_{H^s} \psi_1$. Then $\norm{T_l}_{H^s \to H^{-s+m}} = 1$, so $(T_l)$ is a bounded sequence in $B(H^s, H^{-s+m})$. But if some subsequence $(T_{l_j})$ converges in $B(H^s,H^{-s})$, the limit must be $0$ since one has $T_{l_j}(u) \to 0$ as $j \to \infty$ for any $u \in H^s$. This contradicts the fact that $\norm{T_l}_{H^s \to H^{-s}} = 1$.
\end{remark}

We begin by proving the first part of Theorem \ref{thm_entropy_bounded_operators}(a) which we reduce to the estimates from Theorem \ref{thm_entropy_bounded_operators_initial}.

\begin{proof}[Proof of Theorem \ref{thm_entropy_bounded_operators}(a)]
It follows from the closed graph theorem that 
\[
Z^m(H^s,H^{-s}) \subset B(H^s, H^{-s+m}).
\]
For any $T \in Z^m(H^s,H^{-s})$ the closed graph theorem also gives that $T'$ is bounded as a map from $H^s \to H^{-s+m}$, hence by duality $T$ is bounded as a map from $H^{s-m} \to H^{-s}$. Thus we also have 
\[
Z^m(H^s,H^{-s})\subset B(H^{s-m}, H^{-s}).
\]
It follows that $Z^m(H^s,H^{-s})$ is a Banach space with the given norm and
\[
Z^m(H^s,H^{-s}) \subset B(H^{-s}, H^{-s+m})\cap B(H^{s-m}, H^{-s}).
\]

The proof of the estimate in Theorem \ref{thm_entropy_bounded_operators}(a) is now a direct consequence of the estimate from Theorem \ref{thm_entropy_bounded_operators_initial}. Indeed, we simply factor the mapping as
\begin{gather*}
Z^m(H^s,H^{-s}) \rightarrow B(L^2,H^{m})\cap B(H^{-m},L^2)  \rightarrow B(L^2, L^2) \rightarrow B(H^s, H^{-s}), \\ 
T \mapsto J^{-s} T J^{-s} \mapsto J^{-s} T J^{-s} \mapsto T,
\end{gather*}
where as above $J^{\alpha}$ is an isometry $H^s \to H^{s-\alpha}$. By virtue of the result of Theorem \ref{thm_entropy_bounded_operators_initial} we have that the middle mapping has entropy numbers satisfying $e_k \lesssim k^{- \frac{m}{2n} + \delta}$ for any $\delta>0$. Together with the properties of entropy numbers of compositions (Lemma \ref{lem:prop_entropy}(ii)) this concludes the proof for (a).
\end{proof}

In order to deduce the remaining part (b) of Theorem \ref{thm_entropy_bounded_operators}, we first provide the following abstract lemma which gives bounds for the singular values (hence also entropy numbers) of embeddings between spaces of Hilbert-Schmidt operators. It will also be used later in our study of instability of the Calder\'on problem in low regularity.

\begin{lemma} \label{lemma_hs_embedding}
Let $X$ and $Z$ be separable Hilbert spaces with orthonormal bases $(\varphi_j)$ and $(\psi_k)$, respectively. Define subspaces $X_1 \subset X$ and $Y \subset Z$ by the norms 
\[
\norm{x}_{X_1} = ( \sum_{j=1}^{\infty} \alpha_j^2 \abs{(x, \varphi_j)_X}^2 )^{1/2}, \qquad \norm{y}_{Y} = ( \sum_{k=1}^{\infty} \beta_k^2 \abs{(y, \psi_k)_Z}^2 )^{1/2}
\]
where $0 < \alpha_1 \leq \alpha_2 \leq \ldots$ and $0 < \beta_1 \leq \beta_2 \leq \ldots$. The embedding 
\[
\iota: HS(X,Y) \to HS(X_1, Z)
\]
satisfies for any $M, N \geq 1$ 
\[
\sigma_{MN+1}(\iota) \leq \max(\frac{1}{\alpha_1 \beta_{N+1}}, \frac{1}{\alpha_{M+1} \beta_1}).
\]
\end{lemma}
\begin{proof}
By the minimax principle 
\[
\sigma_{MN+1}^2(\iota) = \lambda_{MN+1}(\iota^* \iota) = \min_{E} \max_{T \perp E, \norm{T}_{HS(X,Y)}=1} \norm{\iota(T)}_{HS(X_1,Z)}^2,
\]
where the minimum is over subspaces $E \subset HS(X,Y)$ with $\dim(E) = MN$. We choose 
\[
E = \{ T \in HS(X,Y) \,;\, (T\varphi_j, \psi_k)_Y = 0 \text{ for } j \geq M+1 \text{ or } k \geq N+1 \}.
\]
Note that $(\alpha_j^{-1} \varphi_j)$ and $(\beta_k^{-1} \psi_k)$ are orthonormal bases of $X_1$ and $Y$, respectively. Thus for any $T \perp E$ one has 
\[
\norm{\iota(T)}_{HS(X_1,Z)}^2 = \sum_{j \geq M+1 \text{ or } k \geq N+1} \alpha_j^{-2} \abs{(T \varphi_j, \psi_k)_Z}^2.
\]
Now $(y,\psi_k)_Y = \beta_k^2 (y,\psi_k)_Z$ for any $y \in Y$, so for $T \in E$ 
\[
\norm{\iota(T)}_{HS(X_1,Z)}^2 = \sum_{j \geq M+1 \text{ or } k \geq N+1} \alpha_j^{-2} \beta_k^{-2} \abs{(T \varphi_j, \beta_k^{-1} \psi_k)_Y}^2.
\]
The last quantity is $\leq \max((\alpha_1 \beta_{N+1})^{-2}, (\alpha_{M+1} \beta_1)^{-2})$ if $\norm{T}_{HS(X,Y)} = 1$.
\end{proof}

\begin{proof}[Proof of Theorem \ref{thm_entropy_bounded_operators}(b)]
First note that for $r > 0$, the space $A^{\sigma,-r}$ defined as the completion of $L^2(M)$ with respect to the $A^{\sigma,-r}$ norm is a Hilbert space that contains all $H^t$ spaces (it can be considered as a space of ultradistributions on $M$). If $T \in W^{\sigma,\rho}(H^s, H^{-s})$ and $u \in H^s$, we have for any $\eps > 0$ 
\[
\norm{Tu}_{A^{\sigma,\rho}} \lesssim \norm{u}_{H^s} \lesssim \norm{u}_{A^{\sigma,\eps}}.
\]
Since $T'$ is bounded as a map from $H^s \to A^{\sigma,\rho}$ we have by duality 
\[
\norm{Tu}_{A^{\sigma,-\eps}} \lesssim \norm{Tu}_{H^{-s}} \lesssim \norm{u}_{A^{\sigma,-\rho}}.
\]
Choosing $\eps = \rho/3$ and using the Stein-Weiss interpolation theorem, we have 
\[
\norm{Tu}_{A^{\sigma,\rho/3}} \lesssim \norm{u}_{A^{\sigma,-\rho/3}}.
\]
This shows that $W^{\sigma,\rho}(H^s, H^{-s}) \subset B(A^{\sigma,-\rho/3}, A^{\sigma,\rho/3})$. As in the proof of Theorem \ref{thm_analytic_inclusion}, it is easy to check that the embedding $A^{\sigma,\rho} \subset A^{\sigma,\rho-\eps}$ has singular values $\sigma_k \lesssim e^{-\eps j^{\frac{1}{n\sigma}}}$.
Thus as in \eqref{eq:z0}, we have a sequence of continuous embeddings 
\begin{equation} \label{w_embeddings}
W^{\sigma,\rho}(H^s, H^{-s}) \subset HS(A^{\sigma,-\rho/4}, A^{\sigma,\rho/4}) \subset HS(A^{\sigma,-\rho/8}, A^{\sigma,\rho/8}) \subset B(H^s, H^{-s}).
\end{equation}
Lemma \ref{lemma_hs_embedding} with $\alpha_j = \beta_j = e^{\frac{\rho}{8} j^{\frac{1}{n\sigma}}}$ shows that the middle embedding has singular values $\sigma_{N^2+1} \lesssim e^{-\frac{\rho}{8} (N+1)^{\frac{1}{n\sigma}}}$. Thus by Lemma \ref{lemma_singular_entropy_relation} the middle embedding has entropy numbers $e_k \lesssim e^{-c k^{\frac{1}{2n\sigma+1}}}$. This proves the bound for $e_k(i)$.
\end{proof}

\begin{remark} \label{remark_entropy_multiplier}
In inverse problems where the measurement is a Dirichlet-to-Neumann type operator, the forward operator actually maps to a subspace of $B(H^s, H^{-s})$ consisting of pseudodifferential operators. In this remark we show that in a related case one can improve the exponents in Theorems \ref{thm_entropy_bounded_operators_initial} and \ref{thm_entropy_bounded_operators} and give a sharp decay rate. 

Let $(M,g)$ be a closed smooth $n$-manifold and consider the space $\Psi^{-m}_{F}$ of Fourier multipliers of order $-m$, consisting of operators $T = T_a$ where $(a_j)_{j=1}^{\infty}$ is a sequence satisfying $\abs{a_j} \lesssim j^{-m}$, and $T_a$ is defined by 
\[
T_a f = \sum_{j=1}^{\infty} a_j (f, \varphi_j) \varphi_j.
\]
Here $\lambda_1 \leq \lambda_2 \leq \ldots $ are the eigenvalues of $-\Delta_g$, and $(\varphi_j)$ is an orthonormal basis of $L^2(M)$ consisting of eigenfunctions. We think of $\Psi^{-m}_{F}$ as a subspace of $B(H^{-m/2}, H^{m/2})$. Using the Sobolev norm $\norm{f}_{H^s}^2 = \sum_{j=1}^{\infty} j^{2s/n} \abs{(f,\varphi_j)}^2$, it is easy to see that 
\begin{align*}
\norm{T_a}_{L^2 \to L^2} &= \sup_{j \geq 1} \,\abs{a_j}, \\
\norm{T_a}_{H^{-m/2} \to H^{m/2}} &= \sup_{j \geq 1} \,j^{m/n} \abs{a_j}.
\end{align*}
It follows that the embedding $i: \Psi^{-m}_F \to B(L^2, L^2)$ may be identified with the embedding $w^{m,\infty} \to \ell^{\infty}$, where $\norm{(a_j)}_{w^{m,\infty}} = \sup_{j \geq 1} \,j^{m/n} \abs{a_j}$. Writing $i = A \circ J^m$ where $J^m: (a_j) \mapsto (j^{m/n} a_j)$ is an isometry between $w^{m,\infty}$ and $\ell^{\infty}$, the entropy numbers of $i$ are the same as the entropy numbers of 
\[
A: \ell^{\infty} \to \ell^{\infty}, \ A((a_j)) = (j^{-m/n} a_j).
\]
This is a diagonal operator whose entropy numbers are estimated in \cite[Proposition 1.3.2]{CS90}. As in the proof of Lemma \ref{lemma_singular_entropy_relation} it follows that $e_k(A) \sim k^{-m/n}$, and thus also $e_k(i: \Psi^{-m}_F \to B(L^2, L^2)) \sim k^{-m/n}$.
\end{remark}

\section{Instability in the presence of global smoothing}
\label{sec:examples1}

In this section we give some examples of instability in inverse problems based on global smoothing properties when the coefficients are smooth or real-analytic. We will be quite brief, since in Section \ref{sec:instab_low_reg} stronger results will be given for low regularity coefficients.

We first combine Lemma \ref{lemma_entropy_typical} with Theorems \ref{thm_sobolev_inclusion} and \ref{thm_analytic_inclusion} and state a result for the case where the forward map acts between Sobolev spaces. Part (a) shows that if the forward map is $C^{\infty}$ smoothing then the inverse problem cannot be H\"older stable. Part (b) shows that a Gevrey/analytic smoothing forward map leads to at best logarithmic stability.

\begin{theorem} \label{thm_smoothing_instability_sobolev}
Let $M$ and $N$ be compact smooth manifolds with or without smooth boundary, let $s, t \in \mR$, and let $\delta > 0$. Let $K$ be a closed ball in $H^{s+\delta}(M)$, and let $F$ be a map $K \to H^t(N)$. Suppose that $\omega$ is a modulus of continuity such that 
\[
\norm{f_1-f_2}_{H^s(M)} \leq \omega(\norm{F(f_1) - F(f_2)}_{H^t(N)}), \qquad f_1, f_2 \in K.
\]
\begin{enumerate}
\item[(a)] 
  If $F$ maps $K$ into a bounded set of $H^{t+m}(N)$ where $m>0$, then $\omega(t) \gtrsim t^{\frac{\delta \dim(N)}{m \dim(M)}}$. In particular, if $F$ maps $K$ into a bounded set of $H^{t+m}(N)$ for any $m > 0$, then for any $\alpha \in (0,1)$ one has $\omega(t) \gtrsim t^{\alpha}$ for $t$ small.
\item[(b)] 
If $N$ is closed and $F$ maps $K$ into a bounded set of $A^{\sigma,\rho}(N)$ for some $1 \leq \sigma < \infty$ and $\rho > 0$, then $\omega(t) \gtrsim \abs{\log t}^{-\frac{\delta(\sigma \dim(N) +1)}{\dim(M)}}$ for $t$ small.
\end{enumerate}
\end{theorem}

Next we state an analogous result where the range of the forward map is in the space $B(H^s, H^{-s})$ as in the case of Dirichlet-to-Neumann type operators. This follows immediately from Lemma \ref{lemma_entropy_typical} and Theorem \ref{thm_entropy_bounded_operators} (we also use the notation $Z^m(H^s,H^{-s})$ and $W^{\sigma,\rho}(H^s, H^{-s})$ from that theorem).

\begin{theorem} \label{thm_smoothing_instability_dnmap}
Let $M$ and $N$ be compact smooth manifolds so that $N$ has no boundary, let $r, s \in \mR$, and let $\delta > 0$. Let $K$ be  closed ball in $H^{r+\delta}(M)$, and let $F$ be a map $K \to B(H^s(N), H^{-s}(N))$. Suppose that $\omega$ is a modulus of continuity such that 
\[
\norm{f_1-f_2}_{H^r} \leq \omega(\norm{F(f_1) - F(f_2)}_{H^s \to H^{-s}}), \qquad f_1, f_2 \in K.
\]
\begin{enumerate}
\item[(a)] 
If $F$ maps $K$ into a bounded set of $Z^m(H^s,H^{-s})$ for any $m > 0$, then for any $\alpha \in (0,1)$ one has $\omega(t) \gtrsim t^{\alpha}$ for $t$ small.
\item[(b)] 
If $F$ maps $K$ into a bounded set of $W^{\sigma,\rho}(H^s, H^{-s})$ for some $1 \leq \sigma < \infty$ and $\rho > 0$, then $\omega(t) \gtrsim \abs{\log t}^{-\frac{\delta(2 \sigma \dim(N) +1)}{\dim(M)}}$ for $t$ small.
\end{enumerate}
\end{theorem}

\subsection{Unique continuation}

Let $M$ be a compact $n$-manifold with smooth boundary, and let $P$ be an elliptic second order operator on $M$ having the form 
\[
P = \Delta_g u + Bu + cu
\]
where $g$ is a smooth Riemannian metric on $M$, $B$ is a smooth vector field on $M$ and $c \in C^{\infty}(M)$. As in Section \ref{subseq_nd}, for any $u \in H^1(M)$ solving $Pu = 0$ in $M$ there is a normal derivative $\p_{\nu} u|_{\p M}$ defined weakly as an element of $H^{-1/2}(\p M)$. If $u$ is smooth, one has in local coordinates $\p_{\nu} u = g^{jk} \p_j u \nu_k|_{\p M}$ where $\nu$ is the unit outer conormal to $\p M$.

The unique continuation principle states that if $\Gamma$ is a nonempty open subset of $\p M$, then any $u \in H^1(M)$ solving $Pu = 0$ in $M$ and satisfying $u|_{\Gamma} = \p_{\nu} u|_{\Gamma} = 0$ must be identically zero. This can be made quantitative, and one has (conditional) logarithmic stability \cite{AlessandriniRondiRossetVessella}. The following result shows that logarithmic stability is optimal, at least when the underlying structures are real-analytic.

\begin{theorem}
Let $M$ and $P$ be as above, and let $\Gamma$ be a nonempty open subset of $\p M$ so that $\p M \setminus \ol{\Gamma} \neq \emptyset$. Let $\delta > 0$, and suppose that $\omega$ is a modulus of continuity so that 
\begin{equation} \label{ucp_smooth_inequality_first}
\norm{u}_{H^1(M)} \leq \omega(\norm{u}_{H^{1/2}(\Gamma)} + \norm{\p_{\nu} u}_{H^{-1/2}(\Gamma)})
\end{equation}
whenever $Pu = 0$ and $\norm{u}_{H^{1+\delta}(M)} \leq 1$.Then $\omega(t) \gtrsim t^{\alpha}$ for any $\alpha \in (0,1)$ when $t$ is small. Moreover, if $M$, the coefficients of $P$ and $\p M$ are real-analytic, then $\omega(t) \gtrsim \abs{\log t}^{-\mu}$ for $t$ small whenever $\mu > \frac{\delta n}{n-1}$.
\end{theorem}
\begin{proof}
We rewrite \eqref{ucp_smooth_inequality_first} in a form where a smoothing operator appears. We assume that $0$ is not a Dirichlet eigenvalue for $P$ in $M$, i.e.\ for any $f \in H^{1/2}(\p M)$ there is a unique solution $u = Sf \in H^1(M)$ of $Pu = 0$ in $M$ with $u|_{\p M} = f$. (Otherwise one can argue with $f$ replaced by $f-Qf$, where $Q$ is a projection to a finite dimensional space.) Elliptic regularity gives that $\norm{u}_{H^{1+t}(M)} \sim \norm{f}_{H^{1/2+t}(\p M)}$ for any $t \geq 0$. Thus \eqref{ucp_smooth_inequality_first} implies that for some $r_0 > 0$ one has 
\begin{equation} \label{ucp_smooth_inequality_second}
\norm{f}_{H^{1/2}(\p M)} \lesssim \omega(\norm{f}_{H^{1/2}(\Gamma)} + \norm{\p_{\nu} Sf}_{H^{-1/2}(\Gamma)}), \qquad \norm{f}_{H^{1/2+\delta}(\p M)} \leq r_0.
\end{equation}

We wish to get rid of the $\norm{f}_{H^{1/2}(\Gamma)}$ term on the right. This can be done by restricting to functions $f$ that vanish near $\Gamma$. Using the condition $\p M \setminus \ol{\Gamma} \neq \emptyset$, there is a neighborhood $\Gamma_1$ of $\ol{\Gamma}$ in $\p M$ and a compact domain $\Sigma \subset \p M$ with smooth boundary so that $\Sigma \cap \ol{\Gamma}_1 = \emptyset$. Let $E$ be a bounded Sobolev extension operator from $H^t(\Sigma)$ to $H^t(\p M)$, chosen so that $Eh|_{\Gamma_1} = 0$. Applying \eqref{ucp_smooth_inequality_second} to $f = Eh$, it follows that for some $r > 0$ 
\[
\norm{h}_{H^{1/2}(\Sigma)} \lesssim \omega(\norm{Ah}_{H^{-1/2}(\p M)}), \qquad \norm{h}_{H^{1/2+\delta}(\Sigma)} \leq r/2,
\]
where $A$ is the linear operator 
\[
A: H^{1/2}(\Sigma) \to H^{-1/2}(\p M), \ Ah = \chi \p_{\nu} S E h
\]
and $\chi \in C^{\infty}(\p M)$ satisfies $\chi=1$ near $\ol{\Gamma}$ and $\supp(\chi) \subset \Gamma_1$.

Now $u = SEh$ satisfies $Pu = 0$ in $M$ and $u|_{\Gamma_1} = 0$. By elliptic regularity it follows that $u$ is smooth near $\Gamma_1$, showing that $A$ maps into $H^m(\p M)$ for any $m$ (continuously, by the closed graph theorem). Theorem \ref{thm_smoothing_instability_sobolev}(a) shows that $\omega(t)$ cannot be a H\"older modulus of continuity.

Suppose now that all the structures are real-analytic. Since $u = SEh$ satisfies $Pu = 0$ in $M$ and $u|_{\Gamma_1} = 0$, elliptic regularity gives that $u$ must be real-analytic near $\Gamma_1$. Moreover, since $\norm{u}_{L^2(M)}$ is uniformly bounded, we have uniform bounds in the Cauchy estimates for $u$ by \cite[Theorem 1.3 in Chapter 8 of vol. III]{LionsMagenes}. It follows that there are uniform bounds in the Cauchy estimates for $\p_{\nu} u|_{\p M}$ in any compact subset of $\Gamma_1$. Now fix $\sigma > 1$ and choose $\chi \in C^{\infty}_c(\Gamma_1) \cap G^{\sigma}(\p M)$ so that $\chi = 1$ near $\ol{\Gamma}$. Then $A$ maps $H^{1/2}(\Sigma)$ to $A^{\sigma,\rho}(\p M)$ for some fixed $\rho > 0$ (continuously, by the closed graph theorem). It follows from Theorem \ref{thm_smoothing_instability_sobolev}(b)
that $\omega(t) \gtrsim \abs{\log t}^{-\frac{\delta(\sigma(n-1)+1}{n-1}}$. Since this is true for any $\sigma > 1$, the result follows.
\end{proof}

\subsection{Linearized Calder\'on problem}

Let $(M,g)$ be a compact $n$-manifold with smooth boundary. We consider the Dirichlet problem 
\[
(-\Delta_g + q) u = 0 \text{ in $M$}, \qquad u|_{\p M} = f,
\]
where $q \in L^{\infty}(M)$ (lower regularity coefficients will be considered in Section \ref{sec:instab_low_reg}). Assuming that $0$ is not a Dirichlet eigenvalue, for any $f \in H^{1/2}(\p M)$ there is a unique weak solution $u \in H^1(M)$. Consider the DN map 
\[
\Lambda_q: H^{1/2}(\p M) \to H^{-1/2}(\p M), \  f \mapsto \Lambda_q f:= \p_{\nu} u|_{\p M},
\]
where the normal derivative $\p_{\nu} u|_{\p M}$ is defined in the weak sense as in Section \ref{subseq_nd}.

The following standard result computes the Fr\'echet derivative of the map 
\[
\Lambda: L^{\infty}(M) \to B(H^{1/2}(\p M), H^{-1/2}(\p M)).
\]
We give the proof for completeness.

\begin{lemma} \label{lem:lin_Cal}
Let $q \in L^{\infty}(M)$ be such that $0$ is not a Dirichlet eigenvalue of $-\Delta_g+q$ in $M$, let $P_q: H^{1/2}(\partial M) \to H^1(M)$ be the solution operator for the Dirichlet problem 
\[
(-\Delta_g +q)P_q f = 0 \text{ in $M$}, \qquad P_q f|_{\p M} = f,
\]
and let $G_q: H^{-1}(M) \to H^1_0(M)$ be the Green operator with vanishing Dirichlet boundary values, 
\[
(-\Delta_g +q)G_q F = F \text{ in $M$}, \qquad G_q F|_{\p M} = 0.
\]
Then the linearized DN map $A_q = (D\Lambda)_q$ is the operator 
\begin{gather*}
A_q: L^{\infty}(M) \rightarrow B(H^{1/2}(\p M), H^{-1/2}(\p M)), \\
A_q(h) f = \p_{\nu} G_q(-h P_q f)|_{\partial M}.
\end{gather*}
\end{lemma}
\begin{proof}
If $\norm{h}_{L^{\infty}}$ is small then $\Lambda_{q+h}$ is well defined. Given $f \in H^{1/2}(\p M)$ one has 
\[
\Lambda_{q+h} f - \Lambda_q f = \p_{\nu}(P_{q+h} f - P_q f)|_{\p M}.
\]
The function $w := P_{q+h} f - P_q f \in H^1_0(M)$ solves 
\[
(-\Delta_g + q)w = -hw -h P_q f \mbox{ in } M, \ w = 0 \mbox{ on } \partial M,
\]
which implies that 
\[
w = G_q(-h P_q f) + G_q(-hw).
\]
If $\norm{h}_{L^{\infty}}$ is chosen small enough, one has $\norm{G_q(-hw)}_{H^1} \leq \frac{1}{2} \norm{w}_{H^1}$ and hence 
\[
\norm{w}_{H^1} \lesssim \norm{h}_{L^{\infty}} \norm{f}_{H^{1/2}}.
\]
It follows that 
\[
(\Lambda_{q+h} - \Lambda_q - A_q(h))f = \p_{\nu}(w - G_q(-h P_q f))|_{\p M} = \p_{\nu}(G_q(-hw))|_{\p M}.
\]
The $H^{-1/2}(\p M)$ norm of the last quantity is $\lesssim \norm{h}_{L^{\infty}} \norm{w}_{H^1} \lesssim \norm{h}_{L^{\infty}}^2 \norm{f}_{H^{1/2}}$. This proves that the Fr\'echet derivative of $q \mapsto \Lambda_q$ at $q$ is $A_q$.
\end{proof}

We note that if $q$ is smooth and if $h$ vanishes near $\p M$, then $A_q(h)$ is a smoothing operator. This implies strong instability properties for the linearized Calder\'on problem where one would like to determine $h$ from the knowledge of $A_q(h)$.

\begin{theorem}
Let $M' \Subset M^{\mathrm{int}}$, let $s > n/2$, and let $\delta > 0$. Consider $A_q$ as an operator $H^s_0(M') \to B(H^{1/2}(\p M), H^{-1/2}(\p M))$ and suppose that the linearized Calder\'on problem has the stability estimate 
\[
\norm{h}_{H^s} \leq \omega(\norm{A_q(h)}_{B(H^{1/2},H^{-1/2})}), \qquad \norm{h}_{H^{s+\delta}} \leq 1.
\]
If $q \in C^{\infty}(M)$, then $\omega(t)$ cannot be a H\"older modulus of continuity. Moreover, if $M$, $g$, $\p M$ and $q$ are real-analytic, then $\omega(t) \gtrsim \abs{\log t}^{- \frac{\delta(2n-1)}{n}}$ for $t$ small.
\end{theorem}
\begin{proof}
Since $h$ vanishes in $M \setminus M'$, the function $u = G_q(-h P_q f)$ solves $(-\Delta_g+q)u = 0$ in $M \setminus M'$ with $u|_{\p M} = 0$. If $q \in C^{\infty}(M)$ it follows that $A_q(h)$ maps $H^{1/2}(\p M)$ boundedly to $H^m(\p M)$ for any $m \geq 0$. By Theorem \ref{thm_smoothing_instability_dnmap}(a) $\omega(t)$ cannot be a H\"older modulus of continuity. Similarly, if all the structures are real-analytic then $u$ is real-analytic near $\p M$ with $\norm{u}_{H^1} \lesssim \norm{h}_{L^{\infty}} \norm{f}_{H^{1/2}}$, so $u$ satisfies uniform Cauchy estimates by \cite[Theorem 1.3 in Chapter 8 of vol. III]{LionsMagenes}. Consequently $A_q(h)$ maps $H^s_0(M')$ boundedly into $W^{1,\rho}(H^{\frac{1}{2}}, H^{- \frac{1}{2}})$ for some $\rho > 0$ (this uses that $A_q(h)$ is formally self-adjoint). By Theorem \ref{thm_smoothing_instability_dnmap}(b) $\omega(t)$ is at best logarithmic with the given exponent.
\end{proof}

\subsection{Calder\'on problem}

We next turn to the instability of the classical Calder\'on problem in smooth and analytic settings (again, results in a low regularity framework will be presented in the next section) and thus prove Theorem \ref{thm:Cald1}.

\begin{proof}[Proof of Theorem \ref{thm:Cald1}]
We seek to apply Theorem \ref{thm_smoothing_instability_dnmap} to the map 
\begin{align*}
F: K \rightarrow B(H^{\frac{1}{2}}(\partial M), H^{-\frac{1}{2}}(\partial M)), \ q \mapsto \Gamma(q):= \Lambda_{q}-\Lambda_{0},
\end{align*}
where $K = \{q \in H^{s}(M): \ \supp(q) \subset M', \ \|q\|_{H^{s+\delta}(M)} \leq C \mbox{ and } \|q\|_{L^{\infty}(M)}\leq \lambda_1/2\}$.
Now by elliptic regularity and the assumption that $\|q\|_{L^{\infty}(M)} \leq \lambda_1/2$, for any $q \in K$ the map $F$ indeed maps into $B(H^{\frac{1}{2}}(\partial M), H^{-\frac{1}{2}}(\partial M))$ (uniformly in $q \in K$) as for $f\in H^{\frac{1}{2}}(\partial \Omega)$ we have that $\Gamma(q)f = \p_{\nu} u_f - \p_{\nu} u_f^0 \in H^{-\frac{1}{2}}(\partial M)$, where $u_f, u_f^0$ are solutions to 
\begin{align*}
-\D_g u + V u &= 0 \mbox{ in } M,\\
u & = f \mbox{ on } \partial M,
\end{align*}
with $V= q$ for $u_f$ and $V=0$ for $u_f^0$. Further, we claim that $F \in Z^m(H^{\frac{1}{2}}, H^{-\frac{1}{2}})$ for any $m>0$. Indeed, this follows from the fact that $v:=u_f-u_0$ satisfies the equation
\begin{align*}
-\D_g v + q v &= -q u_0 \mbox{ in } M,\\
v & = 0 \mbox{ on } \partial M.
\end{align*}
The fact that $\supp(q) \subset M'$ and elliptic regularity give that $\partial_{\nu} v \in C^{\infty}(\partial M)$ (uniformly in $q$ as $\|q\|_{L^{\infty}(M)}\leq \lambda_1/2$). In particular, $\Gamma(q) (H^{\frac{1}{2}}(\partial M)) \subset H^{-\frac{1}{2}+m}(\partial M)$ for any $m >0$ with uniform bounds in $q$. Since $\Gamma(q)$ is a self-adjoint operator, this implies that $\Gamma(q)\in Z^m(H^{\frac{1}{2}},H^{- \frac{1}{2}})$ for each $m>0$. Invoking Theorem \ref{thm_smoothing_instability_dnmap}(a) then implies the impossibility of H{\"o}lder estimates.

If moreover, $M,g$ and $\partial M$ are real analytic, then by \cite[Theorem 1.2 in Chapter 8 of vol. III]{LionsMagenes} $\partial_{\nu} v$ is analytic  (with uniform bounds in $q$) and hence $\Gamma(q) \subset W^{1,\rho}(H^{\frac{1}{2}}, H^{-\frac{1}{2}})$ for some $\rho >0$ for which we again use the self-adjointness of $\Gamma(q)$. The logarithmic lower bound then follows from Theorem \ref{thm_smoothing_instability_dnmap}(b).
\end{proof}

\section{Instability at low regularity}
\label{sec:instab_low_reg}

In the sequel we seek to provide a further instability mechanism, related to direct singular value or entropy bounds, showing that a strong regularity improvement (as in our analyticity arguments from Section \ref{sec:abstract}) is not the only mechanism leading to exponential instability. The main, common mechanism of all our results should rather be regarded as a ``compressing mechanism''. In order to prove this, we estimate the associated singular values and exploit a balance between gaining some decay from regularity and loosing some control through growing constants. We illustrate this argument by applying it to a number of model problems including the backward heat equation with low regularity space-time dependent coefficients (Section \ref{sec:back_heat}), the unique continuation property (Section \ref{sec:UCP}) and the Calder\'on problem (thus proving Theorem \ref{thm:thm_Calderon} in Section \ref{sec:Calderon}). In all these settings our arguments imply that in spite of the low regularity of the coefficients of the problem, one can prove the same instability results for
the associated inverse problems. This gives a complete answer to the question (Q3) from the introduction for the discussed model problems. Most of these instability results are (possibly up to the precise exponents) sharp.

\subsection{Abstract setup for linear inverse problems}

We begin with instability results for linear inverse problems. The following variant of Lemma \ref{lemma_entropy_typical} shows that it is sufficient to find some way of proving decay for the singular values of the forward operator. We only state a version related to exponential instability, which will be sufficient below.

\begin{theorem} \label{thm_lowreg_instability_sobolev}
Let $A: X \to Y$ be a compact injective linear operator between separable Hilbert spaces with $X$ infinite dimensional. Let $X_1 \subset X$ be a closed subspace so that $i: X_1 \to X$ is compact with $e_k(i) \gtrsim k^{-m}$ for some $m > 0$. Let $K = \{ u \in X \,;\, \norm{u}_{X_1} \leq r \}$ for some $r > 0$. Assume that the singular values of $A: X \to Y$ satisfy for some $\rho, \mu > 0$ 
\[
\sigma_k(A) \lesssim e^{-\rho k^{\mu}}, \qquad k \geq 1.
\]
Then there is $c > 0$ with the following property: for any $\eps > 0$ small enough there is $u = u_{\eps}$ such that
\begin{equation} \label{instab_restatement}
\|u\|_{X} \geq \eps, \quad \|u\|_{X_1} \leq r, \quad \|Au\|_{Y} \leq \exp(-c \eps^{-\frac{\mu}{m(\mu+1)}}).
\end{equation}
In particular, if one has the stability property 
\[
\norm{u}_X \leq \omega(\norm{Au}_Y), \qquad u\in K,
\]
then necessarily $\omega(t) \gtrsim \abs{\log\,t}^{-\frac{m(\mu+1)}{\mu}}$ for $t$ small. 
\end{theorem}
\begin{proof}
Let $(\varphi_j)_{j=1}^{\infty}$ be a singular value orthonormal basis of $X$, so that $A \varphi_j = \sigma_j \psi_j$ where $\psi = (\psi_j)_{j=1}^{\infty}$ is orthonormal in $Y$. Let $Y'$ be the span of $\psi$ in $Y$. We define a subspace $a^{\sigma,\rho} = a^{\sigma,\rho}_{1,Y',\psi}$ of $Y'$ using the norm 
\[
\norm{v}_{a^{\sigma,\rho}} = \left( \sum_{j=1}^{\infty} e^{2\rho j^{\frac{1}{\sigma}}} \abs{(v,\psi_j)}^2 \right)^{1/2}.
\]
Since $Au = \sum_{j=1}^{\infty} (u,\varphi_j) \sigma_j \psi_j$, we have $(Au, \psi_j) = \sigma_j (u,\varphi_j)$ and 
\[
\norm{Au}_{a^{1/\mu,\rho}}^2 = \sum_{j=1}^{\infty} e^{2\rho j^{\mu}} \sigma_j^2 \abs{(u,\varphi_j)}^2 \leq C^2 \norm{u}_{X}^2.
\]
Thus $A(K)$ is contained in a bounded subset of $Y_1 := a^{1/\mu, \rho}$. The embedding $j_1: Y_1 \to Y' = h^0_{1,Y',\psi}$ satisfies $e_k(j_1) \lesssim e^{-c k^{\frac{\mu}{\mu+1}}}$ by Theorem \ref{thm_analytic_inclusion}, and thus also the embedding $j: Y_1 \to Y$ has these entropy bounds. It follows from Lemma \ref{lemma_entropy_typical} that $\omega(t) \gtrsim \abs{\log\,t}^{-\frac{m(\mu+1)}{\mu}}$ for $t$ small. By Remark \ref{rmk:equiv_Mandache} this conclusion can be rewritten as \eqref{instab_restatement}.
\end{proof}

\subsection{The backward heat equation}
\label{sec:back_heat}

As a first model case for the type of arguments that we have in mind, we consider parabolic systems of the form
\begin{align}
\label{eq:heat}
\begin{split}
(\p_t - \nabla \cdot a \nabla) u & = 0 \mbox{ in } \Omega \times [0,1],\\
u & = 0 \mbox{ on } \partial \Omega \times [0,1],\\
u & = u_0 \mbox{ on } \Omega \times \{0\},
\end{split}
\end{align}
where $\Omega \subset \R^n$ is an open, bounded $C^1$ domain, $u:\Omega \rightarrow \C^m$ and $a = (a_{ij}^{\alpha \beta}(x,t))^{\alpha, \beta \in \{1,\dots,m\}}_{i,j\in\{1,\dots,n\}} $ are bounded functions which satisfy a coercivity condition, i.e.\ for which there exist constants $\lambda>0,\kappa>0$ such that 
\begin{align}
\label{eq:ell2}
\Ree \int_{\Omega} a \nabla v \cdot \overline{\nabla v} \,dx \geq \lambda\|v\|_{H^1(\Omega)}^2 - \kappa \|v\|_{L^2(\Omega)}^2
\mbox{ for all } v \in H^1(\Omega) \mbox{ and a.e.\ } t \in [0,1].
\end{align}
Here and in the following discussion $\nabla$ always denotes the spatial gradient, $\overline{f}$ denotes the complex conjugate of $f$ and
\begin{align}
\label{eq:ell1}
a \nabla u \cdot \overline{\nabla u}
:= \sum\limits_{i,j\in \{1,\dots,n\}, \ \alpha,\beta \in \{1,\dots,m\}} a^{\alpha \beta}_{ij}\p_j u^{\beta} \overline{\p_i u^{\alpha}}.
\end{align}
Here and in the sequel, with slight abuse of notation, we write $\|u \|_{X(\Omega)}$ instead of $\|u\|_{X(\Omega, \C^m)}$. We remark that all the following results are in particular valid for scalar parabolic equations (which just corresponds to the case $m=1$). As there is no difference in the argument with respect to the systems case, we have opted to formulate the results in the systems case directly. We refer to \cite{McLean} for some background on energy estimates for elliptic systems.

In this set-up we are interested in the inverse problem associated with the forward map 
\begin{align}
\label{eq:J(1,0)}
i_{1,0}: L^2(\Omega)\ni u_0 \mapsto u(1) \in L^2(\Omega).
\end{align}
It is well-known that the backward heat equation is highly ill-posed. In spite of this for $C^1$ regular metrics $a$  the map $i_{1,0}$ is injective. We refer to \cite[Chapter 3.1]{Isakov} and \cite[Section 9.1]{Yamamoto} for a discussion of the scalar backward heat equation. The backward uniqueness property of the heat equation can be quantified to yield quantitative backward uniqueness results in compact sets. Under suitable regularity assumption on the coefficients in the equation, for the recovery of $u(\cdot, t_0)$ with $t_0 \in (0,T)$ this turns into a \emph{H{\"o}lder} stability estimate: 

\begin{proposition}[\cite{Vessella}, estimate (3.65)]
\label{prop:parabolic_pos}
Let $u$ be a solution to \eqref{eq:heat} with $m=1$ and $a \in C^1(\overline{\Omega}\times [0,T], \R^{n\times n})$ uniformly elliptic. Let $t_0 \in (0,T)$. Then there exists $\theta = \theta(t) \in (0,1)$ and $C>0$ depending on $\|a\|_{C^1(\overline{\Omega}\times [0,T])}$ such that
\begin{align}
\label{eq:Hoelder}
\|u( t_0)\|_{L^2(\Omega)} \leq C \|u_0\|_{L^2(\Omega)}^{1-\theta} \|u( T)\|_{L^2(\Omega)}^{\theta}.
\end{align}
\end{proposition}

For isotropic and time-independent metrics $a$, the function $\theta(t)$ can for instance be chosen to be $\theta(t) = \frac{t}{T}$ (see \cite[(3.1.9)]{Isakov}). For $t_0 \rightarrow 0$, this H{\"o}lder estimate degenerates, resulting in an only \emph{logarithmic} bound (see \cite[equation (9.2)]{Yamamoto}, \cite[Chapter 3]{Isakov} or \cite{Klibanov_parabolic}) of the following type: If $\|u(0)\|_{H^2(\Omega)}\leq M$, there exist $\nu>0$, $C_M\geq 1$ such that
\begin{align}
\label{eq:parabolic}
\|u(0)\|_{L^2(\Omega)} \leq C_M |\log(\|u(T)\|_{L^2(\Omega)})|^{-\nu}.
\end{align}

 These quantitative backward uniqueness estimates can be obtained through various methods with Carleman estimates possibly providing the most robust arguments (e.g.\ allowing for low regularity coefficients). We refer to Section 9.1 of the survey article \cite{Yamamoto}, to \cite[Theorems 3, 4]{Klibanov_parabolic} and the references therein for more background on the positive results in this direction.

It is the \emph{optimality} of the logarithmic stability estimate of the type \eqref{eq:parabolic} that we are investigating in this section (we refer to Section \ref{sec:UCP_int} for a discussion of an elliptic analogue of the optimality of the H{\"o}lder estimates in \eqref{eq:Hoelder}). If $a(x,t)$ is independent of $t$ this follows easily from eigenvalue estimates (see Lemma \ref{lem:heat_time_indep}), but in the time-dependent case a different argument is needed. 
As our main result, we show that the logarithmic moduli of continuity for the inversion of the map $i_{1,0}$ are optimal. Further, we prove that this behaviour persists in the low regularity setting for parabolic systems which was described above, see \eqref{eq:heat}, e.g.\ in settings in which the metric $a$ is only bounded and no $C^{\infty}$ or analytic smoothing properties can be used.  

\begin{theorem}
\label{thm:instab_low_reg}
Let $\Omega$ and $a$ be as above, and let $\ell > 0$. If for some modulus of continuity $\omega$ one has 
\begin{align}
\label{eq:instab_low_regest}
\|u_0\|_{L^2(\Omega)} \leq \omega(\|u(1)\|_{L^2(\Omega)}), \qquad \|u_0\|_{H^{\ell}(\Omega)} \leq 1,
\end{align}
where $u$ is the solution to \eqref{eq:heat}, then $\omega(t) \gtrsim \abs{\log\,t}^{-\frac{\ell(n+4)}{2n}}$ for $t$ small.
\end{theorem}

For higher regularity metrics $a$ a similar result can be obtained.
We remark that without any major modification it is possible to replace the $L^2(\Omega)$ norm $\|u(1)\|_{L^2(\Omega)}$ in the estimate \eqref{eq:instab_low_regest} by an $H^{\delta}(\Omega)$ norm for $\delta \in (0,\delta_0)$, where $\delta_0>0$ denotes the regularity exponent from the Sneiberg-type Lemma \ref{lem:Sneiberg} below. Indeed, in order to obtain this, we just apply the result from \eqref{eq:instab_low_regest} with $t=1$ replaced by $t= \frac{1}{2}$ and then use the smoothing property from Lemma \ref{lem:Sneiberg} to infer that $\|u(1)\|_{H^{\delta}(\Omega)} \leq C \|u(\frac{1}{2})\|_{L^2(\Omega)}$.

Theorem \ref{thm:instab_low_reg} will follow from Theorem \ref{thm_lowreg_instability_sobolev} together with the fact that the singular values associated with the mapping \eqref{eq:J(1,0)}
decay exponentially.
In spite of the low regularity set-up for the coefficients, we prove that the remaining little bit of regularisation leads to exponentially decaying singular value estimates.

\begin{theorem} \label{thm:entropy_backward_heat}
Let $\Omega$,  $a$ be as above. There is a constant $c > 0$ so that the singular values of the map $i_{1,0}$ in \eqref{eq:J(1,0)} satisfy 
  \[ \sigma_k(i_{1,0}) \le e^{-c k^{\frac{2}{n+2}}}. \]
\end{theorem}

We further recall the validity of entropy estimates in irregular domains, which had been stated in Proposition \ref{prop_ek_nonsmooth} and which will be used in the proof of Theorem \ref{thm:instab_low_reg}.

\begin{proof}[Proof of Theorem \ref{thm:instab_low_reg} using Theorem \ref{thm:entropy_backward_heat}]
We use Theorem \ref{thm_lowreg_instability_sobolev} with $X = L^2(\Omega)$, $X_1 = H^{\ell}(\Omega)$, and $A = i_{1,0}$. One has $e_k(i) \sim k^{-\frac{\ell}{n}}$ by Proposition \ref{prop_ek_nonsmooth} and $\sigma_k(A) \lesssim e^{-c k^{\mu}}$ with $\mu = \frac{2}{n+2}$ by Theorem \ref{thm:entropy_backward_heat}. Then Theorem \ref{thm_lowreg_instability_sobolev} yields $\omega(t) \gtrsim \abs{\log\,t}^{-\frac{\ell(n+4)}{2n}}$.
\end{proof}

Thus it remains to prove Theorem \ref{thm:entropy_backward_heat}. The first step is a regularity result.

\begin{lemma} 
\label{lem:Sneiberg}
Let $a$, $\Omega$ be as above and $s \in [0,1)$, $t > 0$, and $s+2t \leq 1$.
Then, there exist $\delta>0$ and $C>0$ depending only on $\lambda$, $n$ and $\Omega$ so that 
  \[ \Vert u(s+t) \Vert_{H^\delta(\Omega)} \le C t^{-\frac{\delta}2} \Vert u(s) \Vert_{L^2(\Omega)}. \]
\end{lemma}

\begin{proof}
We argue in three steps: First we extend the problem to a heat equation with controlled right hand side on $\R^n \times (0,4)$ and $L^2$ bounded initial data. Then, we prove an interpolation result. Finally, we return to the desired estimate.

\emph{Step 1: Extension to the whole space and energy estimates.}
We recall that in \cite[Theorem 8.1]{ABES19} it is proved that for solutions to 
\begin{align}
\label{eq:inhom_eq_heat}
\p_t w - \nabla \cdot a \nabla w &= f + \p_i F^i \mbox{ in } \R^n \times (0,4)
\end{align}
the following a priori estimate holds for some $\alpha \in (0,1)$ and $p>2$:
\begin{align}
\label{eq:ABES19}
\begin{split}
&\|\nabla w\|_{L^p(\Omega_1)} + \sup\limits_{t\in I_1}\|w\|_{L^p(B_1)} + \sup\limits_{s,t\in I_1}\frac{\|u(t,\cdot)-u(s,\cdot)\|_{L^2(B_1)}}{|t-s|^{\alpha}} \\
&\leq C \left( \|w\|_{L^2(2\Omega_{1})} + \|f\|_{L^p(2 \Omega_1)} + \|F \|_{L^{p*}(2\Omega_{1})}\right),
\end{split}
\end{align}
where $\Omega_r = B_r \times I_r $, $I_r = [r/2,r]$ and $2 \Omega_{r} = B_{2r} \times [r/4,2r]$.

Now in our setting we consider solutions $u$ to \eqref{eq:heat} in the domain $\Omega$. Thus, by localizing by means of a partition of unity the equation is locally of the form \eqref{eq:ABES19} with $f$ and $F$ depending on $u$, its first derivatives and the cut-off functions. Moreover, in the interior of $\Omega$ we may assume that $f = 0 =F$ and thus obtain estimates in \eqref{eq:ABES19} which only involve quantities controlled by the usual energy estimates on the right hand side. We seek to obtain similar bounds in the neighbourhood of $\partial \Omega$. 

To this end, it suffices to consider solutions $\bar{u}$ which are supported in a patch in a neighbourhood $U$ of $\partial \Omega$. After a change of coordinates and by choosing the support of $u$ possibly even smaller, it is possible to assume that $U=B_2(0)$ and that $U\cap \partial \Omega \subset \R^{n-1} = B_2(0)\cap \R^{n-1}$. Now using the homogeneous Dirichlet conditions and reflecting $u$ oddly across the boundary $\R^{n-1}$ again gives a solution $v$ to the heat equation with a divergence form right hand side which is now defined in the whole space. 
The resulting inhomogeneities $f,F$ depend on $u$ and $\nabla u$, the coefficient matrix $a$, the cut-off functions and the change of coordinates. In $B_{1/2}(0)$ we may, however, assume that $f=0$ and $F=0$ as the cut-offs are not active there. Hence, a (rescaled) version of the energy estimate \eqref{eq:ABES19} holds there with a right hand side which only depends on $v$ (with $f=0,F=0$).
In $B_2(0)\setminus B_{1/2}$ we may estimate $f,F$ in terms of $u$ and $\nabla u$ in the interior of some other boundary patches (thus corresponding to the $B_{1/2}$ setting for this other patch in which we have $f=0=F$) or in some interior patches (where also $f=0=F$). Thus, in the original coordinates, only space-time $L^2$ norms of $u$ in some of the other patches appear as right hand sides for these estimates as well.

Thus, patching together all the local estimates for the functions $\bar{u}$ on the different supports of the partition of unity (using that only finitely many of them overlap), we obtain a new function $\tilde{u}$ which satisfies
\begin{align}
\label{eq:Thm8.1}
\begin{split}
&\|\nabla \tilde{u}\|_{L^p(\R^n \times [1/2,1])} + \sup\limits_{t\in [1/2,1]}\|\tilde{u}\|_{L^p(\R^{n})} + \sup\limits_{s,t\in I_1}\frac{\|u(t,\cdot)-u(s,\cdot)\|_{L^2(\R^n)}}{|t-s|^{\alpha}}  \\
&\leq C  \|\tilde{u}\|_{L^2(\R^{n+1})}.
\end{split}
\end{align}
Extending $u_0$ by zero outside of $\Omega$, energy estimates further yield that 
\begin{align*}
\|\tilde{u}\|_{L^2(\R^{n+1})} \leq C \|u_0\|_{L^2(\R^n)}.
\end{align*}

\emph{Step 2: Interpolation on $\R^n \times [1/2,1]$.}

We next prove the central interpolation result, starting from the observation (see \eqref{eq:Thm8.1}) that
\begin{align*}
\tilde{u} \in L^2(I,H^1(\R^n)) \mbox{ and } \sup\limits_{s,t\in I}\frac{\|u(t,\cdot)-u(s,\cdot)\|_{L^2(\R^n)}}{|t-s|^{\alpha}} \leq C<\infty,
\end{align*}
where $I = [\frac{1}{2},1]$. Moreover, using the arguments from Step 1, without loss of generality, we may assume that $\tilde{u}$ and $\tilde{u}_0$ are compactly supported in space and time.
We seek to prove that under these conditions $\tilde{u}\in L^{\infty}(I,H^{\delta}(\R^n))$ for some small $\delta>0$ (depending on $\alpha \in (0,1)$). 

In order to prove the claim, we consider the temporal convolution of $\tilde{u}$ with a one-dimensional standard mollifier $\eta: \R \rightarrow \R$. For $j\in \N$, we set $\eta_{j}(t):= 2^{j}\eta(2^{j}t)$ and note that it suffices to prove that
\begin{align}
\label{eq:aim_aux}
\| \tilde{u} \ast \eta_{ j} - \tilde{u} \ast \eta_{j+1}\|_{L^{\infty}(I, H^{\delta}(\R^n))} \leq C 2^{- \gamma j}, \ j\in \N,
\end{align}
where $\gamma>0$ depends on $\delta$ and where $\tilde{u} \ast \eta_{ j}$ denotes mollification in $t$. Indeed, by a telescope sum, from \eqref{eq:aim_aux} we obtain that $(\tilde{u}_j(t))_{j\in \N}:= (\tilde{u}\ast \eta_j(t))_{j\in \N}$ is a Cauchy sequence in $L^{\infty}(I, H^{\delta}(\R^n))$, hence a limit exists as $j \rightarrow \infty$. Since it holds that also pointwise almost everywhere $\tilde{u}_j(t) \rightarrow \tilde{u}(t)$, we have that $\tilde{u}_j(t) \rightarrow \tilde{u}(t)$ in $L^{\infty}(I, H^{\delta}(\R^n))$. 

Defining $\varphi(t)= \eta(t)-2 \eta(2 t)$, our claim \eqref{eq:aim_aux} reduces to proving that 
\begin{align}
\label{eq:aim}
\sup\limits_{t\in I}\|(\tilde{u} \ast \varphi_j)(t)\|_{H^{\delta}(\R^n)} \leq C 2^{-\gamma j}.
\end{align}
We deduce this by interpolating between an $L^2(\R^n)$ bound and a $H^1(\R^n)$ estimate.
To this end, on the one hand, we note that by virtue of the fact that $\sup\limits_{s,t\in I_1}\frac{\|u(t,\cdot)-u(s,\cdot)\|_{L^2(\R^n)}}{|t-s|^{\alpha}}\leq C <\infty$  we have
\begin{align*}
\|(\tilde{u}\ast \varphi_j)(t)\|_{L^2(\R^n)} &\leq 2^{-j \alpha}C .
\end{align*}
On the other hand, by Cauchy-Schwarz we obtain that 
\begin{align*}
\|(\tilde{u}\ast \varphi_j)(t)\|_{H^1(\R^n)} &\leq \|\|\tilde{u}\|_{L^2(I)}(\cdot)+ \|\nabla \tilde{u}\|_{L^2(I)}(\cdot)\|_{L^2(\R^n)} \|\eta_j\|_{L^2(\R)}\\
& \leq  2^{j /2}\|\tilde{u}\|_{L^2(I,H^1(\R^n))}.
\end{align*}
Using that 
\begin{align*}
\|(\tilde{u} \ast \varphi_j)(t)\|_{H^{\delta}(\R^n)} \leq \|(\tilde{u} \ast \varphi_j)(t)\|_{H^1(\R^n)}^{\delta} \|(\tilde{u} \ast \varphi_j)(t)\|_{L^2(\R^n)}^{1-\delta},
\end{align*}
choosing $\delta \in (0,1)$ (depending on $\alpha$) sufficiently small and inserting the $L^{\infty}L^2$ and $L^{\infty}H^1$ bounds from above as well as the bounds from \eqref{eq:Thm8.1}, then implies \eqref{eq:aim}. \\

\emph{Step 3: The higher regularity estimate.}
Last but not least, we return to the estimate of the lemma. By the previous two steps, we infer that for some constant $C>0$ and for any solution $\tilde{u}$ obtained from a solution $u$ to \eqref{eq:heat} as described in Step 1 it holds
\begin{align*}
\|\tilde{u}(1)\|_{H^{\delta}(\R^n)} \leq C \|u_0\|_{L^2(\R^{n})}.
\end{align*}
Rescaling parabolically and using that $t\in (0,1)$, we thus obtain that
\begin{align*}
\|\tilde{u}(t)\|_{H^{\delta}(\R^n )} \leq C t^{-\frac{\delta}{2}} \|u_0\|_{L^2(\R^{n})}.
\end{align*}
Noting that $\|u(t)\|_{H^{\delta}(\Omega)}\leq C \|\tilde{u}(t)\|_{H^{\delta}(\R^n )} $, then concludes the argument.
\end{proof}

With the regularity from Lemma \ref{lem:Sneiberg} at our disposal, we address the proof of Theorem \ref{thm:entropy_backward_heat}.

\begin{proof}[Proof of Theorem \ref{thm:entropy_backward_heat}]
  
We first use Proposition \ref{prop_ek_nonsmooth} and note that the compact embedding $\iota: H^\delta (\Omega) \to L^2 (\Omega)$ has bounds for its singular values of the form
\begin{align}
\label{eq:embedding}
\sigma_k(\iota) \leq C k^{-\frac{\delta}n} .
\end{align}

  Let $i_{s+t,s}: L^2(\Omega) \to H^{\delta}(\Omega) \to L^2(\Omega)$, $s \in [0,1), t, s+t \in (0,1]$, be the map which takes $u(s)$ to $u(s+t)$ where $u$ solves \eqref{eq:heat} in $(s,s+t) \times \Omega$.
Then, by the bound from Lemma \ref{lem:Sneiberg} and the estimate \eqref{eq:embedding} we have
  \[ \sigma_k(i_{s+t,s}) \le C k^{-\frac{\delta}n} t^{-\frac{\delta}2}.  \]
  Using the behaviour of singular values under composition of maps, we iterate this $N$ times (with $N\leq k$ to be determined below) to infer
  \[ \sigma_k(i_{1,0}) \le \prod_{j=1}^{N} \sigma_{k/N} ( i_{j/N, (j-1)/N})
  \le ( C^{\frac1{\delta} } (k/N)^{-\frac1n} N^{\frac{1}2} )^{\delta N}.
  \]
Roughly minimizing the right hand side by choosing $N$ as a function of $k$, we obtain
  \[ N(k) = \rho  k^{\frac{2}{n+2}} \]
with $\rho$ a small positive constant to be determined below. Indeed, this follows from computing
    \[ 0 = \frac{d}{dN} \exp( N ( \ln C +(\frac{1}2+ \frac1n) \delta \ln N  - \frac1n \delta \ln k )) \]
  which is equivalent to
  \[  \ln C + \frac{n+2}{2n}  \delta \ln N - \frac1n \delta \ln k + \delta (\frac{n+2}{2n}) = 0        \]
or respectively,
  \[ \ln N = \frac{2}{n+2} \ln k - 1 - \frac{2n}{n+2} \frac{1}{\delta} \ln C. \]

Then,
\[ \sigma_k(i_{1,0}) \le \left( C \rho^{\delta \frac{n+2}{2n}} \right)^{\rho  k^{\frac{2}{n+2}}}.    \]    
Choosing $\rho>0$ such that $ \rho^{\frac{n+2}{2n}\delta} < C^{-1}$, we obtain that with a constant $c > 0$ 
\begin{align}
\label{eq:entropy}
 \sigma_k(i_{1,0}) \le e^{-c k^{\frac{2}{n+2}}} .  
\end{align} 
This concludes the argument.
\end{proof}

As an instructive comparison, we note that in the time-independent case for $m=1$, a different exponent appears in Theorem \ref{thm:instab_low_reg}:

\begin{lemma}
\label{lem:heat_time_indep}
Let $\Omega$ and $a$ be as in \eqref{eq:heat}, \eqref{eq:ell2} and \eqref{eq:ell1} with $m=1$ and assume that $a$ does not depend on $t$. If for some modulus of continuity one has
\begin{align*}
\|u_0\|_{L^2(\Omega)} \leq \omega(\|u(1)\|_{L^2(\Omega)}), \ \|u_0\|_{H^1(\Omega)}\leq 1,
\end{align*}
where $u$ is the solution to \eqref{eq:heat}, then $\omega(t)\gtrsim |\log(t)|^{- \frac{n+2}{2n}}$ for $t$ small.
\end{lemma} 

\begin{proof}
The argument follows in exactly the same way as the proof of Theorem \ref{thm:entropy_backward_heat}, however noting that in the time-independent case $\sigma_k(i_{1,0}) = e^{-\lambda_k}\leq c^{k^{\frac{2}{n}}}$ for some $c\in (0,1)$, where $\lambda_k$ denote the eigenvalues of the operator $-\nabla \cdot a \nabla$. 
\end{proof}

\subsection{Instability of unique continuation up to the boundary}
\label{sec:UCP}
We study the instability of unique continuation up to the boundary. It is well known (see for instance \cite[Theorem 1.9]{AlessandriniRondiRossetVessella}) that for coefficients and domains having certain minimal regularity one has logarithmic stability estimates for this problem, i.e.\ given a relatively open subset $\Gamma \subset \partial\Omega$ there exists $\mu>0$ such that for solutions $u$ of an elliptic equation satisfying $\norm{u}_{H^1(\Omega)} \leq 1$ one has  
\begin{align*}
\|u\|_{L^2(\Omega)} \leq \omega(\|u\|_{H^{\frac{1}{2}}(\Gamma)} + \|\p_{\nu} u\|_{H^{-\frac{1}{2}}(\Gamma)}),
\end{align*}
where $\omega$ is a modulus of continuity satisfying $\omega(t) \lesssim \abs{\log\,t}^{-\mu}$ for some $\mu > 0$.

We show that this type of stability is optimal and that the optimality, in the sense that no modulus of continuity can be bettern than logarithmic, can be proved also in the low regularity setting (where the UCP in general fails when $n \geq 3$). More precisely, we consider the following set-up:
Let $\Omega \subset \R^n$ with $n\geq 2$ be a bounded open set with Lipschitz boundary. Let $\Gamma \subset \partial \Omega$ be a (relatively) open set with $\p \Omega \setminus \overline{\Gamma} \neq \emptyset$. We are interested in the continuity properties of the inverse of the mapping
\begin{align*}
L^2(\Omega) \ni u \mapsto  (u|_{\Gamma}, \p_{\nu}^L u|_{\Gamma}) \in H^{\frac{1}{2}}(\Gamma) \times H^{-\frac{1}{2}}(\Gamma),
\end{align*}
where $u$ is a solution to the equation
\begin{align}
\label{eq:ell_eq}
Lu := -\p_i a^{ij} \p_j u + b_j \p_j u + c u & = 0 \text{ in $\Omega$},
\end{align}
with $a^{ij} \in L^{\infty}(\Omega, \R^{n\times n})$ bounded, symmetric and uniformly elliptic, i.e. $0<\lambda |\xi|^2 \leq \xi_i a^{ij } \xi_j \leq \lambda^{-1} |\xi|^2 < \infty$ for some $\lambda \in (0,1]$, $b_j \in L^n(\Omega, \C)$, $c \in L^{n/2}(\Omega, \C)$. The normal derivative $\p_{\nu}^L u|_{\p \Omega}$ is defined as an element of $H^{-1/2}(\p \Omega)$ as in Section \ref{subseq_nd}.

We show that even in the presence of low regularity coefficients, so that no $C^{\infty}$ or analytic smoothing is available, one can prove that no stability estimate can be better than logarithmic:

\begin{theorem}
\label{thm:UCP_instab}
Let $\Omega $, $\Gamma \subset \partial \Omega$,L be as above. There is $c > 0$ so that for $\eps > 0$ small there exists a solution $u$ to \eqref{eq:ell_eq} such that 
\begin{align*}
\|u\|_{H^1(\Omega)} \geq \eps, \qquad \|u|_{\p \Omega}\|_{H^{1}(\p \Omega)} \leq 1, \qquad \|u\|_{H^{\frac{1}{2}}(\Gamma)} + \|\p_{\nu}^L u\|_{H^{-\frac{1}{2}}(\Gamma)} \leq e^{-c \eps^{-\frac{2(n-1)}{n+1}}}.
\end{align*}
Moreover, if $\p \Omega$ is $C^{1,1}$, $a^{ij} \in W^{1,\infty}$, $b^j \in L^{\infty}$, $c \in L^n$, then there is $c > 0$ so that for $\eps > 0$ small there exists a solution $u$ to \eqref{eq:ell_eq} such that 
\begin{align*}
\|u\|_{L^{2}(\Omega)} \geq \eps, \qquad \|u\|_{H^{1}(\Omega)} \leq 1, \qquad \|u\|_{H^{\frac{1}{2}}(\Gamma)} + \|\p_{\nu}^L u\|_{H^{-\frac{1}{2}}(\Gamma)} \leq e^{-c \eps^{-\frac{n-1}{n}}}.
\end{align*}
Thus logarithmic stability is optimal for these unique continuation problems in the sense that no modulus of continuity can be better than logarithmic.
\end{theorem}

\begin{remark}
\label{rmk:instab_low_reg}
Note that for $n\geq 3$ the case of low regularity metrics ($a^{ij}\in C^{0,\alpha}$ with $\alpha \in (0,1)$) is of particular interest, as the unique continuation property fails in this low regularity regime in general. 
\end{remark}

\begin{remark}
\label{rmk:regularity}
We point out that in the first instability estimate in the low regularity setting in Theorem \ref{thm:UCP_instab} the condition that $\|u|_{\partial \Omega}\|_{H^1(\partial \Omega)} \leq 1$ may seem rather restrictive given the low regularity assumptions. In this low regularity context, a condition of the form $\|u|_{\partial \Omega}\|_{H^{\frac{1}{2}}(\partial \Omega)} \leq 1$ or -- since we deal with compactness conditions -- an assumption of the form $\|u|_{\partial \Omega}\|_{H^{\frac{1}{2}+\delta}(\partial \Omega)} \leq 1$, for $\delta>0$, may appear most natural. We stress that since Theorem \ref{thm:UCP_instab} is an \emph{existence result} and since the inclusions $H^1(\partial \Omega) \subset H^{\frac{1}{2}+\delta}(\partial \Omega) \subset H^{\frac{1}{2}}(\partial \Omega)$ for $\delta \in (0,1)$ hold, for the constructed functions $u$ these possibly ``more natural'' bounds are indeed satisfied in our context. 
\end{remark}

In order to obtain the desired result, we argue similarly as in the previous section exploiting higher integrability results. We only give the proof when $n \geq 3$ using the estimate of Meyers \cite{M63} for the necessary small regularity gain. For $n=2$ the proof is the same, except that one uses the Gehring lemma instead.

\begin{lemma} \label{lemma_meyers}
Let $\Omega, L$ be as above. There are $C > 0$ and $p > 2$ so that whenever $u \in H^1(\Omega)$ solves $Lu = 0$ in $\Omega$ and $B(x_0,3r) \subset \Omega$, then 
\[
\norm{\nabla u}_{L^p(B(x_0,r))} \leq C r^{-(\frac{n}{2}-\frac{n}{p})} ( \norm{\nabla u}_{L^2(B(x_0,2r))} + \norm{u}_{L^{2n/(n-2)}(B(x_0,2r))} ).
\]
Moreover, if $x_0 \in \p \Omega$ and $u = 0$ on $\p \Omega \cap B(x_0,3r)$, then 
\[
\norm{\nabla u}_{L^p(B(x_0,r) \cap \Omega)} \leq C r^{-(\frac{n}{2}-\frac{n}{p})} ( \norm{\nabla u}_{L^2(B(x_0,2r) \cap \Omega)} + \norm{u}_{L^{2n/(n-2)}(B(x_0,2r) \cap \Omega)} ).
\]
\end{lemma}
\begin{proof}
The first estimate follows from \cite[Theorem 2]{M63}. The second estimate reduces to the first one after flattening the boundary near $x_0$ by a bi-Lipschitz map and reflecting the solution and coefficients to the other side in a suitable way.
\end{proof}

\begin{figure}
\includegraphics[scale=0.8]{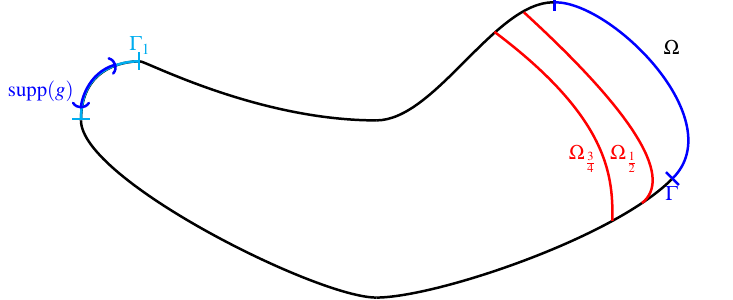}
\caption{A schematic illustration of the ``foliation'' of $\Gamma$ by the domains $\Omega_r$.}
\label{fig:bild1}
\end{figure}

\begin{proof}[Proof of Theorem \ref{thm:UCP_instab}]
Let $\Gamma$ be as above, and choose an open set $\Gamma_1 \subset \p \Omega$ such that $\ol{\Gamma} \cap \ol{\Gamma}_1 = \emptyset$. Further, by possibly making $\Gamma_1$ smaller, without loss of generality, we may assume that by a bi-Lipschitz change of variables the boundary piece $\Gamma_1$ can be flattened to become a ball in $\mR^{n-1}$. We assume the solvability of the Dirichlet problem. If this is not the case, we argue by working in the respective quotient spaces.
We consider solutions $u$ of
\begin{align}
\label{eq:div_form}
\begin{split}
L u & = 0 \mbox{ in } \Omega,\\
u & = g \mbox{ on } \partial \Omega,
\end{split}
\end{align}
where $g \in H^{1/2}_{\ol{\Gamma}_1}(\p \Omega) = \{ f \in H^{1/2}(\p \Omega) \,;\, \supp(f) \subset \ol{\Gamma}_1 \}$. Note that for such boundary values $g$ one has $u|_{\Gamma} = 0$, and the instability of unique continuation can be studied by analyzing the operator 
\[
T: H^{1/2}_{\ol{\Gamma}_1}(\p \Omega) \to H^{-1/2}(\Gamma), \ \ g \mapsto \p_{\nu}^L u|_{\Gamma}.
\] 
We consider a ``foliation'' (see Figure \ref{fig:bild1}) of $\Gamma$ by Lipschitz domains $\Omega_r \subset \Omega$ with $r\in [0,1]$ such that
\begin{itemize}
\item $\Omega_1 = \Omega$,
\item $\Omega_0 = \tilde{\Omega}$, where $\tilde{\Omega} \subset \Omega$ is a bounded, very thin Lipschitz domain of thickness $\delta_0>0$ (depending on $\|b_j\|_{L^{n}(\Omega)}, \|c\|_{L^{\frac{n}{2}}(\Omega)}$ and $\Omega$) such that $\Gamma \subset \partial \tilde{\Omega}$,
\item $\Omega_s \subset \Omega_r$ if $r>s$,
\item $\dist(\Omega_s,\Omega_0)\leq \delta_0$ for $s\in (0,1)$ and for some small, fixed constant $\delta_0>0$ depending on $\|b_j\|_{L^{n}(\Omega)}, \|c\|_{L^{\frac{n}{2}}(\Omega)}$ and $\Omega$,
\item $\Gamma \subset \partial \Omega_r$ for all $r\in [0,1]$, 
\item for all $r\in (0,1)$ we have $\Gamma_1 \cap \partial \Omega_r = \emptyset$ and $\partial \Omega_r \cap \partial \Omega_s \subset \partial \Omega$ (and thus $u= 0$ on $\partial \Omega_r \cap \partial \Omega_s$),
\item for $s\neq r$ and $s,r\in (0,1)$ we assume that $\dist(\partial^0 \Omega_s, \partial^0 \Omega_r) =\frac{\delta_0 |r-s|}{10}$ where $\partial^0 \Omega_t$ for $t\in \{s,r\}$ denotes the part of $\partial \Omega_t$ which is not contained in $\partial \Omega$.  
\end{itemize}
The domains $\Omega_r$, $r\in (0,1)$, are chosen to be uniform in the sense that for $r\in (0,1)$ it holds that the sets $\partial \Omega_r \cap \Omega$ are uniform bi-Lipschitz deformations of a suitable neighbourhood of $\Gamma$.
Here the parameter $\delta_0>0$ is chosen such that the Dirichlet problem is always solvable in the domains $\Omega_r$ with $r\in (0,1)$ (for $s=1$ this solvability is either assumed or one works with associated quotient spaces). If the domains are sufficiently thin, solvability of the Dirichlet problem is guaranteed since the constants in the Poincar\'e and Sobolev inequalities become very small and allow one to absorb all lower order terms in an existence proof (through Lax-Milgram, for instance). Since the constant $\delta_0>0$ is fixed throughout the argument, we do not track its dependence in the sequel.

For $s>r$ we then consider the maps
\begin{align*}
i_{s,r}: H^{\frac{1}{2}}(\partial \Omega_s) \rightarrow H^{\frac{1}{2}}(\partial \Omega_r), \ u|_{\partial \Omega_s} \mapsto u|_{\partial \Omega_r},
\end{align*}
where $u$ solves $Lu = 0$ in $\Omega_s$.

\bigskip

We wish to estimate the singular values of $i_{s,r}$. Using Lemma \ref{lemma_meyers} together with a suitable cover of $\ol{\Omega}_r$ by small balls, Sobolev embedding, and the trace theorem, there are $C > 0$ and $p > 2$ so that 
  \begin{align*}
\|u\|_{W^{1,p}(\Omega_r)} \leq \frac{C}{|r-s|^{\delta}} \|u\|_{H^{1/2}(\p \Omega_s)},
\end{align*}
with $\delta= \frac{n}{2}-\frac{n}{p}$. We combine this with a trace estimate and the Sobolev embedding
  \[ W^{1,p}(\Omega_r) \to B^{1-\frac1p}_{pp}(\partial \Omega_r) \hookrightarrow
  H^{\frac{1}{2}+\delta} (\partial \Omega_r). \]
The trace estimate is given e.g.\ in \cite[Theorem 3.1]{JerisonKenig} or \cite[Chapter VI]{Stein},
and the Sobolev embedding follows from the corresponding embedding in $\mR^{n-1}$ after flattening the boundary locally by bi-Lipschitz maps. 
Thus, as by \cite[Theorem 20.6]{Triebel_fractals} the compact embedding $H^{\frac{1}{2}+\delta} (\partial \Omega_r) \hookrightarrow H^{\frac{1}{2}}(\partial \Omega_r)$ has singular values estimated by
\begin{align*}
\sigma_k(id_{H^{\frac{1}{2}+\delta}(\partial \Omega_r) \rightarrow H^{\frac{1}{2}}(\partial \Omega_r)}) \leq C k^{-\frac{\delta}{n-1}},
\end{align*}
we obtain for the singular values of the mapping $i_{s,r}$ the estimate
\begin{align*}
\sigma_k(i_{s,r}) \leq \frac{C}{|s-r|^{\delta}} k^{-\frac{\delta}{n-1}}.
\end{align*}
The constant $C>0$ can be chosen to be uniform over $r, s \in [0,1]$ with $r < s$ by the uniform choice of the domain geometry.

We now concatenate the mappings $i_{s,r}$ and consider the map
\begin{align*}
T:= T_0 \circ i_{\frac{1}{N}, 0}\circ i_{\frac{2}{N},\frac{1}{N}} \circ \cdots \circ i_{1-\frac{1}{N}, 1- \frac{2}{N}} \circ i_{1,1-\frac{1}{N}},
\end{align*}
where $T_0$ is the bounded map 
\[
T_0: H^{1/2}(\p \Omega_0) \to H^{-1/2}(\Gamma), \ \ u|_{\p \Omega_0} \mapsto \p_{\nu}^L u|_{\Gamma}.
\]
Arguing as in the proof of Theorem \ref{thm:entropy_backward_heat}, using the properties of the singular values under concatenation of mappings, we obtain that
\begin{align} \label{sigmakt_estimate}
\sigma_k(T) \leq C^N \left(N^{\frac{n}{n-1}} k^{-\frac{1}{n-1}} \right)^{\delta N},
\end{align}
which is roughly optimized at $N(k) = \rho k^{\frac{1}{n}}$. With $\rho>0$ chosen appropriately depending on $C, \delta, n$, this yields
\begin{align*}
\sigma_k(T) \leq e^{-c k^{\frac{1}{n}}},
\end{align*}
for some constant $c > 0$ depending only on $C, \delta, n$.

Now applying Theorem \ref{thm_lowreg_instability_sobolev} to $T$ with $X = H^{1/2}_{\ol{\Gamma}_1}(\p \Omega)$ and $X_1 = H^{1}_{\ol{\Gamma}_1}(\p \Omega)$, so that $e_k(i: X_1 \to X) \sim k^{-\frac{1}{2(n-1)}}$ by \cite[Theorem 20.6]{Triebel_fractals}, Theorem \ref{thm_lowreg_instability_sobolev} implies that for $\eps > 0$ small there is a solution $u \in H^1(\Omega)$ with 
\[
\norm{u|_{\p \Omega}}_{H^{1/2}(\p \Omega)} \geq \eps, \quad \norm{u|_{\p \Omega}}_{H^{1}(\p \Omega)} \leq 1, \quad \norm{u}_{H^{1/2}(\Gamma)} + \norm{\p_{\nu}^L u}_{H^{-1/2}(\Gamma)} \leq e^{-c \eps^{-\frac{2(n-1)}{n+1}}}.
\]
Using the trace estimate $\norm{u|_{\p \Omega}}_{H^{1/2}(\p \Omega)}  \leq C \norm{u}_{H^1(\Omega)}$ yields the first estimate in the theorem.

We now show that assuming a bit more regularity, one can change the spaces for $u$ and prove the second estimate in the theorem. We will use the method in Proposition \ref{lemma_ref_basis}. Let $-\Delta_{\Gamma_1}$ be the Laplacian on $\Gamma_1 \subset \p \Omega$, and let $(\varphi_j)_{j=1}^{\infty} \subset H^1_0(\Gamma_1)$ be an orthonormal basis of $L^2(\Gamma_1)$ consisting of Dirichlet eigenfunctions of $-\Delta_{\Gamma_1}$. We recall that $\Gamma_1$ was chosen such that there is a bi-Lipschitz change of coordinates flattening it so that $\Gamma_1$ becomes a ball in $\mR^{n-1}$. Under this change of coordinates $-\D_{\Gamma_1}$ becomes an elliptic operator with $L^{\infty}$ coefficients. Hence, by an application of Theorem \ref{thm_weyl2}, the corresponding eigenvalues satisfy the Weyl asymptotics $\lambda_j \sim j^{\frac{2}{n-1}}$.
For $\abs{s} \leq 1$, we will use the equivalent Sobolev norms 
\[
\norm{\sum a_j \varphi_j}_{H^s(\p \Omega)} = \left( \sum_{j=1}^{\infty} j^{\frac{2s}{n-1}} \abs{a_j}^2 \right)^{1/2}.
\]
Consider the finite dimensional space (with $N$ to be determined later) 
\[
W = \left\{ \sum_{j=1}^N a_j \varphi_j \,;\, a_j \in \mC \right\}.
\]
For any $f = \sum_{j=1}^N a_j \varphi_j \in W$, we have 
\begin{equation} \label{g_sobolev_change}
\norm{f}_{H^{1/2}(\p \Omega)}^2 = \sum_{j=1}^N j^{\frac{1}{n-1}} \abs{a_j}^2 \leq N^{\frac{2}{n-1}} \norm{f}_{H^{-1/2}(\p \Omega)}^2.
\end{equation}
By \eqref{f_w_property} and \eqref{sigmakt_estimate}, there exists $g \in W \setminus \{0\}$ so that 
\[
\norm{Tg}_{H^{-1/2}(\Gamma)} \leq \sigma_N(T) \norm{g}_{H^{1/2}(\p \Omega)} \leq e^{-c N^{\frac{1}{n}}} \norm{g}_{H^{1/2}(\p \Omega)}.
\]

Let $u \in H^1(\Omega)$ solve $Lu = 0$ in $\Omega$ with $u|_{\p \Omega} = g$. We multiply $u$ by a constant so that $\norm{u}_{L^2(\Omega)} = \eps$. Now $\norm{u}_{H^1(\Omega)} \leq C \norm{g}_{H^{1/2}(\p \Omega)}$, and by a duality argument (using the additional regularity of $\p \Omega$ and the coefficients) one also has $\norm{g}_{H^{-1/2}(\p \Omega)} \leq C \norm{u}_{L^2(\Omega)}$. Combining these facts with \eqref{g_sobolev_change}, we obtain that 
\[
\norm{u}_{H^1(\Omega)} \leq C N^{\frac{1}{n-1}} \norm{g}_{H^{-1/2}(\p \Omega)} \leq C' N^{\frac{1}{n-1}} \eps.
\]
Now choose $N = N(\eps) = \rho \eps^{-(n-1)}$ for suitable small $\rho > 0$, so that $C' N^{\frac{1}{n-1}} \eps = 1$. Then 
\[
\norm{u}_{L^2(\Omega)} = \eps, \quad \norm{u}_{H^1(\Omega)} \leq 1, \quad \norm{Tg}_{H^{-1/2}(\Gamma)} \leq e^{-c N^{\frac{1}{n}}} \norm{g}_{H^{1/2}(\p \Omega)} \leq C e^{-c' \eps^{-\frac{n-1}{n}}}
\]
by the trace estimate $\norm{g}_{H^{1/2}(\p \Omega)} \leq C \norm{u}_{H^1(\Omega)} \leq C$. This proves the second estimate in the theorem.
\end{proof}

\subsection{Exponential cost in approximate controllability for the heat equation}
\label{sec:control}

As an example of how our results can also be applied in the context of problems from control theory, we rederive lower bounds on the cost of approximate controllability for the variable coefficient heat equation at low regularity. For constant coefficients in the principal symbol and $L^{\infty}$ potentials this was first treated by \cite{FernandezCaraZuazua} and later revisited in \cite{Phung} where the authors proved \emph{upper} bounds on the cost of controllability. While proving upper bounds requires ``hard'' arguments such as Carleman estimates, we use our strategy from the previous sections to obtain ``soft'' arguments for the \emph{lower} bounds on the cost of controllability -- even at rather low regularities for the coefficients and with essentially optimal dependences (modulo the exponents). To this end, for $\omega \Subset \Omega $ and $\Omega \subset \R^n$ open, bounded $C^{1,\alpha}$, $\alpha \in (0,1)$, domains, consider the equation
\begin{align}
\label{eq:heat_control}
\begin{split}
\p_t u - \p_i a^{ij}\p_j u + b^1_j \p_j u + \p_j (b^2_j u) + cu & = f \chi_{\omega} \mbox{ in } \Omega \times (0,T),\\
u & = 0 \mbox{ on } \partial \Omega \times (0,T),\\
u(\cdot, 0) &= u_0 \mbox{ in } \Omega,
\end{split}
\end{align}
where 
\begin{itemize}
\item[(i)] $\chi_{\omega}$ denotes the characteristic function of $\omega$, 
\item[(ii)] $f\in L^2(\omega \times (0,T))$, $u_0 \in H^{1}_0(\Omega)$, 
\item[(iii)] $a^{ij} \in (C^{0,1}_x \cap C^{0,\frac{1}{2}}_t)(\Omega \times (0,T))$ is a uniformly elliptic matrix with $\lambda |\xi|^2 \leq a^{ij} \xi_i \xi_j \leq \lambda^{-1} |\xi|^2$ for some $\lambda \in (0,1)$, 
\item[(iv)] $b_j^1 , b_j^2 \in C^0((0,T), L^{n+2}(\Omega)) $, $c\in C^0((0,T), L^{\frac{n}{2}}(\Omega ))$ with $$\|b_j^{(\ell)}\|_{C^0((0,T), L^{n+2}(\Omega))} + \|c\|_{C^0( (0,T), L^{\frac{n}{2}}(\Omega ))} \leq \mu \ll 1$$ for $\ell \in \{1,2\}$. If $n=2$ we further assume that $c \in C^0( (0,T), L^{p}(\Omega ))$ for some $p>1$.
\end{itemize}
Under these conditions the Cauchy problem \eqref{eq:heat_control} is well-posed. Here and in the sequel it is always interpreted in its weak form, i.e. $u$ is a solution to \eqref{eq:heat_control} if the following identity holds for all $v\in C^1((0,1), H^1(\Omega))$
\begin{align*}
&-(u,\p_t v)_{L^2(\Omega \times (0,T))}
+ (a^{ij}\p_i u, \p_j v)_{L^2(\Omega \times (0,T))} + (b_j^1 \p_j u, v)_{L^2(\Omega \times (0,T))} \\
& \qquad - (b_j^2 u, \p_j v)_{L^2(\Omega \times (0,T))} + (cu, v)_{L^2(\Omega \times (0,T))}\\
&  = (f \chi_{\omega},v)_{L^2(\Omega \times (0,T))} - (u(T),v(T))_{L^2(\Omega)}+(u(0),v(0))_{L^2(\Omega)}.
\end{align*}
The coefficients are chosen sufficiently regular so that the adjoint equation satisfies the unique continuation principle (see \cite{KochTataru} for weaker conditions guaranteeing this). By standard (duality) arguments from control theory one thus obtains that the equation \eqref{eq:heat_control} is approximately controllable, i.e. for every function $u_d \in H^1_0(\Omega)$ and each $\eps>0$ there exists a control $f \in L^2(\Omega \times (0,T))$ such that the associated solution $u$ of \eqref{eq:heat_control} satisfies
\begin{align*}
\|u(\cdot, T)- u_d\|_{L^2(\Omega)} \leq \eps.
\end{align*}
A more precise question then deals with the \emph{cost of control} which estimates the size of the norm of $f$: As the image at time $T>0$ of $H^1_0(\Omega)$ under the evolution of the equation \eqref{eq:heat_control} is in general \emph{not} the whole space $H^1_0(\Omega)$, this cost has to diverge as $\eps \rightarrow 0$ in general. The \emph{cost of control} provides bounds on this in terms of the size of $\eps$. For the case that $a^{ij}=\delta^{ij}$ and $b_j^1 = b_j^2 =0$ it was proved in \cite{FernandezCaraZuazua, Phung} that there exists $f \in L^2(\Omega \times (0,T))$ such that
\begin{align}
\label{eq:Phung}
\|f\|_{L^2(\Omega \times (0,T))} \leq c\exp\left( C \frac{\|u_d\|_{H^1_0(\Omega)}}{\eps}\right) \|u_d\|_{L^2(\Omega)}.
\end{align} 
Using the methods introduced in the previous sections, we complement \eqref{eq:Phung} with lower bounds (in different function spaces). More precisely, we show that also in the case of low regularity coefficients as in \eqref{eq:heat_control} the $\eps$-dependence in the cost of controllability must be exponential:

\begin{theorem}
\label{thm:controllability_heat_cost}
Let $\omega, \Omega$ and $L$ be as above and let $\delta \in (0, \min\{\frac{1}{2},\delta_0\})$, where $\delta_0>0$ denotes the exponent from Lemma \ref{lem:Sneiberg}. Assume that for any $\eps>0$ and any $u_d \in H^{\delta}(\Omega)$ with $\|u_d\|_{H^{\delta}(\Omega)}=1$ there exists a solution $u$ to \eqref{eq:heat_control} such that
\begin{align*}
\|u(\cdot, T)-u_d\|_{L^2(\Omega)} \leq \eps,
\end{align*}
and
\begin{align*}
\|f\|_{L^2(\Omega\times (0,T))} \leq M(\eps) \|u_d\|_{L^2(\Omega)}.
\end{align*}
Then there is a constant $C>0$ such that $M(\eps)\geq \exp(C \eps^{-\frac{1}{\delta}\frac{n}{n+2}})$.
\end{theorem}

The instability of approximate controllability will be deduced as a consequence of the instability of an associated unique continuation problem. We thus next formulate this result:

\begin{proposition}
\label{prop:UCP_heat_unstable}
Let $\Omega, \omega$ be as above and let $L':= -\p_t - \p_i a^{ij}\p_j - \p_j b_j^1 - b_j^2 \p_j + c$ be the (formal) adjoint operator associated with the operator in \eqref{eq:heat_control}. Let further $\delta \in (0, \min\{\frac{1}{2},\delta_0\})$, where $\delta_0>0$ denotes the exponent from Lemma \ref{lem:Sneiberg}. For $\varphi_T\in L^2(\Omega)$ consider the equation
\begin{align}
\label{eq:heat_adjoint_control}
\begin{split}
L' \varphi &= 0 \mbox{ in } \Omega \times (0,T),\\
 \varphi & = 0 \mbox{ on } \partial \Omega \times (0,T),\\
\varphi(\cdot, T) &= \varphi_T \mbox{ in } \Omega.
\end{split}
\end{align}
Let $\delta_0>0$ be the constant from Lemma \ref{lem:Sneiberg}.
Then, for each $\eps>0$ there exists a function $\varphi_T^{\eps} \in L^2(\Omega)$ and a solution $\varphi^{\eps}$ to \eqref{eq:heat_adjoint_control} with data $\varphi_T^{\eps}$ such that for any $\delta \in (0,\delta_0)$
\begin{align}
\label{eq:UCP_heat}
\|\varphi_T^{\eps}\|_{L^{2}(\Omega)} = 1, \ \|\varphi_T^{\eps}\|_{H^{-\delta}(\Omega)} = \eps, \ \|\varphi^{\eps}\|_{L^2(\omega \times (0,T))} \leq C \exp(-C\eps^{-\frac{1}{\delta} \frac{n}{n+2}}).
\end{align}
\end{proposition}

Before turning to the proof of Proposition \ref{prop:UCP_heat_unstable}, as an auxiliary tool, we provide Davies-Gaffney type Gaussian estimates for solutions to the heat equation:

\begin{lemma}
\label{lem:Gaussian_ests}
Let $\Omega, \omega$ be as above and
let $u$ be a solution to \eqref{eq:heat_control} with initial data $u_0$, where $L$ satisfies the conditions from above. Assume that $\supp(u_0)\cap \overline{\omega} = \emptyset$. Let $d:= \dist(\supp(u_0),\omega)>0$. Then, there exist constants $C,c>0$ such that
\begin{align*}
\|u(t)\|_{L^2(\omega )} \leq C e^{- c\frac{d^2}{t}} \|u_0\|_{L^2(\Omega)}.
\end{align*}
\end{lemma}

In the proof of Proposition \ref{prop:UCP_heat_unstable}, we will reverse time and will apply this to solutions of \eqref{eq:heat_adjoint_control}.

\begin{proof}[Proof of Lemma \ref{lem:Gaussian_ests}]
The proof follows from exponentially weighted estimates.
Let $\psi: \Omega \rightarrow \R$ be a Lipschitz function which is still to be determined below. Then,
\begin{eqnarray*}
&&\frac{d}{dt}\left( \frac{1}{2} \int\limits_{\Omega} e^{2\psi} u^2 dx \right)
= \int\limits_{\Omega} e^{2\psi} u \p_t u dx
\\
&&= - \int\limits_{\Omega} \p_i (e^{2\psi} u) a^{ij} \p_j u dx
- \int\limits_{\Omega} b_j^1 \p_j u e^{2\psi} u dx
 + \int\limits_{\Omega} b_j^2 u \p_j (e^{2 \psi} u) dx 
+ \int\limits_{\Omega} c u^2 e^{2 \psi} dx
\\
&&= -\int\limits_{\Omega} \p_i(e^{\psi} u) a^{ij} \p_j (e^{\psi} u) dx  + \int\limits_{\Omega} \p_i \psi a^{ij}\p_j \psi (e^{\psi} u)^2 dx \\
&& \quad - \int\limits_{\Omega} b_j^1 \p_j (u e^{\psi}) (e^{\psi} u) dx + \int\limits_{\Omega} b_j^1 (u e^{\psi}) (\p_j \psi) (e^{\psi} u) dx\\
&& \quad + \int\limits_{\Omega} b_j^2 (u e^{\psi}) \p_j (e^{\psi} u) dx + \int\limits_{\Omega} b_j^2 (e^{\psi} u) \p_j \psi \p_j (e^{\psi} u) dx + \int\limits_{\Omega} c (e^{\psi}u)^2 dx.
\end{eqnarray*}
Using the ellipticity of $a^{ij}$, Poincar\'e's inequality and Sobolev embedding, this leads to the differential inequality
\begin{align*}
\frac{d}{dt} \left( \frac{1}{2} \| e^{\psi} u\|_{L^2(\Omega)}^2 \right)
&\leq - \lambda \|\nabla (e^{\psi} u)\|_{L^2(\Omega)}^2
 + \lambda^{-1} \|e^{\psi} u \nabla \psi \|_{L^2(\Omega)}^2\\
 & \quad 
+ (\|b_j^1\|_{L^{n+2}(\Omega)} + \|b_j^2\|_{L^{n+2}(\Omega)})(\epsilon \|\nabla (e^{\psi}u)\|_{L^2(\Omega)}^2 + C \epsilon^{-1}\|u e^{\psi}\|_{L^2(\Omega)}^2)\\
& \quad + \|c\|_{L^{\frac{n}{2}}(\Omega)} \|\nabla( e^{\psi}u)\|_{L^{2}(\Omega)}.
\end{align*}
Choosing $0<\epsilon = \frac{\lambda}{10}$ and using the smallness conditions on the norms of the potentials imposed in condition (iv) above, we hence arrive at 
\begin{align*}
\frac{d}{dt} \frac{1}{2} \| e^{\psi} u\|_{L^2(\Omega)}^2
&\leq - \frac{\lambda}{10} \|\nabla (e^{\psi} u)\|_{L^2(\Omega)}^2
 + C\lambda^{-1}( 1+ \|\nabla \psi\|_{L^{\infty}(\Omega)}^2)\|e^{\psi} u  \|_{L^2(\Omega)}^2,
\end{align*}
where $C>0$ depends on the norms of the coefficients.
Invoking Gronwall's lemma, we then obtain that for some constant $C>0$ (depending on $\lambda$ and the coefficient bounds)
\begin{align}
\label{eq:Gauss1}
\|e^{\psi} u\|_{L^2(\Omega)} \leq e^{C (\|\nabla \psi\|_{L^{\infty}(\Omega)}^2 +1)t} \|e^{\psi} u_0\|_{L^2(\Omega)}.
\end{align}
Next, we fix $\psi$ to be Lipschitz continuous in such a way that $\psi = L $ in $\omega $, $\psi = 0$ on $\supp(u_0)=0$ and such that $\|\nabla \psi\|_{L^{\infty}(\Omega)} \leq \frac{L}{d}$. As a consequence, from \eqref{eq:Gauss1} and the bounds on $\psi$ we obtain
\begin{align}
\label{eq:Gauss2}
\| u\|_{L^2(\omega)} \leq e^{C \frac{L^2}{d^2} t - L} \| u_0\|_{L^2(\Omega)}.
\end{align}
Optimizing this in the choice of $L>0$ yields $L = \tilde{C} \frac{d^2}{t}$ which concludes the argument.
\end{proof}

With Lemma \ref{lem:Gaussian_ests} in hand and recalling the regularity estimate \eqref{lem:Sneiberg}, we proceed to the proof of Proposition \ref{prop:UCP_heat_unstable}:

\begin{proof}[Proof of Proposition \ref{prop:UCP_heat_unstable}]
\emph{Step 1: A regularity estimate.} We first note that by duality the estimate in \eqref{lem:Sneiberg} yields that for any $t\in (0,T)$ and any $\delta \in (0,\delta_0)$ (where $\delta_0>0$ denotes the constant from Lemma \ref{lem:Sneiberg})
\begin{align*}
\|\varphi(T-t)\|_{L^2(\Omega)} \leq C t^{-\delta/2} \|\varphi_T\|_{H^{-\delta}(\Omega)}.
\end{align*}
Indeed, given $s\in (0,T)$, we consider a solution $\varphi$ of \eqref{eq:heat_adjoint_control} and a solution $u$ to \eqref{eq:heat_control}, the latter with initial data at time $T-s$. Then, by definition of solutions, we obtain that
\begin{align*}
(u_{T-s}, \varphi(T-s))_{L^2(\Omega)} = (u(T), \varphi_T)_{L^2(\Omega)}.
\end{align*}
As a consequence,
\begin{align*}
\|\varphi(T-s)\|_{L^2(\Omega)} 
&\leq \sup\limits_{\|u_{T-s}\|_{L^2(\Omega)} = 1} (u_{T-s}, \varphi(T-s))_{L^2(\Omega)}
= \sup\limits_{\|u_{T-s}\|_{L^2(\Omega)} = 1} (u(T), \varphi_{T})_{L^2(\Omega)}\\
& \leq \sup\limits_{\|u_{T-s}\|_{L^2(\Omega)} = 1} \|u(T)\|_{H^{\delta}(\Omega)} \|\varphi_T\|_{H^{-\delta}(\Omega)}\\
& \leq \sup\limits_{\|u_{T-s}\|_{L^2(\Omega)} = 1} C s^{-\delta/2} \|u_{T-s}\|_{L^2(\Omega)} \|\varphi_T\|_{H^{-\delta}(\Omega)}
\leq C s^{-\delta/2} \|\varphi_T\|_{H^{-\delta}(\Omega)}.
\end{align*}
Here we used that $\delta \in (0,\frac{1}{2})$ and the Sneiberg type result of Lemma \ref{lem:Sneiberg}.\\

\emph{Step 2: Gaussian bounds.}
Next, we claim that for any $t\in (0,T)$ there exists a constant $\tilde{C}>0$ such that
\begin{align}
\label{eq:bound_Gauss_aux}
\|\varphi\|_{L^2(\omega \times (T-t,T))} \leq \tilde{C} e^{- \frac{\tilde{C}}{t}} \|\varphi_{T}\|_{H^{-\delta}(\Omega)}.
\end{align}
Indeed, by the Gaussian bounds from Lemma \ref{lem:Gaussian_ests} (together with a reversal of time) and by Step 1, for any $s\in (0,T)$ we have that
\begin{align*}
\|\varphi(T-s)\|_{L^2(\omega)} \leq C e^{- \frac{C}{s}} \|\varphi(T-s/2)\|_{L^2(\Omega)} \leq C e^{-\frac{C}{s}} s^{-\delta/2} \|\varphi_T\|_{H^{-\delta}(\Omega)}.
\end{align*}
Squaring and integrating this from $s\in (0,t)$, we obtain
the desired bound \eqref{eq:bound_Gauss_aux}.\\

\emph{Step 3: Conclusion.}
We next consider the two maps
\begin{align*}
&i_1(t): H^{-\delta}(\Omega ) \ni \varphi_T \mapsto \varphi|_{\omega \times (T-t,T)} \in L^2(\omega \times (T-t,T)),\\
&i_2(t):  H^{-\delta}(\Omega ) \ni \varphi_T \mapsto \varphi|_{\omega \times (0,T-t)} \in L^2(\omega \times (0,T-t)).
\end{align*}
Now, on the one hand, by Lemma \ref{lem:Gaussian_ests} and by Theorem \ref{thm:entropy_backward_heat} (with a time-step larger or equal to $t$ and not a time step of the order one) and an integration argument as in Step 2, we obtain
\begin{align}
\label{eq:i2}
\sigma_k(i_2(t)) \leq e^{-c k^{\frac{2}{n+2}} t^{\frac{n}{n+2}}}.
\end{align}
On the other hand, Step 2 implies that 
\begin{align}
\label{eq:i1}
\sigma_1(i_1(t)) \leq e^{-\frac{C}{t}}.
\end{align}
Since the the map $i: L^{2}(\Omega) \ni \varphi_T \mapsto \varphi|_{\omega \times (0,T)} \in L^2(\omega \times (0,T))$ can we written as $i = i_2(t) + i_1(t)$, for the singular values of $i$ we infer that
\begin{align*}
\sigma_k(i) \leq \sigma_1(i_1(t)) + \sigma_k(i_2(t))
\leq Ce^{-\frac{C}{t}} + e^{-c k^{\frac{2}{n+2}} t^{\frac{n}{n+2}}}.
\end{align*}
Optimizing this in $t$, we consider $t= \tilde{c} k^{- \frac{1}{n+1}}$ for some $\tilde{c}>0$ and hence arrive at
\begin{align*}
\sigma_k(i) \leq e^{-c k^{\frac{1}{n+1}}}.
\end{align*}
Invoking Theorem \ref{thm_lowreg_instability_sobolev} with $X=L^2(\Omega)$, $X_1 = H^{-\delta}(\Omega)$ we thus deduce that
\begin{align*}
\omega(t) \geq C |\log(t)|^{- \delta \frac{n+2}{n}}.
\end{align*}
This concludes the proof.
\end{proof}

With Proposition \ref{prop:UCP_heat_unstable} in hand, the proof of Theorem \ref{thm:controllability_heat_cost} reduces to a duality argument.

\begin{proof}[Proof of Theorem \ref{thm:controllability_heat_cost}]
We prove that if there was a cost of control that was better than exponential, then also the associated unique continuation estimate \eqref{eq:UCP_heat} would have to be better than logarithmic (which by \eqref{eq:UCP_heat} cannot be the case). The following duality argument is analogous to the strategy outlined in \cite{Phung}:
Considering a solution $u$ to \eqref{eq:heat_control} and a solution $\varphi$ to \eqref{eq:heat_adjoint_control}, integrating by parts and using the Dirichlet boundary conditions, we obtain  \eqref{eq:heat_adjoint_control}
\begin{align*}
\int\limits_{0}^T \int\limits_{\omega} f \varphi dx dt
= \int\limits_{\Omega} u(x, T) \varphi_T dx
= \int\limits_{\Omega} (u(x,T) - u_d(x)) \varphi_T(x) dx + \int\limits_{\Omega} u_d(x) \varphi_T(x)dx.
\end{align*}
Rearranging and assuming the approximation bounds stated in Theorem \ref{thm:controllability_heat_cost} then yields
\begin{align*}
\left| \int\limits_{\Omega} u_d \varphi_T dx \right|
&\leq \|u(\cdot, T)-u_d\|_{L^2(\Omega)} \|\varphi_T\|_{L^2(\Omega)} + \|f\|_{L^2(\omega \times (0,T))} \|\varphi\|_{L^2(\omega \times (0,T))}\\
&\leq \eps \|u_d\|_{H^{\delta}(\Omega)} \|\varphi_T\|_{L^2(\Omega)}+ M(\eps) \|u_d\|_{L^2(\Omega)}\|\varphi\|_{L^2(\omega \times (0,T))}.
\end{align*}
By the definition of the $H^{-\delta}(\Omega)$ norm through duality (where we use that $\delta \in (0,1/2)$), we hence infer that 
\begin{align}
\label{eq:UCP_new}
\|\varphi_T\|_{H^{-\delta}(\Omega)} \leq \eps\|\varphi_T\|_{L^{2}(\Omega)}  + M(\eps)\|\varphi\|_{L^2(\omega \times (0,T))}.
\end{align}
Inserting the functions $\varphi^{\tilde{\eps}}$ with $\tilde{\eps}=2\eps$ from Proposition \ref{prop:UCP_heat_unstable} into \eqref{eq:UCP_new} implies 
\begin{align}
\label{eq:UCP_new1}
2 \eps \leq \eps + M(2 \eps)C \exp(-C (2\eps)^{-\frac{1}{\delta}\frac{n}{n+2}}).
\end{align}
As a consequence, $M(2 \eps)$ has to be of the order of $\exp(C(2\eps)^{-\frac{1}{\delta}\frac{n}{n+2}})$ for some $C>0$.
\end{proof}

\subsection{Instability of the Calder\'on problem at low regularity}
\label{sec:Calderon}

In this section, we show that similarly as for the stability estimates in unique continuation, also for the stability properties of the Calder\'on problem, one cannot hope for a gain due to the presence of low regularity metrics and potentials (thus proving Theorem \ref{thm:thm_Calderon} and answering question (Q3) from the introduction for the example of the Calder\'on problem). 

Let $\Omega \subset \mR^n$ be a bounded open set with Lipschitz boundary. As in Section \ref{subseq_nd} we consider the operator $L = - \p_j a^{jk} \p_k + b_j \p_j + \p_j c_j + q_0$ with $(a^{jk}) \in L^{\infty}(\Omega, \R^{n\times n})$ uniformly elliptic, $b_j, c_j \in L^{n}(\Omega)$ and $q_0 \in L^{\frac{n}{2}}(\Omega)$. If $u \in H^1(\Omega)$ is a weak solution of $Lu = 0$, we recall that the normal derivative $\p_{\nu}^L u|_{\p \Omega}$ is defined weakly as an element of $H^{-1/2}(\p \Omega)$ by \eqref{eq:bilinear}. If $L$ is such that zero is not a Dirichlet eigenvalue, we may thus define the Dirichlet-to-Neumann operator
\begin{align}
\label{eq:DtN_Calderon}
\Lambda_{L}: H^{\frac{1}{2}}(\partial \Omega) \rightarrow H^{-\frac{1}{2}}(\partial \Omega), \ f\mapsto \p_{\nu}^L u|_{\partial \Omega},
\end{align}
 where $u$ is a solution to $L u = 0$ in $\Omega$ with $u|_{\partial \Omega} = f$.
 
In the following, we seek to prove that for the Dirichlet-to-Neumann operator the instability result from Theorem \ref{thm:thm_Calderon} holds. In this context, the main novelty of the associated estimates with respect to the instability results from Mandache \cite{Mandache} and Di Cristo-Rondi \cite{DiCristoRondi} is that our mechanism is very robust in that we can treat
\begin{itemize} 
\item potentials $q_0+q_1, q_0+ q_2$ which are \emph{not necessarily compactly supported},
\item coefficients of \emph{scaling critical low regularity}, 
\item and consider the Calder\'on problem in \emph{relatively rough and general domains}.
\end{itemize}

The proof of Theorem \ref{thm:thm_Calderon} will be based on the following abstract result. Recall that $Y'$ is the set of bounded linear functionals on $Y$, and that for $Y$ Hilbert any $A \in B(Y, Y')$ has a formal adjoint $A' \in B(Y, Y')$.

\begin{theorem} \label{thm_lowreg_calderon_abstract}
Let $F: X \to B(Y,Y')$ be a continuous map, where $X$ is a Banach space and $Y$ is a separable Hilbert space. Let $X_1 \subset X$ be a closed subspace so that $i: X_1 \to X$ is compact with $e_k(i) \gtrsim k^{-m}$ for some $m > 0$. Let $K = \{ u \in X \,;\, \norm{u}_{X_1} \leq r \}$ for some $r > 0$. Assume that there is an orthonormal basis $(\varphi_j)_{j=1}^{\infty}$ of $Y$ and constants $C, \rho, \mu > 0$ so that $F(u)$ and $F(u)'$ satisfy 
\begin{equation} \label{fu_decay}
\norm{F(u) \varphi_k}_{Y'}, \ \norm{F(u)' \varphi_k}_{Y'} \leq C e^{-\rho k^{\mu}}, \qquad k \geq 1,
\end{equation}
uniformly over $u \in K$. Then there is $c > 0$ with the following property: for any $\eps > 0$ small enough there are $u_1, u_2$ such that
\begin{equation} \label{instab_restatement1}
\|u_1-u_2\|_{X} \geq \eps, \quad \|u_j\|_{X_1} \leq r, \quad \|F(u_1)-F(u_2)\|_{B(Y,Y')} \leq \exp(-c \eps^{-\frac{\mu}{m(\mu+2)}}).
\end{equation}
In particular, if one has the stability property 
\[
\norm{u_1-u_2}_X \leq \omega(\norm{F(u_1)- F(u_2)}_{B(Y,Y')}), \qquad u_1, u_2 \in K,
\]
then necessarily $\omega(t) \gtrsim \abs{\log\,t}^{-\frac{m(\mu+2)}{\mu}}$ for $t$ small.
\end{theorem}

Let us comment on this abstract result which will be the main ingredient in our discussion of the instability properties of the Calder\'on problem at low regularity. The main point is that there is a fixed orthonormal basis $(\varphi_j)$ that gives the decay \eqref{fu_decay} uniformly over $u \in K$ (cf. also \cite{DiCristoRondi}). As a sufficient condition for \eqref{fu_decay}, it is enough to find a linear operator $A: Y \to Z$ which dominates $F(u)$ and $F(u)'$ in the sense that 
\[
\norm{F(u)f}_{Y'}, \ \norm{F(u)'f}_{Y'} \leq C \norm{Af}_Z
\]
uniformly over $u \in K$, and that the singular values of $A$ satisfy $\sigma_k(A) \leq C e^{-\rho k^{\mu}}$ (one can then take $(\varphi_j)$ as a singular value basis of $A$). In our discussion of the Calder\'on problem below, we will follow this strategy by considering $F(q)= \Lambda_{L+q} - \Lambda_L$ and by taking $A$ to be the operator which maps $u|_{\p \Omega}$ to $u|_{\p \Omega'}$ where $u$ solves $L u = 0$ in $\Omega$ and $\Omega' \Subset  \Omega$. We already saw in the proof of Theorem \ref{thm:UCP_instab} that the singular values of $A$ decay exponentially.

\begin{proof}[Proof of Theorem \ref{thm_lowreg_calderon_abstract}]
Define the norm 
\[
\norm{f}_{\tilde{Y}} = \left( \sum_{j=1}^{\infty} e^{-(\rho/2) j^{\mu}} \abs{(f, \varphi_j)_Y}^2 \right)^{1/2}
\]
and let $\tilde{Y}$ be the completion of $Y$ with respect to this norm. We claim that for any $u \in K$, the map $F(u)$ extends as a bounded operator $\tilde{Y} \to \tilde{Y}'$. In fact, write $a_{jk} = a_{jk}(u) = (F(u) \varphi_k)(\varphi_j)$, so that 
\begin{equation} \label{ajk_decay}
\abs{a_{jk}} \leq C e^{-\rho (\max(j,k))^{\mu}}
\end{equation}
by \eqref{fu_decay}. For any $f, g \in Y$ we have $F(u)f \in Y'$ and 
\begin{align}
\abs{(F(u)f)(g)} &\leq \abs{\sum_{j,k} a_{jk} (f, \varphi_k) (g, \varphi_j)} \leq \left( \sum_{j,k} \abs{a_{jk}}^2 e^{(\rho/2)(j^{\mu} + k^{\mu})} \right)^{1/2} \norm{f}_{\tilde{Y}} \norm{g}_{\tilde{Y}} \notag \\
 &\leq C \norm{f}_{\tilde{Y}} \norm{g}_{\tilde{Y}}. \label{fufg_estimate}
\end{align}
Let $\tilde{\varphi}_j = e^{(\rho/4) j^{\mu}} \varphi_j$, so that $(\tilde{\varphi}_j)$ is an orthonormal basis of $\tilde{Y}$. Then as in \eqref{fufg_estimate} 
\[
\norm{F(u)\tilde{\varphi}_l}_{\tilde{Y}'} \leq \sup_{\norm{g}_{\tilde{Y}} = 1} \abs{\sum_j a_{jl} e^{(\rho/4) l^{\mu}}(g, \varphi_j)} \leq \left( \sum_{j} \abs{a_{jl}}^2 e^{(\rho/2)(j^{\mu} + l^{\mu})} \right)^{1/2} \leq C e^{-(\rho/4) l^{\mu}}
\]
where we used \eqref{ajk_decay}. It follows that 
\[
\norm{F(u)}_{HS(\tilde{Y}, \tilde{Y}')}^2 = \sum_{l=1}^{\infty} \norm{F(u) \tilde{\varphi}_l}_{\tilde{Y}'}^2 \leq C < \infty.
\]
The constants are uniform over $u \in K$.

Now $F(K)$ is contained in a bounded subset of $HS(\tilde{Y}, \tilde{Y}')$, and the embedding $$\iota: HS(\tilde{Y}, \tilde{Y}') \to HS(Y, Y')$$ has singular values satisfying $\sigma_k(\iota) \lesssim C e^{-c k^{\mu/2}}$ by Lemma \ref{lemma_hs_embedding} (we can take $\alpha_j = \beta_j = e^{(\rho/4) j^{\mu}}$). Lemma \ref{lemma_singular_entropy_relation} gives that $e_k(\iota) \lesssim e^{-c k^{\frac{\mu}{2+\mu}}}$. Now the theorem follows from Lemma \ref{lemma_entropy_typical}.
\end{proof}

\begin{proof}[Proof of Theorem \ref{thm:thm_Calderon}]
Since $L$ has discrete spectrum and $0$ is not a Dirichlet eigenvalue of $L$ in $\Omega$, there is $\eps_0 > 0$ so that $0$ is not a Dirichlet eigenvalue of $L + q$ in $\Omega$ whenever $\norm{q}_{L^{n/2}(\Omega)} \leq \eps_0$. Now choose a bounded smooth domain $\Omega' \subset \subset \Omega$, and let 
\[
K := \{ q \in W^{\delta,n/2}_0(\Omega') \,;\, \norm{q}_{W^{\delta,n/2}(\Omega')} \leq \eps_0 \}.
\]
Identifying $q$ with its zero extension to $\Omega$, we have that $K$ is a compact subset of $X := L^{n/2}(\Omega)$.

Write $Y := H^{1/2}(\p \Omega)$, and define the map 
\[
F: X \to B(Y, Y'), \ \ F(q) = \Lambda_{L+q}-\Lambda_L.
\]
In fact $F$ is well defined near the set where $\norm{q}_{L^{n/2}(\Omega)} \leq \eps_0$ which contains $K$, and this is enough for our purposes. The embedding $i: W^{\delta,n/2}_0(\Omega') \to L^{n/2}(\Omega)$ has entropy numbers $e_k(i) \gtrsim k^{-\delta/n}$ by \cite[Theorem 1 in Section 3.3.3]{ET08} (note that the functions $f$ in that theorem are compactly supported). Given $q \in K$ and $f \in Y$, let $u, u_0 \in H^1(\Omega)$ solve $(L + q) u = L u_0 = 0$ in $\Omega$ with $u|_{\p \Omega} = u_0|_{\p \Omega} = f$. Writing $w := u-u_0$, we have 
\[
F(q) f = \p_{\nu}^{L} w|_{\p \Omega}
\]
where $w \in H^1(\Omega)$ solves 
\[
(L+q) w = -q u_0 \text{ in $\Omega$}, \qquad w|_{\p \Omega} = 0.
\]
Elliptic regularity and the fact that $q$ is supported in $\ol{\Omega}'$ give that 

\begin{align*}
\norm{\p_{\nu}^L w}_{H^{-1/2}(\p \Omega)} &\lesssim \norm{w}_{H^1(\Omega)} \lesssim \norm{qu_0}_{H^{-1}(\Omega)} \lesssim \norm{q}_{L^{n/2}(\Omega')} \norm{u_0}_{H^1(\Omega')} \\
 &\lesssim \norm{u_0|_{\p \Omega'}}_{H^{1/2}(\p \Omega')}.
\end{align*}
Writing $Z := H^{1/2}(\p \Omega')$, we thus have 
\[
\norm{F(q)f}_{Y'} \leq C \norm{Af}_{Z}
\]
uniformly over $q \in K$, where $A: Y \to Z$ maps $u_0|_{\p \Omega}$ to $u_0|_{\p \Omega'}$. 
The constant is uniform over $q \in K$. By the argument in Theorem \ref{thm:UCP_instab} there is $c > 0$ so that 
\[
\sigma_k(A) \leq e^{-c k^{1/n}}, \qquad k \geq 1.
\]
Choosing $(\varphi_j)$ to be a singular value basis for $A$ and using that $F(q)$ is self-adjoint, we see that \eqref{fu_decay} holds with $\mu = 1/n$. The result now follows from Theorem \ref{thm_lowreg_calderon_abstract}.
\end{proof}

\begin{remark}[Possible extensions]
\label{rmk:gen_Cald}
Finally, we comment on possible extensions of this result: 
\begin{itemize}
\item \emph{Cauchy data.} Instead of working with the Dirichlet-to-Neumann map a perhaps more natural setting is that of working with Cauchy data
\begin{align*}
\mathcal{C}_q = \left\{ \left( u|_{\partial \Omega}, \p_{\nu}^L u|_{\partial \Omega} \right) \in H^{\frac{1}{2}}(\partial \Omega) \times H^{-\frac{1}{2}}(\partial \Omega): \ u \mbox{ is a solution to } Lu = 0 \mbox{ in } \Omega \right\},
\end{align*}
endowed with the distance
\begin{align*}
\dist(\mathcal{C}_{q_1}, \mathcal{C}_{q_2})
&= \max\left\{ \max\limits_{(f,g)\in \mathcal{C}_{q_1}} \min\limits_{(\tilde{f},\tilde{g})\in \mathcal{C}_{q_2}} \frac{\|(f,g)-(\tilde{f},\tilde{g})\|_{H^{\frac{1}{2}}(\partial \Omega) \times H^{-\frac{1}{2}}(\partial \Omega)}}{\|(f,g)\|_{H^{\frac{1}{2}}(\partial \Omega) \times H^{-\frac{1}{2}}(\partial \Omega)}}, \right.\\
& \quad \left. \max\limits_{(f,g)\in \mathcal{C}_{q_2}} \min\limits_{(\tilde{f},\tilde{g})\in \mathcal{C}_{q_1}} \frac{\|(f,g)-(\tilde{f},\tilde{g})\|_{H^{\frac{1}{2}}(\partial \Omega) \times H^{-\frac{1}{2}}(\partial \Omega)}}{\|(f,g)\|_{H^{\frac{1}{2}}(\partial \Omega) \times H^{-\frac{1}{2}}(\partial \Omega)}} \right\} ,
\end{align*}
where $\|(f,g)\|_{H^{\frac{1}{2}}(\partial \Omega) \times H^{-\frac{1}{2}}(\partial \Omega)} = \|f\|_{H^{\frac{1}{2}}(\partial \Omega)} + \|g\|_{H^{-\frac{1}{2}}(\partial \Omega)}$. In order to avoid technicalities we did not formulate the Calder\'on problem in this setting above. Similar arguments as above could however be used also in the framework involving Cauchy data (dealing additionally with finite dimensional subspaces).
\item \emph{Potentials in lower regularity classes.}
We emphasize that our Schr{\"o}dinger operators also include the setting of potentials $q_0 \in W^{-1,n}(\Omega)$. Indeed, any such potential can be represented in the form $q_0 = F_0 + \p_i F_i$ with $F_0 \in L^{\frac{n}{2}}(\Omega)$ and $F_i \in L^{n}(\Omega)$ for $i\in \{1,\dots,n\}$. In order to prove an instability result for the Calder\'on problem in these spaces, one would have to vary not only the potentials $q_0$ but also the drift terms $b_j, c_j$. Such an instability result would follow using our strategy in a similar way as the result from Theorem \ref{thm:thm_Calderon}.
\end{itemize}
\end{remark}

\section{Instability in the presence of microlocal smoothing}
\label{sec:micro_smooth_1}

In the previous sections we considered forward operators that have \emph{global} smoothing properties, in the sense that they either erase all singularities of any function that they are applied to, or at least provide a certain (possibly small) amount of global regularization. In this section, we will now consider forward operators that are only \emph{microlocally} smoothing at certain points. If $A$ is linear and microlocally smoothing at $(x_0, \xi_0)$, then no singularity of $u$ at $(x_0,\xi_0)$ can be recovered from the knowledge of $Au$. It is known (see e.g.\ \cite[Theorem 4.4]{MSU15}) that in such cases a stability estimate of the type 
\begin{equation} \label{u_hs_general}
\norm{u}_{H^{s_1}} \leq C(\norm{Au}_{H^{s_2}} + \norm{u}_{H^{s_3}})
\end{equation}
cannot hold for any choices of $s_1, s_2, s_3 \in \mR$ for which $s_1 > s_3$. This situation is relevant for many inverse problems, including limited data X-ray tomography \cite{Quinto, Quinto1} or the study of geodesic X-ray transforms \cite{MSU15}.

Let $A$ be linear and microlocally smoothing at $(x_0, \xi_0)$. Then the fact that an estimate like \eqref{u_hs_general} cannot be valid can be proved by testing \eqref{u_hs_general} with wave packets, or ``coherent states", $u$ that concentrate in a conic neighborhood of $(x_0, \xi_0)$ in phase space (for such functions $Au$ is very small because of microlocal smoothing). See e.g.\ \cite[Proposition 5]{SU09} for an argument of this type. From the point of view of the entropy/capacity approach in Section \ref{sec:abstract} this means that not all of $A(K)$ is contained in a compressed space, but there is a large enough set $K_1 \subset K$ so that $A(K_1)$ lies in a compressed space. The set $K_1$ contains functions that concentrate in a conic neighborhood of $(x_0, \xi_0)$, and there are relatively large $\eps$-discrete sets contained in $K_1$ (in fact such sets could be constructed directly from wave packets using almost orthogonality). Replacing $K$ by $K_1$ in the analysis in Section \ref{sec:abstract} then leads to instability.

We can use the remarks in the preceding paragraph to give a more general instability argument that also applies to nonlinear inverse problems. The idea is that finding $\eps$-discrete sets in $K_1$ can be reduced to obtaining lower bounds for entropy/capacity numbers of microlocal cutoffs to a conic neighborhood of $(x_0, \xi_0)$. This in turn reduces to a standard Weyl law for pseudodifferential operators (Theorem \ref{thm_weyl3}).

\begin{theorem}
\label{thm:Weyl_microlocal}
Let $(M,g)$ be a closed smooth $n$-manifold and let $P \in \Psi^0_{\mathrm{cl}}(M)$ be noncharacteristic (i.e.\ have nonvanishing principal symbol) at some $(x_0, \xi_0) \in T^* M \setminus \{0\}$.  For any $s_1 > s_2$, the compact operator $P: H^{s_1}(M) \to H^{s_2}(M)$ satisfies 
\[
\sigma_k(P), \ e_k(P) \sim k^{-\frac{s_1-s_2}{n}}.
\]
\end{theorem}

\begin{proof}
By Lemma \ref{lemma_singular_entropy_relation} it is enough to prove the claim for $\sigma_k(P)$. We seek to reduce the claim to Theorem \ref{thm_weyl3}. To this end, let $J^{\alpha} = (1-\Delta_g)^{\alpha/2}$ be the Bessel potential which induces an isomorphism $H^s(M) \to H^{s-\alpha}(M)$ for any $s$, write $m = s_1-s_2 > 0$, and let $A = J^{s_2} P J^{-s_1} \in \Psi^{-m}_{\mathrm{cl}}(M)$. Then the theorem follows if we can show that any $A \in \Psi^{-m}_{\mathrm{cl}}(M)$ which is noncharacteristic at some $(x_0, \xi_0)$, considered as an operator $A: L^2(M) \to L^2(M)$, has singular values satisfying 
\[
\sigma_k(A) \sim k^{-m/n}.
\]
This, however, follows from Theorem \ref{thm_weyl3}.
\end{proof}

Combining the Weyl law of Theorem \ref{thm:Weyl_microlocal} with Theorem \ref{thm:Mandache} and Lemma \ref{lemma_entropy_typical} we obtain the following abstract instability result. We will later apply it to the Radon transform in $\mR^2$, and hence we will allow noncompact manifolds. For this we recall the localized Sobolev spaces $H^{s}_L(M)$ from Definition \ref{def:loc_Hs}.

\begin{theorem} \label{thm_microlocal_smoothing}
Let $M, N$ be smooth manifolds, let $L \subset M$ and $Q \subset N$ be compact, let $s, t \in \mR$, and let $\delta > 0$. Let $K_r = \{ f \in H^{s+\delta}_L(M) \,;\, \norm{f}_{H^{s+\delta}} \leq r \}$, and let $F$ be a map $K_r \to H^t_Q(N)$. Suppose that $\omega$ is a modulus of continuity such that 
\[
\norm{f_1-f_2}_{H^s(M)} \leq \omega(\norm{F(f_1) - F(f_2)}_{H^t(N)}), \qquad f_1, f_2 \in K_r.
\]
\begin{enumerate}
\item[(a)] 
If there is $P \in \Psi^0_{\mathrm{cl}}(M)$, noncharacteristic at some point of $T^* L^{\mathrm{int}} \setminus \{0\}$, and $r_0 > 0$ so that $P(K_{r_0}) \subset K_r$ and $F \circ P$ maps $K_{r_0}$ into a bounded set of $H^{t+m}_Q(N)$ for any $m > 0$, then for any $\alpha \in (0,1)$ one has $\omega(t) \gtrsim t^{\alpha}$ for $t$ small.
\item[(b)] 
If there is $P \in \Psi^0_{\mathrm{cl}}(M)$, noncharacteristic at some point of $T^* L^{\mathrm{int}}  \setminus \{0\}$, and $r_0 > 0$ so that $P(K_{r_0}) \subset K_r$ and $F \circ P$ maps $K_{r_0}$ into a bounded set of $A^{\sigma,\rho}_Q(N)$ for some $1 \leq \sigma < \infty$ and $\rho > 0$, then $\omega(t) \gtrsim \abs{\log t}^{-\frac{\delta(\sigma \dim(N) +1)}{\dim(M)}}$ for $t$ small.
\end{enumerate}
\end{theorem}
\begin{proof}
We first give the proof under the assumption that $L = M$, so that $M$ is compact. By Theorem \ref{thm:Weyl_microlocal} and \eqref{entropy_capacity_relation}, the capacity numbers of $P: H^{s+\delta}(M) \to H^s(M)$ satisfy $c_k(P) \gtrsim k^{-\delta/n_M}$. Arguing as in the proof of Lemma \ref{lemma_entropy_typical}, we see that for $\eps > 0$ small there is an $\eps$-discrete set in $(P(K_{r_0}), \norm{\,\cdot\,}_{H^s})$ having $> g(\eps)$ elements, where 
\[
g(\eps) = e^{c \eps^{-n_M/\delta}}.
\]
On the other hand, since $F \circ P$ maps $K_{r_0}$ into a bounded set of $H^{t+m}_Q(N)$ and since $j: H^{t+m}_Q(N) \to H^t_Q(N)$ has entropy numbers satisfying $e_k(j) \lesssim j^{-m/n_N}$, we see that for $\delta > 0$ small there is a $\delta$-net covering $F(P(K_{r_0}))$ with $\leq f(\delta) := e^{C \delta^{-n_N/m}}$ elements. Part (a) now follows from Theorem \ref{thm:Mandache} where $K$ is replaced by $P(K_{r_0})$. Part (b) follows similarly.

We now consider the case where $L \neq M$. Then there is a smooth compact subdomain $D$ of $M$ so that $L \subset D^{\mathrm{int}}$. Let $M_1$ be a closed manifold containing $D$, so that we may identify functions in $H^s_L(M)$ with functions in $H^s_L(M_1)$. By assumption there is $(x,\xi) \in T^* L^{\mathrm{int}} \setminus \{0\}$ so that $P$ is noncharacteristic at $(x,\xi)$. Let $\tilde{P} = P \chi$ where $\chi \in C^{\infty}_c(L^{\mathrm{int}})$ satisfies $\chi = 1$ near $x$. Also choose $r_1$ so that $\norm{\chi u}_{H^{s+\delta}(M_1)} \leq r_0$ when $\norm{u}_{H^{s+\delta}(M_1)} \leq r_1$. 

By the assumption $P(K_{r_0}) \subset K_r$ and linearity of $P$, we see that $\tilde{P}$ maps functions in $H^{s+\delta}_L(M)$ to functions supported in $L$. Hence we may consider $\tilde{P}$ as a $\Psi$DO on $M_1$. By Theorem \ref{thm:Weyl_microlocal} and \eqref{entropy_capacity_relation}, the capacity numbers of $\tilde{P}: H^{s+\delta}(M_1) \to H^s(M_1)$ satisfy $c_k(\tilde{P}) \gtrsim k^{-\delta/n_M}$. Then there is an $\eps$-discrete set in $(\tilde{P}(\{ u \in H^{s+\delta}(M_1) \,;\, \norm{u}_{H^{s+\delta}} \leq r_1 \}), \norm{\,\cdot\,}_{H^s})$ having $> g(\eps)= e^{c \eps^{-n_M/\delta}}$ elements. Since $\tilde{P}u = P(\chi u)$ and $\chi$ is supported in $L$, this yields an $\eps$-discrete set in $(P(K_{r_0}), \norm{\,\cdot\,}_{H^s})$ having $> g(\eps)$ elements. The rest of the proof proceeds as before.
\end{proof}

Our discussion of microlocal smoothing effects continues as follows. In Section \ref{sec:micro_abstract} we begin by deducing instability results for general linear operators under the assumption of microlocal $C^{\infty}$ or Gevrey smoothing. In the next Section \ref{sec:applicationsI} we then discuss explicit inverse problems with $C^{\infty}$ and real-analytic (or Gevrey) microlocal smoothing. Here we show that using our abstract results from the first part of this section one can actually prove that there is super-polynomial or exponential instability in these applications.

\subsection{Instability for microlocally smoothing linear operators}
\label{sec:micro_abstract}

We now discuss the microlocal smoothing properties of general linear operators based on wave front sets of their Schwartz kernels. Let $X$ and $Y$ be open sets in Euclidean spaces. Recall from \cite[Section 8.2]{Hoermander} that for a linear operator $\mathcal{K}$ from $C_c^{\infty}(Y)$ to $\mathcal{D}'(X)$ having Schwartz kernel $K \in \mathcal{D}'(X \times Y)$, we have a unique way of defining $\mathcal{K} u$ for any $u\in \mathcal{E}'(Y)$ with $WF(u)\cap WF'(K)_Y = \emptyset$ (see \cite[Theorem 8.2.13]{Hoermander}). In this case
\begin{align*}
WF(\mathcal{K} u) \subset WF(K)_X \cup WF'(K)\circ WF(u).
\end{align*} 
Here we have used the notations 
\begin{align*}
WF'(K)&=\{(x,y,\xi,\eta): \ (x,y,\xi,-\eta)\in WF(K)\},\\
WF(K)_X&= \{(x,\xi): \ (x,y,\xi,0)\in WF(K)\mbox{ for some } y \in Y\},\\
WF'(K)_Y&= \{(y,\eta): \ (x,y,0,-\eta)\in WF(K) \mbox{ for some } x \in X\}.
\end{align*}
Moreover, let $\mathcal{K}_2$ be a linear operator with kernel $K_2 \in \mathcal{D}'(Y\times Z)$ so that the projection $\mathrm{supp}(K_2) \ni (y,z) \mapsto z$ is proper (then we say that $\mathcal{K}_2$ \emph{preserves compact supports}), and let $\mathcal{K}_1$ be an operator with kernel $K_1 \in \mathcal{D}'(X\times Y)$ such that $WF'(K_1)_Y \cap WF(K_2)_Y = \emptyset$. Then the composition $\mathcal{K}_1\circ \mathcal{K}_2$ is defined as a map from $C_c^{\infty}(Z)$ to $\mathcal{D}'(X)$ (see \cite[Theorem 8.2.14]{Hoermander}). Its wave front set satisfies
\begin{align}
WF'(K) &\subset WF'(K_1)\circ WF'(K_2) \cup(WF(K_1)_X \times Z \times \{0\}) \notag \\
& \quad \cup(X \times \{0\}\times WF'(K_2)_{Z}). \label{wf_composition_formula}
\end{align}
The same statements are valid on smooth manifolds if one works with half densities (see \cite[Section 2.5]{Hormander_FIO1}), or on smooth Riemannian manifolds if we use the volume form to define Schwartz kernels and spaces of distributions.

We first discuss the absence of H{\"o}lder stability bounds in the presence of $C^{\infty}$ microlocal smoothing.

\begin{theorem} \label{thm_microlocal_smoothing_general_v2}
Let $M_1, M_2$ be smooth manifolds, and let $A: C^{\infty}_c(M_1) \to \mathcal{D}'(M_2)$ be a continuous linear operator which preserves compact supports. Suppose that $A$ is microlocally smoothing at $(y_0,\eta_0) \in T^* M_1 \setminus \{0\}$, in the sense that $$WF(K_A)_{M_2} = \emptyset$$ and $(x,y,\xi,\eta) \notin WF'(K_A)$ for  $(x,\xi) \in T^* M_2$ and for $(y,\eta) \in V$, where $V \subset T^{\ast} M_1$ is some conic neighbourhood of $(y_0,\eta_0)$. Let also $L_1 \subset M_1$ be compact with $y_0 \in L_1^{\mathrm{int}}$. If for some $s_1, s_2 \in \mR$ and $\delta, r > 0$ one has the stability estimate 
\begin{equation} \label{microlocal_smoothing_holder_claim_1}
\norm{u}_{H^{s_1}(M_1)} \leq \omega(\norm{Au}_{H^{s_2}(M_2)}), \qquad \norm{u}_{H^{s_1+\delta}_{L_1}(M_1)} \leq r,
\end{equation}
then for any $\alpha \in (0,1)$ one has $\omega(t) \gtrsim t^{\alpha}$ for $t$ small.
\end{theorem}

\begin{proof}[Proof of Theorem \ref{thm_microlocal_smoothing_general_v2}]
Since $A$ preserves compact supports, there is a compact $L_2 \subset M_2$ so that $A$ maps functions supported in $L_1$ to distributions supported in $L_2$. We seek to apply Theorem \ref{thm_microlocal_smoothing} with 
\begin{align*}
K_r = \{ u \in H^{s_1+\delta}_{L_1}(M_1); \ \|u\|_{H^{s_1+\delta}(M_1)}\leq r\}
\end{align*}
and a pseudodifferential operator $P \in \Psi^0_{\mathrm{cl}}(M_1)$ with the property that $P$ has nonvanishing principal symbol at $(y_0, \eta_0)$, $P$ is microlocally smoothing outside of $V$, and $P(K_{r_0}) \subset K_r$ for suitable $r_0 > 0$. In fact, $P$ can be constructed by left quantizing the local coordinate symbol $\chi(y) \psi(\eta)$ where $\chi$ is a cutoff supported near $y_0$ with $\chi(y_0) \neq 0$, and $\psi$ is a Fourier cutoff supported in $V$ with $\psi(\eta_0) = 1$ and $\psi$ is homogeneous of degree $0$ for $\abs{\eta} \geq 1$.

It is enough to show that the operator $AP: C^{\infty}_c(M_1) \to \mathcal{D}'(M_2)$ has a $C^{\infty}$ integral kernel. For if this is true, then for each $N \geq 0$ there is $R=R_N>0$ such that
\begin{align} \label{eq:map1}
AP(K_{r_0})\subset \{ v \in H^{s_2 + N}_{L_2}(M_2): \ \|v\|_{H^{s_2 + N}(M_2)}\leq R\},
\end{align}
using that $\|u\|_{H^{s_1+\delta}_{L_1}(M_1)}\leq r_0$ for $u \in K_{r_0}$. Then Theorem \ref{thm_microlocal_smoothing}(a) implies that $\omega(t) \gtrsim t^{\alpha}$ for all $\alpha \in (0,1)$. This proves the claimed result.

To show that $AP$ has smooth integral kernel, we note that $WF(K_A)_{M_2} = \emptyset$ by assumption and $WF'(K_P)_{M_1} = \emptyset$, since $P$ is a $\Psi$DO and hence $WF'(K_P)$ is contained in the diagonal $\{ (y,y,\eta,\eta) \,;\, \eta \neq 0 \}$. Then \eqref{wf_composition_formula} gives 
\begin{align}
\label{eq:WF_set_comp}
WF'(K_{AP}) \subset WF'(K_A) \circ WF'(K_P).
\end{align}
But $WF'(K_P) \subset \{ (y,y,\eta,\eta) \,;\, (y,\eta) \in V \}$ since $P$ is microlocally smoothing away from $V$, and by assumption $WF'(K_A)$ does not have elements of the form $(x,y,\xi,\eta)$ with $(y,\eta) \in V$. Thus $WF'(K_{AP}) = \emptyset$ and $AP$ is indeed smoothing.
\end{proof}

\begin{example}
As a special case of Theorem \ref{thm_microlocal_smoothing_general_v2}, let $A \in \Psi^m(M)$ be a pseudodifferential operator on a closed manifold $M$. One says that $A$ is microlocally smoothing at $(y_0, \eta_0) \in T^* M \setminus \{0\}$ if $(y_0,y_0,\eta_0,\eta_0) \notin WF'(K_A)$, or equivalently if the full symbol of $A$ in some local coordinates is of order $-\infty$ in some conic neighborhood of $(y_0,\eta_0)$ (see \cite[Proposition 18.1.26]{Hoermander}). In such a case $A$ satisfies the conditions of the previous theorem.
\end{example}

We next seek to extend Theorem \ref{thm_microlocal_smoothing_general_v2} to the setting of Gevrey regularizing operators. To this end, we refer to Section \ref{sec:analytic} for the calculus of Gevrey pseudodifferential operators and the definition of the Gevrey wave front set $WF_{G,\sigma}(u)$.

In the setting of Gevrey (and analytically) regularizing operators, we will deduce that an inverse problem can at best have logarithmic stability.

\begin{theorem} \label{thm_microlocal_smoothing_general_analytic}
Let $M_1, M_2$ be $G^{\sigma}$ manifolds with Gevrey regularity index $\sigma>1$, and let $A: C^{\infty}_c(M_1) \to \mathcal{D}'(M_2)$ be a continuous linear operator which preserves compact supports. Assume that $A$ is microlocally $G^{\sigma}$ smoothing at $(y_0,\eta_0) \in T^* M_1 \setminus \{0\}$, in the sense that $$WF_{G,\sigma}(K_A)_{M_2} = \emptyset$$ and $(x,y,\xi,\eta) \notin WF'_{G,\sigma}(K_A)$ for  $(x,\xi) \in T^* M_2$ and for $(y,\eta) \in V \subset T^{\ast} M_1$ where $V$ is a conic neighbourhood of $(y_0,\eta_0)$. Let also $L_1 \subset M_1$ be compact with $y_0 \in L_1^{\mathrm{int}}$. If for some $s_1, s_2 \in \mR$ and $\delta, r > 0$ one has the stability estimate 
\begin{equation} \label{microlocal_smoothing_log_claim}
\norm{u}_{H^{s_1}(M_1)} \leq \omega(\norm{Au}_{H^{s_2}(M_2)}), \qquad \norm{u}_{H^{s_1+\delta}_{L_1}(M_1)} \leq r,
\end{equation}
then necessarily $\omega(t) \gtrsim \abs{\log t}^{-\frac{\delta(\sigma \dim(M_2) +1)}{\dim(M_1)}}$ for $t$ small.
\end{theorem}

\begin{proof}[Proof of Theorem \ref{thm_microlocal_smoothing_general_analytic}]
We argue similarly as in the proof of Theorem \ref{thm_microlocal_smoothing_general_v2} again relying on the arguments from Theorem \ref{thm_microlocal_smoothing}. As above, we set 
\begin{align*}
K_r = \{f \in H^{s_1+\delta}_{L_1}(M_1); \ \|f\|_{H^{s_1+\delta}(M_1)}\leq r\}
\end{align*}
and fix a Gevrey pseudodifferential operator $P \in \Psi^0_{\mathrm{cl}}(M_1)$ with the property that $P$ has nonvanishing principal symbol at $(y_0, \eta_0)$ and is $G^{\sigma}$ microlocally smoothing outside of $V$, and $P(K_{r_0}) \subset K_r$ for suitable $r_0 > 0$. More precisely, $P$ can be constructed by quantizing the symbol $\chi(y) \psi(\eta)$ where $\chi$ is a $G^{\sigma}$ cutoff function supported near $y_0$, and $\psi(\eta)$ is a $G^{\sigma}$ Fourier cutoff near $\eta_0$ as in \cite[Proposition 3.4.4]{Rodino_book}.

We seek to prove that $AP$ is $G^{\sigma}$ regularizing which implies the condition from  Theorem \ref{thm_microlocal_smoothing}(b).  
In order to infer this, we note that \eqref{wf_composition_formula} remains true in the $G^{\sigma}$ setting (see \cite[Section 8.5]{Hoermander}). This implies that \eqref{eq:WF_set_comp} determines the Gevrey wave front set of $K_{AP}$, i.e.\ that
\begin{align*}
WF_{G,\sigma}'(K_{AP}) \subset WF_{G,\sigma}'(K_A) \circ WF_{G,\sigma}'(K_P).
\end{align*}
The assumed regularity of $K_A$ and the construction of $P$ entail that as in the proof of Theorem \ref{thm_microlocal_smoothing_general_v2} we also infer that $WF'_{G,s}(K_{AP}) = \emptyset$. Finally, we observe that this yields the desired smoothing result 
\begin{align}
\label{eq:map2}
AP(K_{r_0}) \subset \{y\in A^{\sigma,\rho}_{L_2}: \ \|y\|_{A^{\sigma,\rho}}\leq R\} \mbox{ for some } \rho, R>0,
\end{align}
since elements in the set $K_{r_0}$ are uniformly bounded and the kernel $K_{AP}$ satisfies uniform $G^{\sigma}$ bounds which combined yield \eqref{eq:map2} as in the proof of Theorem \ref{thm:eigenval_bds}. Hence, Theorem \ref{thm_microlocal_smoothing}(b) can be invoked and it implies the claimed estimate.
\end{proof}

\subsection{Applications}
\label{sec:applicationsI}

In this section we discuss a number microlocal instability results. These include variants of the Radon and geodesic X-ray transforms and applications of these, for instance, in the inverse transport problem.

\subsubsection{The Radon transform with limited data}

As a first example of the microlocal regularizing argument from the previous section, we consider the two-dimensional Radon transform. To this end, we introduce the following coordinates: Let $\theta=\theta(\varphi)=(\cos(\varphi),\sin(\varphi))$, $\theta^{\perp}=\theta^{\perp}(\varphi)=(-\sin(\varphi),\cos(\varphi))$ with $\varphi \in \T$, where $\T=[0,2\pi]$ with $0,2\pi$ being identified. We consider a line perpendicular to $\theta=\theta(\varphi)$ and with distance $s \in \R$ to the origin: 
\begin{align*}
L(\varphi, s):=\{x\in \R^2: x\cdot \theta(\varphi)=s\}.
\end{align*}
For $f\in C^{\infty}_c(\R^2)$ we then define the Radon transform as 
\begin{align}
\label{eq:Radon}
R(f)(\varphi,s):= \int\limits_{x\in L(\varphi,s)}f(x) \,dx = \int\limits_{-\infty}^{\infty} f(s \theta + t \theta^{\perp}) \,dt.
\end{align}
The operator $R$ extends by duality to $\mathcal{E}'(\mR^2)$. By the definition of $R$ we obtain that $R(f)(\varphi + \pi, -s) = R(f)(\varphi,s)$ since the symmetry $S: (\varphi,s)\mapsto (\varphi + \pi,-s)$ is also present in the definition of $L(\varphi,s)$.

It is well known that the \emph{full data} Radon transform is invertible in suitable function spaces. For example, if $L \subset \mR^2$ is a compact set, $s \in \mR$, and $H^s_L(\mR^2)$ is the space of those $f \in H^s(\mR^2)$ with $\mathrm{supp}(f) \subset L$, then one has \cite[Theorem II.5.1]{Natterer_book}
\[
\norm{f}_{H^s(\mR^2)} \lesssim \norm{Rf}_{H^{s+1/2}(\T \times \mR)} \lesssim \norm{f}_{H^s(\mR^2)}, \qquad f \in H^s_L(\mR^2).
\]
However, it is also known that the \emph{limited data} Radon transform may have severe instabilities.
Here we are interested in deducing (exponential) instability results for the limited data Radon transform from our microlocal regularity arguments.

We consider the Radon transform acting on $f \in H^s_L(\mR^2)$, where $L$ is the closed unit ball in $\mR^2$. Then clearly $Rf(\varphi,s) = 0$ for $\abs{s} > 1$. We consider the situation in which we only measure the Radon transform on the lines $L(\varphi,s)\in \Lambda$, where $\Lambda \subset \T \times (-1,1)$ is an open set with $S(\Lambda) = \Lambda$ and such that $(\T \times (-1,1)) \setminus \ol{\Lambda}$ is open and non-empty. This means that we do not have measurements over some fixed open set of lines.

In order to understand the microlocal regularizing properties of the Radon transform, we recall the following result describing the singularities of the Schwartz kernel of the Radon transform (see \cite{Quinto1}, Theorem 2.8 and \cite{Quinto}, Theorem 3.1):

\begin{lemma}
\label{lem:prop_singularities_Radon}
Let $R$ be the Radon transform defined in \eqref{eq:Radon} above. 
Then, $R$ has the integral kernel
\begin{align*}
K_R(\varphi,s,x):= \frac{1}{2 \pi} \int\limits_{-\infty}^{\infty} e^{i (s- \theta \cdot x) \sigma} \,d\sigma,
\end{align*}
where the integral is understood as an oscillatory integral.
As a consequence, for any $\mu>1$ we have that
\begin{align}
\label{eq:WFRadon} 
WF(K_R),\,WF_{G,\mu}(K_R) \subset \{(\varphi, s, x,-(\theta^{\perp}\cdot x) \sigma, \sigma, - \theta \sigma ) \,;\, \ \sigma\neq 0, \ s = \theta\cdot x\}.
\end{align}
\end{lemma}

\begin{proof}
The claim follows from the fact that by the Fourier slice theorem for $f\in \mathcal{S}(\R^2)$ it holds
\begin{align*}
R(f)(\varphi,s)= \F_s^{-1} \F_s(f)(\varphi,s)
= \frac{1}{2 \pi} \int\limits_{-\infty}^{\infty} \int\limits_{\R^2} e^{i(s-\theta \cdot x) \sigma} f(x) \,dx \,d\sigma,
\end{align*}
where $\F_s$ denotes the Fourier transform in the $s$-variable,
see \cite[equation (A.7)]{Quinto1}. Interpreted as an oscillatory integral, this proves the identity for the kernel. The expression for the wave front set now follows from the usual decay estimates for oscillatory integrals in the case of the $C^{\infty}$ wave front set and from similar considerations for the case of the Gevrey wave front set, see for instance Lemma \ref{prop:analytic_FIO} in the appendix. 
\end{proof}

With these properties in hand, we can formulate our main result on the optimality of logarithmic stability estimates for the limited data Radon transform.

\begin{proposition}
\label{prop:LDRT}
Let $\Lambda \subset \T \times (-1,1)$ be an open set with $S(\Lambda) = \Lambda$ and $(\T \times (-1,1)) \setminus \ol{\Lambda}$ nonempty. If there exists a modulus of continuity $\omega$ such that for some $s_1, s_2 \in \R$, $\delta>0$ one has the stability estimate 
\begin{align*}
\|f\|_{H^{s_1}(\R^2)} \leq  \omega(\|Rf\|_{H^{s_2}(\Lambda)}), \qquad \norm{f}_{H^{s_1+\delta}_L(\mR^2)} \leq 1,
\end{align*}
then for any $\mu>1$ there exists some constant $C_{\mu}>0$ such that 
\begin{align*}
\omega(t)\geq C_{\mu} |\log(t)|^{-\frac{\delta(2\mu+1)}{2}}, \qquad \text{$t$ small}.
\end{align*}
\end{proposition}

\begin{remark}
\label{rmk:Gevrey}
We remark that the result of Proposition \ref{prop:LDRT} includes the limited angle problem or the exterior problem as special cases (see for instance Chapter VI in \cite{Natterer_book}).
\end{remark}

\begin{proof}
By the assumption on $\Lambda$, there exists an open set $U$ in $\T \times (-1,1)$ with $\ol{U} \cap \ol{\Lambda} = \emptyset$ and $S(U) = U$. Fix $\mu > 1$ and let $\chi_{\Lambda}$ be a $G^{\mu}$ cutoff function in $\T \times \mR$ so that $\chi_{\Lambda} = 1$ near $\ol{\Lambda}$ and $\chi_{\Lambda} = 0$ near $\ol{U}$. We wish to use Theorem \ref{thm_microlocal_smoothing_general_analytic}. Define 
\[
A = \chi_{\Lambda} R: C^{\infty}_c(\mR^2) \to C^{\infty}_c(\T \times \mR).
\]
Note that $A$ preserves compact supports. The stability condition in the proposition implies that 
\[
\norm{f}_{H^{s_1}(\mR^2)} \leq \omega(\norm{A f}_{H^{s_2}(\T \times \mR)}), \qquad \norm{f}_{H^{s_1+\delta}_L} \leq 1.
\]

We now analyze the wave front set of $K_A$. Since $\chi_{\Lambda} \in G^{\mu}$, \eqref{eq:WFRadon} gives that 
\begin{align}
WF_{G,\mu}'(K_A) &\subset \{ (\varphi,s,x,\hat{\varphi},\hat{s},\xi) \in WF_{G,\mu}'(K_R) \,;\, (\varphi,s) \in \supp(\chi_{\Lambda}) \} \notag \\
 &\subset  \{(\varphi, s, x,-(\theta^{\perp}\cdot x) \sigma, \sigma, \theta \sigma ) \,;\, \ \sigma\neq 0, \ s = \theta\cdot x, \ (\varphi, s) \in \supp(\chi_{\Lambda}) \}. \label{wfgmu_radon}
\end{align}
Since $\theta \sigma \neq 0$ in the last expression, we always have $WF_{G,\mu}(K_A)_{\T \times \mR} = \emptyset$.

 Now fix some $(\varphi_0, s_0) \in U$, so that lines near $L(\varphi_0, s_0)$ are not in the data. Let $x_0 \in L(\varphi_0, s_0)$ and $\xi_0 = \theta(\varphi_0)$. We claim that $A$ is microlocally $G^{\mu}$ smoothing near $(x_0, \xi_0)$. For if 
 \[
 (\varphi,s,x_0,\hat{\varphi}, \hat{s}, \xi_0) \in WF_{G,\mu}'(K_A),
 \]
 then \eqref{wfgmu_radon} implies that $\theta(\varphi) \sigma = \xi_0$ and $s = \theta(\varphi) \cdot x_0$, which yields $(\varphi,s) = (\varphi_0, s_0)$ or $(\varphi,s) = S(\varphi_0, s_0)$. However, we must also have $(\varphi, s) \in \supp(\chi_{\Lambda})$, which is impossible since $\chi_{\Lambda} = 0$ near $\ol{U} = \ol{S(U)}$. Since $U$ is open, we may vary $(x_0,\xi_0)$ slightly in the previous argument. This proves that $A$ is microlocally $G^{\mu}$ smoothing near $(x_0, \xi_0)$. The result now follows from Theorem \ref{thm_microlocal_smoothing_general_analytic}.
\end{proof}

\subsubsection{Superpolynomial instability of the geodesic X-ray transform in the presence of conjugate points}
\label{sec:Xray}

As another example of an instability arising from microlocal smoothing, we discuss the instability properties of the geodesic X-ray transform on two-dimensional Riemannian surfaces $(M,g)$. Throughout this section, we assume that $(M,g)$ is a compact Riemannian manifold with smooth boundary which is \emph{non-trapping}, i.e.\ all geodesics starting within $M$ meet $\partial M$ in finite time in both directions, and that its boundary $\partial M$ is \emph{strictly convex}. While it is known that for \emph{simple} manifolds, i.e.\ for manifolds which are non-trapping, have strictly convex boundary and no conjugate points, the X-ray transform is injective and Lipschitz stable in suitable Sobolev norms \cite{SU04, AS20}, in our discussion we allow for the presence of conjugate points in $(M,g)$. 

Following the article \cite{MSU15}, the set-up here is the following: $M$ is two-dimensional, $\gamma_0$ is a fixed directed maximal geodesic in $M$, and $f $ is a function supported away from the endpoints of $\gamma_0$. In this setting, the article \cite{MSU15} investigates which parts of the wave front set $WF(f)$ of $f$ can be recovered from knowing the geodesic X-ray transform 
\begin{align*}
I f (x,v):= \int_0^{\tau(x,v)} f(\pi_1 \varphi_s(x,v)) ds.
\end{align*}
Here, for $\nu_+(x)$ being the outward normal vector at $x\in \partial M$, we consider   
\begin{align*}
(x,v) \in \partial_+ SM:=\{(x,v)\in SM: \ x \in \partial M, \ \langle v, \nu_+(x) \rangle < 0 \},
\end{align*}
where $SM = \{ (x,v) \in TM \,;\, \abs{v}_g = 1 \}$ is the unit sphere bundle, $\tau \in C^{\infty}(\partial_+ SM)$ denotes the boundary hitting time, and $\pi_1 \varphi_s(x,v)$ denotes the projection onto the first component of the geodesic flow $\varphi_s$ onto $(M,g)$. In order to focus on a neighbourhood of $\gamma_0$ only, we assume that $(x,v)$ is such that the associated geodesics are close enough to $\gamma_0$, and we identify $\gamma_0$ with $(x_0,v_0)\in \p_+ SM$. 

Below, we will use several basic properties of Fourier integral operators (FIOs) which may be found in \cite[Chapter 25]{Hoermander}. A first key observation (see \cite[Proposition 3.1 and Theorem 3.2]{MSU15}) on the two-dimensional geodesic X-ray transform is the following result:

\begin{itemize}
\item[(i)] $I$, restricted to functions supported near $\gamma_0$ and evaluated for geodesics near $\gamma_0$, is an FIO of order $-\frac{1}{2}$ whose canonical relation $C$ can be explicitly computed and which is \emph{locally} the graph of a canonical map $\mathcal{C}$ which is locally a diffeomorphism (see \cite[Theorem 3.2]{MSU15}). The map $\mathcal{C}$ becomes a \emph{global} diffeomorphism if no conjugate points are present. The absence of conjugate points is known as the Bolker condition.
\end{itemize}

The microlocal smoothing arguments from \cite{MSU15} exploit the violation of the Bolker condition, i.e.\ hold if conjugate points are present.  To summarize the results from \cite{MSU15} let us assume that $x_1,x_2 \in M^{\mathrm{int}}$ are conjugate points on the geodesic $\gamma_0$ from above with no further conjugate points in between them. Let $v_1, v_2$ be the tangent vectors to $\gamma_0$ at the points $x_1, x_2 \in M$ and $\xi_j = (R_{-\pi/2} v_j)^{\flat}$ correspond to their rotations by $-90^{\circ}$, respectively. Let $V_1$ and $V_2$ be small neighbourhoods in $T^* M$ of $(x_1, \xi_1)$ and $(x_2,-\xi_2)$. Then, the following assertion may be found in \cite[Theorem 4.2]{MSU15} (note that our $V_1$ and $V_2$ are $V^1_+$ and $V^2_-$ in \cite{MSU15}):
\begin{itemize}
\item[(ii)] If $V_1$, $V_2$ are sufficiently small and $\mathcal{C}_j = \mathcal{C}|_{V_j}$, the restrictions of the canonical relation $C$ of $I$ to $V_1, V_2$ are canonical graphs $C|_{V_1} = \{ (\mathcal{C}_1(x,\xi), x,\xi) \,;\, (x,\xi) \in V_1 \}$ and $C|_{V_2} =  \{ (\mathcal{C}_2(x,\xi), x,\xi) \,;\, (x,\xi) \in V_2 \}$ with $\mathcal{C}_1(V_1) = \mathcal{C}_2(V_2)$. The map $\mathcal{C}_2^{-1} \circ \mathcal{C}_1$ is a diffeomorphism $V_1 \to V_2$. 
\end{itemize} 
Further, in view of an application of Theorem \ref{thm_microlocal_smoothing_general_v2} we let $L \subset M^{\mathrm{int}}$ be a compact set containing the $x$-projection of $V_1 \cup V_2$.

We can now formally describe the microlocal smoothing argument of \cite{MSU15} based on cancellation of singularities. Suppose that $f_1 \in C^{\infty}(M)$ is supported near $x_1$ and has wave front set in $V_1$. We wish to find a function $f_2 \in C^{\infty}(M)$, supported near $x_2$ with wave front set in $V_2$, so that 
\[
I(f_1 + f_2) \in C^{\infty}.
\]
Let $I_j$ denote $I$ acting on functions whose wave front set is contained in $V_j$. Since the canonical relation of $I_j$ is the graph of $\mathcal{C}_j$ and since $I_j$ is an elliptic FIO of order $-1/2$, see \cite[proof of Theorem 4.3]{MSU15}, $I_j$ has a parametrix $I_j^{-1}$ (an FIO of order $1/2$ whose canonical relation is the graph of $\mathcal{C}_j^{-1}$) so that $I_j \circ I_j^{-1} = \mathrm{Id} + R_j$ where $R_j$ is microlocally smoothing near $\mathcal{C}_j(x_j, \xi_j)$. Since $I(f_1 + f_2) = I_1 f_1 + I_2 f_2$, we may just define 
\[
f_2 := -I_2^{-1} I_1 f_1
\]
and then $I(f_1 + f_2) \in C^{\infty}$. In the last step we used that $\mathcal{C}_1(V_1) = \mathcal{C}_2(V_2)$, which makes it possible to apply $I_2^{-1}$ to $I_1 f_1$.

We will next combine the idea above with our microlocal smoothing arguments to show that in the presence of conjugate points, the two-dimensional geodesic X-ray transform cannot be H\"older stable with respect to any Sobolev norms (even if it were injective). The fact that Lipschitz stability is not possible was already stated in \cite{MSU15}. It is likely that if $(M,g)$ is analytic/Gevrey, then the stability could not be better than logarithmic in this setting. However, this would require some facts about analytic/Gevrey FIOs which may not be easily available in the literature, and we will not consider this further.

\begin{theorem}
\label{thm:geodesic}
Let $(M,g)$ be a compact two-dimensional non-trapping Riemannian manifold with interior conjugate points and strictly convex boundary $\partial M$. Let $I$ be as above. Assume that for some $s_1, s_2 \in \R$ and $\delta>0$ and for some modulus of continuity $\omega$ it holds
\begin{align*}
 \| f \|_{H^{s_1}(M)} \leq \omega(\|I f\|_{H^{s_2}(\partial_+ SM)})
\end{align*}
for all $f$ with $\|f\|_{H^{s_1 +\delta}(M)}\leq 1$.
Then, for any $\alpha \in (0,1]$ there exists a constant $C=C_{\alpha}>0$ such that $\omega(t)\geq C t^{\alpha}$ for $t$ small. 
\end{theorem}

\begin{proof}
We will follow the scheme from \cite{MSU15} described above. In deducing the instability property, we do not directly invoke Theorem \ref{thm_microlocal_smoothing_general_v2}, but instead follow its main ideas in a slightly modified functional setting. In the framework laid out in the proof of Theorem \ref{thm_microlocal_smoothing} and Lemma \ref{lemma_entropy_typical} we consider 
\begin{align*}
I: P(K) \rightarrow Y,
\end{align*}
with 
\begin{align}
\label{eq:func_set}
\begin{split}
&X = H^{s_1}_{L}(M),\ X_1:= H^{s_1 +  \delta}_{L}(M), \
 K = \{ f \in H^{s_1+ \delta}_{L}(M) \,;\, \norm{f}_{H^{s_1+\delta}} \leq 1 \},\\
&Y = H^{s_2}(\partial_+ SM),\ Y_1:= H^{s_2 + N}(\partial_+ SM) \mbox{ with } N \in \N \mbox{ arbitrary but fixed}, \\
 &P = P_1 - P_2 I_2^{-1} I_1 P_1.
\end{split}
\end{align}
Here $P_j = \chi_j(x)Op(\chi_j(x) \psi_j(\xi))\chi_j(x) \in \Psi^0_{\mathrm{cl}}(M^{\mathrm{int}})$ is a microlocal cutoff to $V_j$ obtained by quantizing the symbol $\chi_j(x) \psi_j(\xi)$ and multiplying with $\chi_j(x)$ from the right and left. Here $\chi_j$ is a smooth function supported near $x_j$ with $\chi_j = 1$ near $x_j$, and $\psi_j$ is homogeneous of degree $0$ for $\abs{\xi} \geq 1$ with $\psi_j(\xi_j) = 1$. Moreover, $I_j$ is the microlocalized version $I$ as described after (ii) above. We also note that functions in $P(K)$ are supported in $L$, and hence functions in $I(P(K))$ are compactly supported in $(\p_+ SM)^{\mathrm{int}}$, see e.g.\ \cite{PSU20+} for this simple fact which uses strict convexity of $\p M$. We will prove below that $I$ indeed maps $P(K)$ to $C^{\infty}_c((\p_+ SM)^{\mathrm{int}})$ and hence to $Y$.

The main difference in the functional setting with respect to the proof of Theorem \ref{thm_microlocal_smoothing} is that $P$ is not a $\Psi$DO, but rather an FIO of order zero. However, modulo smoothing operators, one has 
\[
P^* P = P_1^* P_1 + P_1^* (I_2^{-1} I_1)^* P_2^* P_2 I_2^{-1} I_1 P_1.
\]
Since $I_2^{-1} I_1$ is an FIO associated with the canonical graph of $\mathcal{C}_2^{-1} \circ \mathcal{C}_1$, its adjoint is associated with the graph of $\mathcal{C}_1^{-1} \circ \mathcal{C}_2$. Here we use that $\mathcal{C}_2^{-1} \circ \mathcal{C}_1$ is a diffeomorphism $V_1 \to V_2$. Thus by the composition calculus of FIOs, $P^* P$ is $\Psi$DO of order zero, and one can check that its principal symbol is nonvanishing at $(x_1, \xi_1)$. Now if $M$ were a closed manifold, we could use Theorem \ref{thm:Weyl_microlocal} to show that $P$ as an operator $H^{s_1+\delta}(M) \to H^{s_1}(M)$ would have singular value bounds $\sigma_k(P) \sim k^{-\frac{\delta}{2}}$. Since $M$ has a boundary, we consider $P$ as an operator $H^{s_1+\delta}_L(M) \to H^{s_1}_L(M)$ instead, and argue as in the end of proof of Theorem \ref{thm_microlocal_smoothing} to obtain the analogous capacity estimates.

In order to conclude the desired instability result, it remains to check the conditions stated in Theorem \ref{thm_microlocal_smoothing} in our functional setting from \eqref{eq:func_set}. The result then follows by invoking Lemma \ref{lemma_entropy_typical} and Theorem \ref{thm:Mandache}.
However, by construction of $P$ and the characterization of the wave front sets in \cite{MSU15} (see our discussion after (ii) above, where we may choose $f_1 = P_1 f$ for $f \in K$) we infer that $I(P(K)) \subset C^{\infty}(\p_+ SM)$ and hence  $I(P(K)) \subset \{y \in Y_1: \|y\|_{Y_1}\leq R\}$ for some $R=R(N)>0$. This concludes the proof.
\end{proof}

We remark that the conditions on the manifold are not void. For both analytic and numerical examples of such manifolds, we refer to \cite[Section 6]{MSU15}. Partial results in higher dimensional settings are proved in \cite{HU18}.

\subsubsection{Exponential instability in quantitative unique continuation for the wave equation}

\label{sec:wave_ucp}

As a further application of the microlocal analytic (or Gevrey) smoothing mechanism to instability, we discuss the exponential instability of the unique continuation property for the wave equation. To this end, consider $(M,g)$ an analytic, closed $n$-dimensional manifold. Let $\Omega \subset M$ with $\Omega \neq \emptyset$ be open and such that $M \setminus \overline{\Omega} \neq \emptyset$ is also open. We are interested in the quantitative unique continuation properties for the wave equation
\begin{align}
\label{eq:wave}
\begin{split}
\p_t^2 u - \D_g u & = 0 \mbox{ in } M \times (0,T),\\
(u, \p_t u) &= (u_0, u_1) \mbox{ on } M \times \{0\},
\end{split}
\end{align}
on $(M,g)$ with observation data in $\Omega$. Due to the finite speed of propagation of the wave equation, we assume that the observation time $T$ always satisfies
\begin{align*}
T > 2 \max\{\dist(x,\Omega), \ x \in M\}.
\end{align*}
By duality to unique continuation, this guarantees that approximate observability holds \cite{Tataru_wave, LaurentLeautaud}.
In order to show the strength of our Gevrey microlocal instability argument, we assume that there exists a unit speed geodesic $\gamma(t)$ in $(M,g)$ which never meets $\ol{\Omega}$ for $0 \leq t \leq T$. This is a situation in which the geometric control condition of \cite{BardosLebeauRauch} is not satisfied. 

Under this condition, we reprove the fact that the stability estimates which were deduced in \cite{LaurentLeautaud} and which are logarithmic are indeed optimal (similar observations were already made in earlier works starting with Lebeau \cite{Lebeau}; by discussing this example, it is not our main purpose to present a new result, but to show the versatility of the instability ideas from the previous sections):

\begin{theorem}
\label{thm:wave_instab_UCP}
Let $(M,g)$ and $\Omega$ be as above. Assume that there exists some modulus of continuity $\omega$ satisfying \eqref{eq:modulus_cond} such that for all $(u_0, u_1) \in H^1(M) \times L^2(M)$ and $u\in L^2((0,T), H^1(M))$ solving \eqref{eq:wave} it holds
\begin{align*}
\|(u_0,u_1)\|_{L^2(M)\times H^{-1}(M)}
\leq \omega\left( \frac{\|u\|_{L^2((0,T),H^1(\Omega))}}{\|(u_0,u_1)\|_{H^1(M)\times L^2(M)}} \right) \|(u_0,u_1)\|_{H^1(M)\times L^2(M)}.
\end{align*}
Then for any $\sigma>1$ there is $C_{\sigma} >0$ such that $\omega(t) \geq C_\sigma |\log(t)|^{- \frac{\sigma(n+1)+1}{n}}$ for $t$ small.
\end{theorem}
\begin{proof}
For any $(y, \eta) \in T^* M \setminus \{0\}$, denote by $\gamma_{y,\eta}$ the geodesic in $(M,g)$ with $\gamma_{y,\eta}(0) = y$ and $\dot{\gamma}(y,\eta) = \eta^{\sharp}$. By assumption, there exists $(y_0, \eta_0) \in T^* M$ with $\abs{\eta_0}_g = 1$ so that the geodesic $\gamma_{y_0,\eta_0}(t)$ never meets $\ol{\Omega}$ for $0 \leq t \leq T$. By compactness, there is a positive distance between the sets $\gamma_{y_0,\eta_0}([-\eps,T+\eps])$ and $\ol{\Omega}$ for some $\eps > 0$. Thus there is a small conic neighborhood $C$ of $(y_0,\eta_0)$ so that no geodesic $\gamma_{y,\eta}$ meets $\ol{\Omega}$ for $-\eps \leq t \leq T+\eps$ when $(y,\eta) \in C$ and $\abs{\eta}_g = 1$.

In general, a singularity at $(y,\eta)$ at time $t=0$ for a solution of the wave equation will propagate in two opposite directions. In order to prove our statement we need to focus on forward propagating singularities. To this end, we factorize the wave operator as 
\[
\p_t^2 - \Delta_g = (\p_t - iB)(\p_t + iB)
\]
where $B = (-\Delta_g)^{1/2}$ is an elliptic $\Psi$DO of order one. If $(\p_t + iB) w = 0$ with $w(0) = u_0$, then $w$ also solves the wave equation with $(w(0), \p_t w(0)) = (u_0, -iBu_0)$. For such solutions one has (by the discussion around \eqref{eq:modulus_cond})  
\[
\norm{u_0}_{L^2(M)} \leq \omega(\|w\|_{L^2((0,T),H^1(\Omega))}), \qquad \norm{u_0}_{H^1(M)} \leq 1.
\]
Moreover, the $G^{\sigma}$ singularities of solutions of $(\p_t + iB) w = 0$ propagate along forward bicharacteristics in the following sense (see e.g.\ \cite[Theorem 1 in Lecture VI]{Mizohata}): one has $(y, \eta) \notin WF_{G,\sigma}(u_0)$ if and only if $\gamma_{y,\eta}(t) \notin WF_{G,\sigma}(w(\,\cdot\,,t))$.

We wish to use Theorem \ref{thm_microlocal_smoothing} in this setting. Fix a Gevrey pseudodifferential operator $P \in \Psi^0_{\mathrm{cl}}(M)$ obtained by quantizing the symbol $\chi(y) \psi(\eta)$ where $\chi$ is a $G^{\sigma}$ cutoff function supported near $y_0$, and $\psi(\eta)$ is a $G^{\sigma}$ Fourier cutoff near $\eta_0$ as in \cite[Proposition 3.4.4]{Rodino_book}. We may assume that $P$ is $G^{\sigma}$ smoothing away from $C$.

We next define an operator 
\[
A: L^2(M) \to L^2(\mR,H^1(M)), \ \ Au_0(x,t) = \psi_1(t) \psi_2(x) w(x,t)
\]
where $\psi_1$ is a $G^{\sigma}$ cutoff function supported in $(-\eps,T+\eps)$ with $\psi_1(t) = 1$ near $[0, T]$, and $\psi_2$ is a $G^{\sigma}$ cutoff function in $M$ with $\psi_2 = 1$ near $\ol{\Omega}$ so that the geodesics $\gamma_{y,\eta}$ for $(y,\eta) \in C$ and $\abs{\eta}_g = 1$ never meet $\supp(\psi_2)$ for $-\eps \leq t \leq T+\eps$. It follows that we have 
\[
\norm{u_0}_{L^2(M)} \leq \omega(\|Au_0\|_{L^2(\mR,H^1(M))}), \qquad \norm{u_0}_{H^1(M)} \leq 1.
\]

Now let $u_0 = Pv$ with $v \in L^2(M)$, and let $w$ solve $(\p_t + iB) w = 0$ with $w(0) = u_0 = Pv$ as above. By $G^{\sigma}$ propagation of singularities and the fact that $Pv$ is $G^{\sigma}$ away from $C$ when $v \in L^2(M)$, it follows that $w(\,\cdot\,,t)$ is a $G^{\sigma}$ function near $\supp(\psi_2)$ for any $t \in [-\eps,T+\eps]$. Since $\p_t w = -iBw$ is also $G^{\sigma}$ near $\supp(\psi_2)$ for any fixed $t$, it follows from \cite[end of Lecture IV]{Mizohata} and finite speed of propagation that the solution $w(x,t)$ of the wave equation is $G^{\sigma}$ for $(x,t)$ near $\supp(\psi_2) \times \supp(\psi_1)$. This proves that $A(Pv)$ is $G^{\sigma}$ in $M \times \mR$ whenever $v \in L^2(M)$. This remains true for any $v$ in the dual of $G^{\sigma}(M)$ (Gevrey ultradistributions), since $P$ is a Gevrey $\Psi$DO and since propagation of singularities is valid in this setting \cite{CZ}. This implies as in the proof of Theorem \ref{thm:eigenval_bds}(ii) that the map $v \mapsto A(Pv)$ has $G^{\sigma}$ integral kernel and hence is bounded $H^1(M) \to A^{\sigma,\rho}_Q(M \times \mR)$ for some $\rho > 0$ and for $Q = M \times [-\eps,T+\eps]$. The result now follows from Theorem \ref{thm_microlocal_smoothing}(b).
\end{proof}

\begin{remark}
\label{rmk:dual_control}
By a duality argument as outlined in Section \ref{sec:control}, we also obtain that exponential cost of control is optimal in the associated steering problem for the wave equation.
\end{remark}

As a corollary of Theorem \ref{thm:wave_instab_UCP} we obtain the logarithmic instability of unique continuation in space-like directions for the wave equation. This had already been noted earlier in \cite{John} in special geometries by a separations of variables argument and an analysis of the asymptotic behaviour of Bessel functions. More precisely, the result in \cite{John} states that if $0 < \rho < 1$ and $k \geq 1$, and if for some modulus of continuity $\omega$ one has the stability estimate 
\[
\norm{u}_{L^{\infty}(B_1 \times \mR)} \leq \omega(\norm{u}_{L^{\infty}(B_{\rho} \times \mR)})
\]
for any solution of $(\p_t^2 - \Delta) u = 0$ in $\mR^2 \times \mR$ with $\norm{u}_{W^{k,\infty}(\mR^2 \times \mR)} \leq 1$, then $\omega(t) \gtrsim \abs{\log t}^{-k}$ for $t$ small. Since rays of geometric optics that are tangent to $\p B_1 \times \mR$ never enter $\ol{B}_{\rho} \times \mR$, instability results of this kind with $L^2$ based norms can be obtained from Theorem \ref{thm:wave_instab_UCP} together with energy conservation and finite propagation speed.

\begin{remark} \label{rmk:alternative}
We remark that alternatively, at lower coefficient regularity, an analogue of \cite{John} could also have been obtaind by a separation of variables argument in combination with a ``blow-up'' argument as in Appendix \ref{sec:Carl}. This would then lead to the construction of solutions to Helmholtz-like equations which are perturbations of Bessel functions. Their asymptotic behaviour was analysed in \cite{John} to infer the instability in the constant coefficient case.
\end{remark}

\subsubsection{An inverse problem for the transport equation}

Up to now we discussed microlocal instability for linear inverse problems. As a direct example in which the instability of the geodesic X-ray transform translates into instability of a nonlinear inverse problem, we discuss an inverse problem for the transport equation.

Let $(M,g)$ be a compact, non-trapping Riemannian manifold with smooth strictly convex boundary, and consider a potential $q \in C^{\infty}(M)$. Let $SM = \{ (x,v) \in TM \,;\, |v| = 1 \}$ be the unit sphere bundle, let $X$ be the geodesic vector field on $SM$, and consider the boundary value problem 
\[
Xu + qu = 0 \text{ in $SM$}, \qquad u|_{\partial_+(SM)} = h
\]
where $\partial_{\pm}(SM) = \{ (x,v) \in \partial(SM) \,;\, \mp \langle v, \nu \rangle \geq 0 \}$ denote the incoming and outgoing boundary of $SM$ ($\nu$ is the unit outer normal to $\partial M$). On a fixed geodesic the equation $Xu + qu=0$ is an ODE $\frac{d}{dt} (u(\varphi_t)) + q(\varphi_t) u(\varphi_t) = 0$ where $\varphi_t$ denotes the geodesic flow on $SM$. The solution is 
\[
u(\varphi_t(x,v)) = \exp \left[ -\int_0^t q(\varphi_s(x,v)) \,ds \right] h(x,v), \qquad (x,v) \in \partial_+(SM).
\]
Consider the boundary map 
\[
\Lambda_q: C^{\infty}_c(\partial_+(SM)) \to C^{\infty}(\partial_-(SM)), \ \ \Lambda_q h = u|_{\partial_-(SM)}.
\]
This map takes the value of $u$ on the incoming boundary to the value of $u$ in the outgoing boundary, and is an analogue of the Dirichlet-to-Neumann map in the setting of transport equations. The map $\Lambda_q$ is equivalent information with the scattering data for the potential $q$ (see e.g.\ \cite{PSU_GAFA} where $q$ corresponds to the Higgs field $\Phi$).

We have the explicit expression 
\begin{align}
\label{eq:sol}
\Lambda_q h(\alpha(x,v)) = \exp(-Iq(x,v)) h(x,v), \qquad (x,v) \in \partial_+(SM),
\end{align}
where $Iq$ is the geodesic X-ray transform of $q$, and $\alpha: \partial(SM) \to \partial(SM)$ is the scattering relation defined so that $\alpha^2 = \mathrm{Id}$ and $\alpha$ is a $C^{\infty}$ diffeomorphism which is an isometric map with respect to natural weighted $L^2$ spaces on $\partial_{\pm}(SM)$, see \cite{PestovUhlmann}. In particular, the explicit expression shows that $\Lambda_q$ is a FIO of order $0$ with canonical relation given by the scattering relation $\alpha$. It also follows that 
\[
\| \Lambda_{q_1} - \Lambda_{q_2} \|_{L^2 \to L^2} \sim \| \exp(-Iq_1) - \exp(-Iq_2) \|_{L^{\infty}(\partial_+(SM))}.
\]
This shows that, under a priori bounds on $q_j$, one has 
\begin{align}
\label{eq:rep_transport}
\| \Lambda_{q_1} - \Lambda_{q_2} \|_{L^2 \to L^2} \sim \| I(q_1 - q_2) \|_{L^{\infty}(\partial_+(SM))}.
\end{align}
Thus the (in)stability properties of the inverse problem of recovering $q$ from $\Lambda_q$ are the same as those for the geodesic X-ray transform. 

\begin{theorem}
\label{thm:transport}
Let $(M,g)$ be compact and non-trapping with strictly convex boundary and let $q$ be as above. Then the following instability results hold:
\begin{itemize}
\item 
If $(M,g)$ is not simple, i.e.\ if conjugate points are present, but has injective X-ray transform, there is no H\"older stability in Sobolev norms, i.e.\ for no choice of $s_1, s_2, s_3$ with $s_3>s_2$ and $\alpha\in (0,1]$ do we have
\begin{align*}
\|q_1-q_2\|_{H^{s_1}(M)} \leq C\|\Lambda_{q_1}-\Lambda_{q_2}\|_{H^{s_2}(\partial_+(SM)) \rightarrow H^{s_2}(\partial_+(SM))}^{\alpha}
\end{align*}
with $q_1, q_2 \in K:=\{q\in L^2(M): \ \|q\|_{H^{s_3}(M)} \leq 1 \}$.
\item If $(M,g)$ does not have injective X-ray transform, then both uniqueness and H\"older stability fail.
\end{itemize}
\end{theorem}

\begin{proof}
We only prove the first claim, as by virtue of the representation \eqref{eq:sol} the lack of injectivity of the X-ray transform also implies the lack of injectivity (and hence stability) for the DtN map.

Assuming that we had H{\"o}lder stability in spite of working on a the non-simple manifold would amount to the existence of $s_1,s_2,s_3 \in \R$ with $s_3>s_2$ and $\alpha \in (0,1]$ such that
\begin{align}
\label{eq:assume_stab}
\|q_1 -q_2\|_{H^{s_1}(M)}
\leq C \|\Lambda_{q_1}-\Lambda_{q_2}\|_{H^{s_2}(\partial_+(SM)) \rightarrow H^{s_2}(\partial_+(SM))}^{\alpha}
\end{align}
for $q_1,q_2 \in K:=\{q\in L^2(M): \ \|q\|_{H^{s_3}(M)} \leq 1 \}$. Without loss of generality, we may further assume that $s_3\gg 1$ is as large as necessary (since, if a stability estimate were to hold for a small $s_3$, it would also hold for a large value of $s_3$).

Now, using the representation formula \eqref{eq:sol}, we show that a stability estimate of the form \eqref{eq:assume_stab} would imply a H{\"o}lder stability estimate for the geodesic X-ray transform in the presence of conjugate points. This, however would contradict Theorem \ref{thm:geodesic}.

For $s_3>0$ sufficiently large, by an expansion of the exponential function, we obtain
\begin{align*}
&\|[\exp(-Iq_1)-\exp(-I q_2)]h\|_{H^{s_2}(\partial_+ SM)}\\
&\leq \sum\limits_{k=1}^{\infty} \frac{C(\|q_1\|_{H^{s_3}(M)}, \|q_2\|_{H^{s_3}(M)})}{(k-1)!} \|(Iq_1 - I q_2) h\|_{H^{s_2}(\partial_+ SM)}\\
& \leq C(\|q_1\|_{H^{s_3}(M)}, \|q_2\|_{H^{s_3}(M)})\|(Iq_1 - I q_2) h\|_{H^{s_2}(\partial_+ SM)}.
\end{align*}
Since 
\begin{align*}
& \|\Lambda_{q_1}-\Lambda_{q_2}\|_{H^{s_2}(\partial_+(SM)) \rightarrow H^{s_2}(\partial_+(SM))}\\
& = C_{\alpha}\sup\limits_{\|h\|_{H^{s_2}(\partial_+ SM)}=1} 
 \|[\exp(-Iq_1)-\exp(-I q_2)]h\|_{H^{s_2}(\partial_+ SM)},
\end{align*}
\eqref{eq:assume_stab} would thus imply the H{\"o}lder stability of the geodesic X ray transform in the presence of conjugate points, a contradiction.
\end{proof}

\begin{remark}
\label{rmk:stable_simple}
We note that if $(M,g)$ is simple, then the geodesic X-ray transform is Lipschitz stable in suitable Sobolev spaces. Using \eqref{eq:rep_transport} together with Sobolev embedding, then also implies Lipschitz stability of the inverse transport problem.
\end{remark}

\begin{appendix}

\section{Instability for linear inverse problems}
\label{sec:abstr}

In this appendix we complement the results from the main text by giving several instability results for \emph{linear} inverse problems based on direct arguments using singular value decay or abstract smoothing properties. Compared to our strategy in the main part of the text, the final results will be similar, though in some cases slightly sharper, to those obtained by using the entropy/capacity estimates in Section \ref{sec:abstract} (which also apply to nonlinear inverse problems).

Let $A: X \to Y$ be a bounded linear operator between two separable Hilbert spaces. We consider the basic inverse problem for $A$: given $g \in \mathrm{Ran}(A)$, find $f \in X$ with 
\[
Af = g.
\]
In the linear inverse problems that we are interested in, the following conditions will hold:
\begin{itemize}
\item 
$A$ is injective. This means that the inverse problem for $A$ has a unique solution $f \in X$ for any $g \in \mathrm{Ran}(A)$.
\item 
$X$ is infinite dimensional. This rules out trivial cases, and since $A$ is injective it implies that $\mathrm{Ran}(A)$ (and hence also $Y$) is infinite dimensional.
\item 
$A$ is compact. Then $A$ has a singular value decomposition (see Proposition \ref{prop_singular_value_decomposition}) with singular values $(\sigma_j)_{j=1}^{\infty}$ converging to zero as $j \to \infty$, a singular value orthonormal basis $(\varphi_j)$ of $X$, and an orthonormal set $(\psi_j)$ in $Y$ so that $A\varphi_j = \sigma_j \psi_j$ and $A^* \psi_j = \sigma_j \varphi_j$.
\end{itemize}

The fact that the singular values converge to zero implies that the inverse problem for $A$ is ill-posed (not Lipschitz stable) in $X$, i.e.\ there is no constant $C > 0$ such that 
\begin{equation*}
\norm{f}_{X} \leq C \norm{Af}_{Y}, \qquad f \in X.
\end{equation*}
This follows just by taking $f = \varphi_j$ and letting $j \to \infty$.

In many cases, one expects that rapid decay of singular values of $A: X \to Y$ implies strong ill-posedness for the corresponding inverse problem. However, the choice of the compact set $K$ also plays an important role. We remark that for \emph{instability results}, one often does not need to know the decay of the full sequence of singular values, but only the decay on a suitable subspace. This will be exploited in the instability mechanism based on microlocal smoothing, see Section \ref{sec:micro_smooth_1}. In contrast, for stability mechanisms lower bounds are needed for all singular values.

We will next describe three approaches to instability for linear inverse problems, related to different choices of the compact set $K$ and to smoothing properties/singular value decay of $A$. The first two approaches consider compact sets $K$ given in terms of some orthonormal basis $\varphi = (\varphi_j)$ of $X$ and an increasing sequence $\kappa = (\kappa_j)$ with $0 < \kappa_1 \leq \kappa_2 \leq \ldots \to \infty$. We define the set 
\begin{align}
\label{eq:ex_K}
K = K_{\kappa, \varphi} = \{ f \in X \ ;\  \norm{f}_{\kappa} := \left( \sum_{j=1}^{\infty} \kappa_j^2 |(f, \varphi_j) |^2 \right)^{1/2} \leq 1 \}.
\end{align}
This set is compact in $X$, since the inclusion map $j: (X, \norm{\,\cdot\,}_{\kappa}) \to X$ is easily shown to be approximated by finite rank operators. Moreover, this setup is general in the sense that any compact subset of $X$ is contained in some $K_{\kappa,\varphi}$, as proved in the next lemma:

\begin{lemma}
\label{lem:compact_sets}
Let $H$ be a separable Hilbert space, let $(e_j)_{j\in \N}$ be a fixed (but else arbitrary) orthonormal basis  of $H$, and let $K \subset H$. Then, $K$ is compact if and only if K is closed and there exists an increasing, positive sequence $(\kappa_j)_{j\in \N}$ with $\kappa_j \rightarrow \infty$ such that for any $u \in K$ 
\begin{align}
\label{eq:finite}
\sum\limits_{j=1}^{\infty}|(u,e_j)|^2 \kappa_j^2 \leq 1.
\end{align} 
\end{lemma}

\begin{remark}
\label{rmk:compact_caveat}
Although the characterization of compact sets obtained in Lemma \ref{lem:compact_sets} is very useful if the singular value bases $\{\varphi_j\}$, $\{\psi_j\}$ associated with an operator $A$ are known, in general, there is the obvious caveat, that these bases are not explicitly known (and a computation may not be feasible). Hence, the set $K_{\kappa,\varphi}$ may not have an easy characterization in terms of standard function spaces in general. 
\end{remark}

\begin{remark}
The condition \eqref{eq:ex_K} can also be viewed as a source type condition. Indeed, this corresponds to a source condition $f = \varphi(A^* A) v$ with $\| v \|_{X} \leq 1$, with the choice $\varphi(\sigma_j^2) = \kappa_j^{-1}$, where the numbers $\sigma_j$ correspond to the singular values of the operator $A$. We, for instance, refer to \cite{EHN} for Hölder source type conditions, to \cite{H00} for logarithmic ones and to \cite{MP03} for general ones. In many places, in this subsection, it would be possible to formulate the results in terms of such source conditions.
\end{remark}

\begin{proof}[Proof of Lemma \ref{lem:compact_sets}]
We first assume that $K$ is compact. Standard functional analysis arguments then in particular entail that $K$ is closed and bounded. It thus suffices to prove the existence of an increasing sequence $\{\kappa_j\}_{j\in \N}$ such that all elements $u\in K$ obey the summability condition \eqref{eq:finite}. However, compact sets in separable Hilbert spaces can be characterized by having equi-small tails (see \cite[Proposition 25, Section 12, Chapter 3]{Melrose}), i.e.\ for each $\eps>0$ there exists $N\in \N$ such that for any $u\in K$
\begin{align*}
\sum\limits_{j\geq N}|(u,e_j)|^2\leq \eps^2.
\end{align*}
From this we can however construct the desired sequence $\{\kappa_j\}_{j\in \N}$ (or rather one possible choice of this sequence). Indeed, considering a sequence $\eps_k:= \frac{1}{k^2}$, there exists a corresponding sequence $N_k \in \N$, $N_{k+1}> N_{k}$, such that for all $u\in K$
\begin{align*}
\sum\limits_{j=N_k}^{\infty}|(u, e_j)|^2 < \eps_k^2= \frac{1}{k^4}.
\end{align*}
In particular, we have that
\begin{align*}
\sum\limits_{k=1}^{\infty} \sum\limits_{j=N_k}^{N_{k+1}} k^2 |(u,e_j)|^2 \leq \sum\limits_{k=1}^{\infty} k^{-2} \leq C_1.
\end{align*}
Hence defining 
\begin{align*}
\kappa_j^2:=
\left\{ 
\begin{array}{ll}
& 1/C_1 \mbox{ if } j \in \{1,\dots, N_2\},\\
& l^2/C_1 \mbox{ if } j \in \{N_l, \dots, N_{l+1}\} \mbox{ for } l \geq 2,
\end{array}
\right.
\end{align*}
then yields a possible choice for the increasing sequence $\kappa_j$. We remark that the sequence $\kappa_j$ can also be chosen to be strictly increasing by gradually growing from $l$ to $l+1$ in the interval $j\in\{N_l, \dots, N_{l+1}\}$.

Next we argue that closedness and the growth condition \eqref{eq:finite} characterize compactness. To this end, we argue similarly as in the proof of \cite[Proposition 25]{Melrose}: By boundedness (which follows from \eqref{eq:finite}) and closedness of $K$, it suffices to show that any sequence $\{u_m \}_{m\in \N} \subset K$ has a convergent subsequence. However, by boundedness of the sequence $\{u_m\}_{m\in \N}$ we have that for any $k\in \N$ the sequence $(u_m, e_k)$ is bounded. Hence, for each $k$ there exists a subsequence $\{u_{m_k}\}_{m_k\in\N}$ such that $(u_{m_k}, e_k)$ is convergent. Moreover, we may assume that for each $k\in \N$ it holds $\{u_{m_k}\}_{m_k\in \N} \subset \{u_{m_{k-1}}\}_{m_{k-1}\in \N}$. Using a Cantor diagonal argument and setting $v_m:= u_{m_m}$ then yields a sequence such that $(v_m, e_k)$ converges for each $k\in \N$. We claim that $\{v_m\}_{m\in \N}$ is the desired convergent subsequence of $\{u_m\}_{m\in \N}$. Indeed, let $n,l \in\N$ to be fixed later. Then,
\begin{align*}
\|v_n - v_{n+l}\|^2 \leq \sum\limits_{k\leq N} |(v_n-v_{n+l}, e_k)|^2 + 2 \sum\limits_{k>N} |(v_n, e_k)|^2 + 2 \sum\limits_{k>N} |(v_{n+l}, e_k)|^2.
\end{align*}
By assumption, the condition \eqref{eq:finite} holds. In particular, we have that
$$\sum\limits_{k>N} |(v_n, e_k)|^2 + \sum\limits_{k>N} |(v_{n+l}, e_k)|^2 \leq 2 \kappa_N^{-2}\,.
$$
 Now, let $\eps>0$ be arbitary. Choosing $N\in \N$ so large that $\kappa_N^{-2}\leq \eps^2/2$ and choosing $n\in \N$ so large that for $k \in \{1, \dots, N\}$ it holds that $|(v_n-v_{n+l}, e_k)|\leq \frac{\eps^2}{2N}$ (which is possible by the convergence of $(v_n, e_k)$ for all $k\in \N$), we infer that
\begin{align*}
\|v_n - v_{n+l}\|^2 \leq \sum\limits_{k\leq N} |(v_n-v_{n+l}, e_k)|^2 + 4\kappa_N^{-2}
\leq \eps^2.
\end{align*}
This proves the desired convergence of the subsequence under consideration.
\end{proof}

\subsection{Compact sets described by a singular value basis}

In cases where $K = K_{\kappa,\varphi}$ and where $\varphi = (\varphi_j)$ happens to be a singular value basis for the operator $A$, the next result completely characterizes the stability properties of the inverse problem (for typical moduli of continuity) in terms of the sequences $(\sigma_j)$ and $(\kappa_j)$.

\begin{proposition} 
\label{lemma_holder_stability_abstract}
Let $A$, $\sigma_j$, $\varphi_j$ be as in Proposition \ref{prop_singular_value_decomposition}, and let $K = K_{\kappa,\varphi}$ for some sequence $(\kappa_j)$ with $0 < \kappa_1 \leq \kappa_2 \leq \ldots \to \infty$ as in \eqref{eq:ex_K}. Let also $\eta(t)$ be a strictly increasing concave function so that $\eta(t)/t$ is nonincreasing.

Then the inverse problem for $A$ in $K$ is $\eta$-stable in the sense that 
\begin{equation} \label{linear_eta_stable}
\norm{f}_X^2 \leq \eta(\norm{Af}_Y^2), \qquad f \in K,
\end{equation}
if and only if 
\begin{equation} \label{kappai_sigmai_relation}
\kappa_j^{-2} \leq \eta((\sigma_j/\kappa_j)^2), \qquad j \geq 1.
\end{equation}
\end{proposition}
\begin{proof}
If \eqref{linear_eta_stable} holds, then choosing $f = \kappa_j^{-1} \varphi_j \in K$ yields \eqref{kappai_sigmai_relation}. Conversely, assume that \eqref{kappai_sigmai_relation} holds for some strictly increasing concave function $\eta$. Then $\eta^{-1}$ is convex. Let $f = \sum a_j \varphi_j \in K$ where $\sum \kappa_j^2 a_j^2 = 1$, and define $c_j = \kappa_j^2 a_j^2$. Since $\sum c_j = 1$, Jensen's inequality and \eqref{kappai_sigmai_relation} yield 
\[
\eta^{-1}(\norm{f}_X^2) = \eta^{-1}(\sum a_j^2) = \eta^{-1}(\sum c_j \kappa_j^{-2}) \leq \sum c_j \eta^{-1}(\kappa_j^{-2}) \leq \sum \sigma_j^2 a_j^2 = \norm{Af}_Y^2.
\]
Thus \eqref{linear_eta_stable} holds for any $f \in K$ with $\sum \kappa_j^2 a_j^2 = 1$. If instead $\sum \kappa_j^2 a_j^2 = \lambda^2 \leq 1$, then \eqref{linear_eta_stable} holds for $f/\lambda$, so 
\[
\norm{f}_X^2 \leq \frac{\eta(\norm{Af}_Y^2/\lambda^2)}{1/\lambda^2} \leq \eta(\norm{Af}_Y^2)
\]
using that $\eta(t)/t$ is nonincreasing.
\end{proof}

\begin{example} \label{example_holder_abstract}
We can draw some immediate conclusions:
\begin{itemize}
\item 
The inverse problem for $A$ is $\alpha$-H\"older stable in $K$ (i.e.\ $\eta(t) = C t^{\alpha}$) if and only if $\kappa_j^{1-\alpha} \sigma_j^{\alpha} \geq c > 0$ for $j \geq 1$. In particular, no matter what the singular values $\sigma_j$ of $A$ are, if one chooses $\kappa_j = \sigma_j^{-s}$ for some $s > 0$ then the inverse problem for $A$ in $K$ will be H\"older stable. This is the case even if the singular values of $A$ are exponentially decaying, however then the set $K$ may be restrictively small (similar to a space of real-analytic functions). 
\item 
As a second possible scenario, we consider the case in which the singular values of $A$ satisfy $\sigma_j \leq e^{-cj^{\mu}}$, i.e.\ they are exponentially decaying. Further we choose \emph{$\kappa_j = j^{s}$} for some $s > 0$ (if $(\varphi_j)$ by chance also corresponds to an eigenbasis of some elliptic operator, as in Proposition \ref{prop_sobolev_sequence_space}, then this implies that $K$ is a bounded set in a Sobolev type space). In this case, Proposition \ref{lemma_holder_stability_abstract} implies that the inverse problem for $A$ in $K$ has at best a logarithmic modulus of continuity. Moreover, if one also has lower bounds $\sigma_j \geq e^{-c_1 j^{\mu}}$ for the singular values, the proposition implies that the inverse problem is indeed logarithmically stable.
\end{itemize}
\end{example}

The advantage of using a singular value basis to describe $K$ is the fact that the action of $A$ on an element $f = \sum a_j \varphi_j$ can be easily computed. However, the drawback of this approach is that the relation of the norm used in the definition of $K$ and ``standard norms" such as Sobolev norms is not obvious in general. For example, if $X=L^2(M)$ for a compact $n$-manifold $M$, then a more natural choice would be to consider $K = \{ f \in H^s(M) \,;\, \norm{f}_{H^s(M)} \leq 1 \}$ for some $s > 0$, which would correspond to \eqref{eq:ex_K} with $\kappa_j \sim j^{s/n}$ but where the basis $(\varphi_j)$ would consist of the eigenfunctions of the Laplacian on $M$. However, the action of $A$ is much harder to compute on a general basis $(\varphi_j)$ than for a singular value basis.

Analogously to Proposition \ref{prop_sobolev_sequence_space}, a special case where the singular value basis can be used to describe standard Sobolev spaces, so that Proposition \ref{lemma_holder_stability_abstract} applies, is the case where $A$ is an elliptic operator (see \cite[Section 7.10 in vol.\ II]{Taylor}). 

It may also be possible to apply Proposition \ref{lemma_holder_stability_abstract} in cases where there is very high symmetry. One such situation, known as ``Slepian's miracle", was discovered by  Landau, Pollak, Slepian, see \cite{SP61} for more information and references on this. However, we will not discuss these points further, since the other two approaches described below will give more general instability results.

\subsection{Compact sets described by a reference basis}

If one is only interested in an instability result, i.e.\ in showing that the inequality 
\[
\norm{f}_X \leq \omega(\norm{Af}_Y), \qquad f \in K,
\]
can only hold for certain moduli of continuity $\omega$, it is enough to exhibit elements $f \in K$ so that $\norm{Af}_Y$ is relatively small but $\norm{f}_X$ is not too small. The following result shows that this is possible whenever $K$ is of the form \eqref{eq:ex_K} where $(\varphi_j)$ can be any fixed orthonormal basis (for instance a basis used for defining Sobolev type spaces).

\begin{proposition} \label{lemma_ref_basis}
Let $\varphi = (\varphi_j)_{j=1}^{\infty}$ be a fixed orthonormal basis of $X$, and let $K = K_{\kappa,\varphi}$ as in \eqref{eq:ex_K} for some sequence $\kappa$ with $0 < \kappa_1 \leq \kappa_2 \leq \ldots$ and $\kappa_j \to \infty$. Define 
\[
N(\eps) = \sup \{ j \geq 1 \,;\, \kappa_j \leq 1/\eps \}.
\]
Let $A: X \to Y$ be a compact injective operator with singular values $\sigma_1 \geq \sigma_2 \geq \ldots$. If 
\[
\norm{f}_X \leq \omega(\norm{Af}_Y), \qquad f \in K,
\]
then 
\[
\omega(t) \geq g^{-1}(t), \qquad \text{$t$ small},
\]
where $g(\eps) := \eps \sigma_{N(\eps)}$ is a strictly decreasing function.
\end{proposition}
\begin{proof}
We first observe that if $W$ is any finite dimensional subspace of $X$, with dimension $\dim(W) = N$, then 
\begin{equation} \label{f_w_property}
\text{there is $f \in W \setminus \{ 0 \}$ with $\norm{Af}_Y \leq \sigma_N \norm{f}_X$}.
\end{equation}
To prove this, let $P_W$ be the orthogonal projection to $W$ and note that $Af = AP_{W} f$ for any $f \in W$. The operator $AP_W$ has at most $N$ nonzero singular values $\mu_1 \geq \ldots \geq \mu_N$, and they satisfy 
\[
\mu_j = \sigma_j(A P_W) \leq \sigma_j(A) \sigma_1(P_W) \leq \sigma_j(A).
\]
Thus, choosing $f$ to be a singular vector of $AP_W$ corresponding to $\mu_N$ yields \eqref{f_w_property}.

Next we want to find a finite dimensional subspace $W$ of $X$, with as large dimension as possible, so that 
\begin{equation} \label{f_kappa_bound}
\norm{f}_{\kappa} \leq \frac{1}{\eps} \norm{f}_X, \qquad f \in W.
\end{equation}
This is easily arranged by choosing $W = \{ \sum_{j=1}^N a_j \varphi_j \}$ where $N = N(\eps)$, since then 
\[
\norm{f}_{\kappa}^2 = \sum_{j=1}^N \kappa_j^2 \abs{(f,\varphi_j)}^2 \leq \frac{1}{\eps^2} \sum_{j=1}^N \abs{(f,\varphi_j)}^2 = \frac{1}{\eps^2} \norm{f}_X^2, \qquad f \in W.
\]

Now let $f$ be the vector in \eqref{f_w_property}, scaled so that it satisfies $\norm{f}_{X} = \eps$. Then \eqref{f_kappa_bound} gives that $\norm{f}_{\kappa} \leq 1$, i.e. $f \in K$, and \eqref{f_w_property} gives that $\norm{Af}_Y \leq \sigma_{N(\eps)} \norm{f}_X = \sigma_{N(\eps)} \eps$. The $\omega$-stability of the inverse problem gives that 
\[
\eps \leq \omega(\eps \sigma_{N(\eps)}),
\]
which proves the result since $g(\eps) = \eps \sigma_{N(\eps)}$ is strictly decreasing and hence has an inverse function.
\end{proof}

\begin{example} \label{example_ref_basis}
In order to use Lemma \ref{lemma_ref_basis}, one needs to know both upper bounds for $\kappa_j$ and a decay rate for the singular values $\sigma_j$. Here are a few examples:
\begin{itemize}
\item 
If $\kappa_j \sim j^s$ for some $s > 0$ (Sobolev type spaces) and $\sigma_j \lesssim j^{-m}$, then $N(\eps) \sim \eps^{-1/s}$ and $g(\eps) \lesssim \eps^{\frac{s+m}{s}}$. If the inverse problem for $A$ in $K$ is $\omega$-stable, Lemma \ref{lemma_ref_basis} yields that $\omega(t) \geq c t^{\frac{s}{s+m}}$ and thus $\alpha$-H\"older stability is only possible for $\alpha \leq \frac{s}{s+m}$.
\item 
If $\kappa_j \sim j^s$ for some $s > 0$ but the singular values of $A$ are exponentially decaying, i.e.\ $\sigma_j \lesssim e^{-cj^{\mu}}$ for some $c, \mu > 0$, then $N(\eps) \sim \eps^{-1/s}$ and 
\[
g(\eps) = \eps \sigma_{N(\eps)} \lesssim \eps e^{-c \eps^{-\mu/s}} \lesssim e^{-c_1 \eps^{-\mu/s}}
\]
for any $c_1 < c$. It follows that $\omega(t) \geq c \abs{\log\,t}^{-s/\mu}$, i.e.\ logaritmic stability is best possible in this setting.
\end{itemize}
\end{example}

The above examples confirm the expectation that rapidly decaying singular values imply strong illposedness, at least when the compact set $K$ is related to a standard Sobolev space.

\subsection{Compact sets described by an interpolation property}

Finally, we show instability results in cases where $A$ satisfies an abstract smoothing property, and $K$ is a bounded set in a space satisfying certain abstract interpolation estimates. The point is that in this setup, no information about the decay rate of the singular values of $A$ is needed.

\begin{proposition} \label{prop_abstract_instability_triple}
Let $V \subset E \subset H$ be infinite dimensional separable Hilbert spaces such that the inclusion $i: V \to H$ is compact, and one has the abstract interpolation inequality 
\[
\norm{u}_E \leq \norm{u}_H \, g(\norm{u}_V/\norm{u}_H), \qquad u \in V \setminus \{0\},
\]
for some $C^1$ function $g:[1,\infty) \to [0,\infty)$ with $g' > 0$. Let $A: H \to Y$ be bounded where $Y$ is a Banach space, and assume that $A$ is smoothing in the sense that there is a bounded operator $\bar{A}: V' \to Y$ extending $A$ so that 
\[
\bar{A}(\iota(u)) = A(u), \qquad u \in H,
\]
where $\iota: H \to V'$ is the natural inclusion with $\iota(u)(v) = (u, i(v))_H$ for $u \in H$ and $v \in V$.

Under these assumptions, if $A$ satisfies the stability inequality 
\[
\norm{u}_H \leq \omega(\norm{Au}_Y), \qquad \norm{u}_E \leq 1,
\]
for some modulus of continuity $\omega$, then there is a sequence $t_j \to 0$ with 
\[
\omega(t_j) \geq \frac{1}{g(1/f^{-1}(t_j))}
\]
where $f(\sigma) := \norm{\bar{A}} \sigma/g(1/\sigma)$.
\end{proposition}
\begin{proof} 
Since $i: V \to H$ is compact and injective, there are orthonormal sets $(\varphi_j) \subset V$ and $(\psi_j) \subset H$ such that 
\[
i(\varphi_j) = \sigma_j \psi_j, \qquad i^*(\psi_j) = \sigma_j \varphi_j,
\]
where the singular values $(\sigma_j)$ of $i$ satisfy $\sigma_1 \geq \sigma_2 \geq \ldots > 0$ and $\sigma_j \to 0$. The interpolation inequality implies 
\[
\norm{\psi_j}_E \leq g(\norm{\psi_j}_V) = g(\frac{1}{\sigma_j} \norm{\varphi_j}_V) = g(\frac{1}{\sigma_j}).
\]
We define 
\[
u_j := \frac{\psi_j}{g(1/\sigma_j)}.
\]
Then $\norm{u_j}_E \leq 1$.

Applying the stability inequality to $u_j$ and writing $C = \norm{\bar{A}}$, we get 
\[
\norm{u_j}_H \leq \omega(\norm{\bar{A}(\iota(u_j))}_Y) \leq \omega(C \norm{\iota(u_j)}_{V'}).
\]
Now $\norm{u_j}_H = \frac{1}{g(1/\sigma_j)}$ and 
\[
\norm{\iota(u_j)}_{V'} = \sup_{\norm{v}_V=1} (u_j, i(v))_H = \sup_{\norm{v}_V=1} \frac{1}{g(1/\sigma_j)} (i^*(\psi_j), v)_V \leq \frac{\sigma_j}{g(1/\sigma_j)}.
\]
Thus we have 
\[
\frac{1}{g(1/\sigma_j)} \leq \omega(C \sigma_j/g(1/\sigma_j)).
\]
Writing $t_j := f(\sigma_j)$, it follows that 
\[
\omega(t_j) \geq \frac{1}{g(1/f^{-1}(t_j))}
\]
provided that $f^{-1}$ is well defined. But the derivative of $f$ is 
\[
f'(\sigma) = C \left( \frac{1}{g(1/\sigma)} + \frac{g'(1/\sigma)}{\sigma g(1/\sigma)^2} \right) > 0,
\]
so $f$ is strictly increasing and $f^{-1}$ is well defined.
\end{proof}

\begin{example} \label{ex:interpol_triple}
To illustrate the proposition, we have the following results: 
\begin{itemize}
\item 
Let $V = X^{s+m}, E = X^{s+\delta}, H = X^s$ where $X^t$ are Sobolev type spaces satisfying the interpolation inequality 
\[
\norm{u}_{X^{s+\delta}} \leq \norm{u}_{X^s}^{1-\delta/m} \norm{u}_{X^{s+m}}^{\delta/m}
\]
such that $X^{s+m} \subset X^s$ is compact. Then one can take $g(t) = t^{\delta/m}$ and $f(\sigma) = \sigma^{1+\delta/m}$. The conclusion of Proposition \ref{prop_abstract_instability_triple} is that 
\[
\omega(t_j) \geq t_j^{\frac{\delta}{m+\delta}}, \qquad t_j \to 0.
\]
The fact that $A$ extends as a map $V' \to Y$, where the dual of $V = X^{s+m}$ is taken with respect to $H = X^s$, can be interpreted so that $A$ is smoothing of order $m$. In particular, if $A$ is smoothing of infinite order, no H\"older stability bound is possible.
\item 
Let $H = \ell^2 = \{ x = (x_1, x_2, \ldots) \}$, let $E$ be the Sobolev space $\norm{x}_E = (\sum j^{2s} x_j^2)^{1/2}$, and let $V$ be the real analytic type space 
\[
\norm{x}_V = ( \sum e^{2\beta j} x_j^2)^{1/2}
\]
where $s, \beta > 0$. The inclusion $V \subset H$ is compact. If $\sum x_j^2 = 1$, applying Jensen's inequality with convex function $h(t) = \exp(2 \beta t^{1/(2s)})$, which has inverse $h^{-1}(y) = ( \frac{1}{2\beta} \log(y) )^{2s}$, implies that 
\[
h(\sum j^{2s} x_j^2) \leq \sum x_j^2 h(j^{2s})
\]
and therefore 
\[
\norm{x}_E^2 \leq h^{-1}( \sum x_j^2 h(j^{2s}) ) = h^{-1}( \sum x_j^2 e^{2\beta j} ), \qquad \sum x_j^2 = 1.
\]
By homogeneity, this shows that 
\[
\norm{x}_E \leq \norm{x}_H ( \frac{1}{\beta} \log(\norm{x}_V/\norm{x}_H) )^{s}.
\]
Thus one can choose $g(t) = ( \frac{1}{\beta} \log(t) )^{s}$ and $f(\sigma) = \sigma ( \frac{1}{\beta} \log(1/\sigma) )^{-s}$. Since $f(\sigma)$ behaves almost like $\sigma$, the conclusion of Proposition \ref{prop_abstract_instability_triple} implies that $\omega(t)$ cannot be much better than $\abs{\log t}^{-s}$.
We remark that similar Jensen type inequalities had already been used in the 70s and 80s. As an non-exhaustive set of references we point to \cite{TV82,LV85a,LV85b,MP85} and the literature mentioned therein.
\end{itemize}
\end{example}

\subsection{Comparison between different instability results}

In the next example we compare the entropy based instability results from Section \ref{sec:abstract} to the alternative linear instability results in Appendix \ref{sec:abstr}. We consider an abstract inverse problem where the forward map $A = (-\Delta_g+1)^{-m/2}$ is a linear elliptic pseudodifferential operator on a closed manifold. It turns out that Proposition \ref{lemma_holder_stability_abstract} gives a sharp characterization for $\alpha$-H\"older stability and instability, whereas Propositions \ref{lemma_ref_basis} and \ref{prop_abstract_instability_triple} and Lemma \ref{lemma_entropy_typical} give the same almost sharp instability result. Thus at least in this example the entropy based methods of Section \ref{sec:abstract} give the optimal instability result except for missing the endpoint exponent.

Let $(M,g)$ be a closed connected smooth $n$-manifold, and let $(\varphi_j)_{j=0}^{\infty}$ be an orthonormal basis of $L^2(M)$ consisting of eigenfunctions of the Laplacian such that $-\Delta_g \varphi_j = \lambda_j \varphi_j$ with $\lambda_0 = 0$ and $\lambda_j \sim j^{2/n}$ for $j \geq 1$. Let $m > 0$ and define 
\[
A: L^2(M) \to L^2(M), \ \ A u = (-\Delta_g+1)^{-m/2} u = \sum_{j=0}^{\infty} (1+\lambda_j)^{-m/2} (u,\varphi_j) \varphi_j.
\]
Then $A$ has singular values $\sigma_j = (1+\lambda_{j-1})^{-m/2} \sim \br{j}^{-m/n}$ by Weyl asymptotics (Theorem \ref{thm_weyl3}), and the entropy numbers satisfy $e_j \sim \br{j}^{-m/n}$ using Lemma \ref{lemma_singular_entropy_relation}. Consider a compact set 
\[
K = \{ u \in H^l(M) \,;\, \norm{u}_{H^l(M)} \leq 1 \}.
\]
Then $K = K_{\varphi,\kappa}$ with $\kappa_j = (1+\lambda_j)^{l/2} \sim \br{j}^{l/n}$.

We can say the following about the $\alpha$-H\"older stability for the inverse problem for $A$ in $K$:

\begin{itemize}
\item 
By Proposition \ref{lemma_holder_stability_abstract} (see Example \ref{example_holder_abstract}) the problem is $\alpha$-H\"older stable if and only if $l(1-\alpha)-m\alpha > 0$, i.e.\ $\alpha < \frac{l}{l+m}$.
\item 
By Proposition \ref{lemma_ref_basis} (see Example \ref{example_ref_basis}) the problem can only be $\alpha$-H\"older stable if $\alpha \leq \frac{l}{l+m}$.
\item 
By Proposition \ref{prop_abstract_instability_triple} (see Example \ref{ex:interpol_triple}) the problem can only be $\alpha$-H\"older stable if $\alpha \leq \frac{l}{l+m}$.
\item 
By Lemma \ref{lemma_entropy_typical}, since $A(K)$ lies in a bounded subset of $H^{l+m}$, the problem can only be $\alpha$-H\"older stable if $\alpha \leq \frac{l}{l+m}$.
\end{itemize}

\section{Proofs of some results in Section \ref{sec:pre}} \label{sec_appendix_nonlinear}

In this section we present the proofs of some statements from Section \ref{sec:pre}. We begin with a lemma related to Gevrey spaces.

\begin{lemma} \label{lemma_arho_comega}
Let $(M,g)$ be a closed smooth $n$-manifold, let $1 \leq \sigma < \infty$ and $\rho > 0$.
\begin{enumerate}
\item[(a)] 
$u \in G^{\sigma}(M)$ if and only if there are $C, R > 0$ such that 
\begin{equation} \label{gevrey_bounds_ltwo}
\norm{(-\Delta_g)^t u}_{L^2} \leq C R^{2t} (2t)^{2t\sigma}, \qquad t \geq 0.
\end{equation}
\item[(b)] 
If $u \in A^{\sigma,\rho}(M)$, then $u \in G^{\sigma}(M)$ and 
\[
\norm{(-\Delta_g)^t u}_{L^2} \leq \norm{u}_{A^{\sigma,\rho}(M)} R^{2t} (2t)^{2t\sigma}, \qquad t \geq 0,
\]
for any $R \geq C_0 (\sigma/\rho)^{\sigma}$ where $C_0 > 0$ only depends on $(M,g)$.
\item[(c)] 
If $u \in G^{\sigma}(M)$ satisfies \eqref{gevrey_bounds_ltwo}, then $u \in A^{\sigma,\rho}(M)$ for any $\rho \leq c_0 R^{-1/\sigma}$ and 
\[
\norm{u}_{A^{\sigma,\rho}(M)} \leq c_1 C
\]
where $c_0, c_1 > 0$ depend on $(M,g)$, $n$ and $\sigma$, and $c_1$ also depends on $R$.
\end{enumerate}
\end{lemma}
\begin{proof}
(a) Let first $k \geq 0$ and $0 \leq l \leq k-1$. Recall that $\nabla$ denotes the covariant derivative induced by the Levi-Civita connection $(M,g)$ and that $\abs{\nabla^l u}$ denotes the $g$-norm of the $l$-tensor field $\nabla^l u$, see e.g.\ \cite{Besse}. Using the facts that $\Delta v = \tr_g \nabla^2 v$ and that $\tr_g$ (the $g$-trace with respect to given two indices) commutes with the covariant derivative $\nabla$, we have 
\[
\abs{\nabla^l \Delta^{k-l} u} = \abs{\nabla^l \tr_g \nabla^2 \Delta^{k-l-1} u} = \abs{\tr_g \nabla^{l+2} \Delta^{k-l-1} u} \leq n^{1/2} \abs{\nabla^{l+2} \Delta^{k-l-1} u}.
\]
In the last step we used Cauchy-Schwarz. Using this repeatedly gives 
\[
\abs{\Delta^k u} \leq n^{1/2} \abs{\nabla^2 \Delta^{k-1} u} \leq \ldots \leq n^{k/2} \abs{\nabla^{2k} u}.
\]
Now if $u \in G^{\sigma}(M)$ with $\norm{\nabla^k u}_{L^{\infty}(M)} \leq C_0 R_0^k k^{\sigma k}$ for $k \geq 0$, the previous inequality gives  
\[
\norm{\Delta^k u}_{L^2} \leq C_0 \,\mathrm{Vol}_g(M)^{1/2} (n^{1/4} R_0)^{2k} (2k)^{2\sigma k}, \qquad k \geq 0.
\]
This proves \eqref{gevrey_bounds_ltwo} when $t$ is an integer. If $0 \leq t \leq 1$ then \eqref{gevrey_bounds_ltwo} is trivial, and if $t \geq 1$ then $t = k\theta$ for some $k \in \mZ_+$ and $\theta \geq 1/2$, so that 
\[
\norm{\Delta^t u}_{L^2} \leq \norm{u}_{L^2}^{1-\theta} \norm{\Delta^k u}_{L^2}^{\theta} \leq C R^t (2k)^{2\sigma t} \leq C (R\theta^{-2\sigma})^t (2t)^{2\sigma t}.
\]
We have proved \eqref{gevrey_bounds_ltwo} for all $t \geq 0$. Conversely, if \eqref{gevrey_bounds_ltwo} holds then by Sobolev embedding and elliptic regularity 
\[
\norm{\nabla^{2k} u}_{L^{\infty}} \leq C_0 \norm{u}_{H^{2k+2n}} \leq C_0 C_1^k(\norm{u}_{L^2} + \norm{(-\Delta_g)^{k+n} u}_{L^2})
\]
where $C_0, C_1 > 0$ only depend on $(M,g)$ and $n$. Thus one gets $\norm{\nabla^{2k} u}_{L^{\infty}} \leq \tilde{C} \tilde{R}^{2k} (2k)^{2\sigma k}$ for some $\tilde{C}, \tilde{R} > 0$, showing that $u \in G^{\sigma}(M)$.

(b) Let $u \in A^{\sigma, \rho}(M)$. Then $u \in C^{\infty}(M)$, and for $t \geq 0$ 
\[
(-\Delta_g)^t u = \sum_{j=0}^{\infty} \lambda_j^t (u,\varphi_j) \varphi_j.
\]
Consequently,  using the Weyl asymptotics $\lambda_j \leq C_0 j^{2/n}$ (Theorem \ref{thm_weyl1}), 
\[
\norm{(-\Delta_g)^t u}_{L^2}^2 = \sum_{j=0}^{\infty} \lambda_j^{2t} \abs{(u,\varphi_j)}^2 \leq C_0^{2t} (\sup_{j \geq 0} j^{2t/n} e^{-\rho j^{1/(n\sigma)}})^2 \norm{u}_{A^{\sigma,\rho}}^2.
\]
The function $f(s) = s^{2t/n} e^{-\rho s^{1/(n\sigma)}}$ obtains its maximum over $s \geq 0$ when $s = (2t\sigma/\rho)^{n\sigma}$. Thus we obtain 
\[
\norm{(-\Delta_g)^t u}_{L^2} \leq C_0^t (2t\sigma/\rho)^{2t\sigma} e^{-2t\sigma} \norm{u}_{A^{\sigma,\rho}}.
\]

(c) Let $u \in G^{\sigma}(M)$ satisfy \eqref{gevrey_bounds_ltwo}. We observe that for any $t \geq 0$ 
\[
(u, \varphi_j)_{L^2} = \lambda_j^{-t} (u, (-\Delta_g)^t \varphi_j)_{L^2} = \lambda_j^{-t} ( (-\Delta_g)^t u, \varphi_j)_{L^2}.
\]
Using \eqref{gevrey_bounds_ltwo}, the Weyl asymptotics $\lambda_j \geq c_0 j^{2/n}$ and the normalization $\norm{\varphi_j}_{L^2} = 1$, we get 
\[
\abs{(u,\varphi_j)}_{L^2} \leq (c_0 j^{2/n})^{-t} C R^{2t} (2t)^{2t\sigma}.
\]
The function $f(s) = (\frac{R}{c_0^{1/2} j^{1/n}})^s s^{s\sigma}$ is minimal over $s \geq 0$ when $s = \frac{1}{e} (\frac{c_0^{1/2} j^{1/n}}{R})^{1/\sigma}$. Thus choosing $t = t_0 := \frac{1}{2e} (\frac{c_0^{1/2} j^{1/n}}{R})^{1/\sigma}$ gives 
\[
\abs{(u,\varphi_j)}_{L^2} \leq C e^{-2\sigma t_0} \leq C e^{-cj^{1/(n\sigma)}}
\]
with $c = \frac{\sigma}{e} (c_0)^{\frac{1}{2\sigma}} R^{-1/\sigma}$. Thus if $\rho \leq c/2$, one has 
\[
\norm{u}_{A^{\sigma,\rho}(M)}^2 \leq C^2 \sum_{j=0}^{\infty} e^{2\rho j^{1/(n\sigma)}} e^{-2cj^{1/(n\sigma)}} \leq C^2 \sum_{j=0}^{\infty} e^{-cj^{1/(n\sigma)}} = C^2 c_1^2
\]
where $c_1$ only depends on $(M,g)$, $n$, $\sigma$ and $R$.
\end{proof}

As the final part of this section, we present a proof of the Weyl law in Theorem \ref{thm_weyl3} using semiclassical calculus following \cite[Section 14.3]{Zworski}. 

\begin{proof}[Proof of Theorem \ref{thm_weyl3}]
By passing to $A^* A$ we may assume that $A$ is self-adjoint and nonnegative. Let $\lambda_1 \geq \lambda_2 \geq \ldots \geq 0$ be the eigenvalues of $A$ on $L^2(M)$. We wish to count the number of eigenvalues that are $\geq \eps = h^m$, where $h > 0$ is a semiclassical parameter. In order to do this we will relate $A$ to the semiclassical $\Psi$DO $h^m \mathrm{Op}_h(a)$ where $\mathrm{Op}_h$ denotes Weyl quantization, and estimate the eigenvalues of the latter by using semiclassical analysis.

First note that the operator $A$ is a classical self-adjoint $\Psi$DO of order $-m$. Thus there is $\tilde{a} \in C^{\infty}(T^* M \setminus \{0\})$, real and homogeneous of degree $-m$ in $\xi$,  so that 
\[
a(x,\xi) := (1-\psi(x,\xi)) \tilde{a}(x,\xi)
\]
is a principal symbol for $A$. Here $\psi(x,\xi) = \tilde{\psi}(\abs{\xi}_{g(x)})$ where $\tilde{\psi} \in C^{\infty}_c(\R)$ is a function with $0 \leq \tilde{\psi} \leq 1$, $\tilde{\psi} = 1$ near $0$, and $\tilde{\psi}(t) = 0$ for $\abs{t} \geq 1$. We now use the fact that $\mathrm{Op}_h(b(x,\xi)) = \mathrm{Op}(b(x, h\xi))$ together with the homogeneity of $\tilde{a}$ and support properties of $\psi$, which yields that for $h > 0$ small 
\[
h^m \mathrm{Op}_h(a) = \mathrm{Op}((1-\psi(x,h\xi)) \tilde{a}(x,\xi)) = \mathrm{Op}((1-\psi(x,h\xi))a(x,\xi)).
\]
This implies that 
\begin{equation} \label{a_hmopha_relation}
h^m \mathrm{Op}_h(a) = A - A \mathrm{Op}_h(\psi) + R
\end{equation}
where $R$ is a $\Psi$DO in $\Psi^{-m-1}$, which depends on $h$ but has uniform bounds with respect to $h$ small. The equation \eqref{a_hmopha_relation} is the desired relation between $A$ and $\mathrm{Op}_h(a)$.

Next let $B$ be any semiclassical $\Psi$DO on $M$ of negative order, so that $B$ is compact on $L^2(M)$. We wish to estimate the counting function for the singular values $\sigma_j(B)$. By passing to $B^* B$, we may assume that $B$ is self-adjoint and nonnegative and it is enough to study its eigenvalues. Given any $\chi \in C^{\infty}_c(\mR)$, as in \cite[Section 14.3]{Zworski} there is a semiclassical $\Psi$DO $\chi(B)$ with eigenvalues $\chi(\lambda_j(B))$ (as an operator on $L^2(M)$) and principal symbol $\chi(b)$, where $b$ is a principal symbol for $B$. The Weyl law follows by computing the trace of $\chi(B)$ in two ways. First, by functional calculus one has 
\[
\mathrm{Tr}(\chi(B)) = \sum_{j=1}^{\infty} \chi(\lambda_j(B)).
\]
On the other hand, by expressing the trace in terms of the Schwartz kernel of $\chi(B)$ which has principal symbol $\chi(b)$, one has (see \cite[proof of Theorem 15.3]{Zworski}) 
\[
\mathrm{Tr}(\chi(B)) = (2\pi h)^{-n} \int_{T^* M} \chi(b(x,\xi)) \,dV_{T^* M} + O(h^{1-n}).
\]
Combining the previous two identities, we obtain the Weyl law 
\begin{equation} \label{weyl_first}
\sum_{j=1}^{\infty} \chi(\lambda_j(B))= (2\pi h)^{-n} \int_{T^* M} \chi(b(x,\xi)) \,dV_{T^* M} + O(h^{1-n})
\end{equation}
where the implied constant depends on $\chi$.

Finally we invoke the noncharacteristic assumption for $A$, which ensures that $\abs{a(x,\xi)} \geq c \abs{\xi}^{-m}$ for $\abs{\xi} \geq 1$ in a conic neighborhood of some $(x_0, \xi_0)$. Since $\abs{a(x,\xi)} \leq C \abs{\xi}^{-m}$ for $\abs{\xi} \geq 1$, we have that $V(\eps) := (2\pi)^{-n} \int_{\{\abs{a} \geq \eps \}}$ is finite and satisfies $V(\eps) \sim \eps^{-n/m}$ as $\eps \to 0$. Given $\eps$, choose $\chi = \chi_{\eps} \in C^{\infty}_c(\mR)$ so that $0 \leq \chi \leq 1$, $\chi(t) = 1$ for $2\eps \leq t \leq \norm{\mathrm{Op}_h(a)}$, and $\chi(t) = 0$ for $t \leq \eps$. It follows from \eqref{weyl_first} that 
\[
\# \{ k \,;\, \sigma_k(\mathrm{Op}_h(a)) \geq \eps \} \geq V(2 \eps) h^{-n} + O(h^{1-n})
\]
with implied constant depending on $\eps$. Moreover, applying \eqref{weyl_first} to $\mathrm{Op}_h(\psi)$, we see that 
\[
\# \{ k \,;\, \sigma_k(\mathrm{Op}_h(\psi)) > 0 \} \sim h^{-n}.
\]
In particular $\sigma_k(\mathrm{Op}_h(\psi)) = 0$ for $k \geq C_1 h^{-n}$. We now apply the Weyl inequality for singular values to the identity \eqref{a_hmopha_relation}, which gives 
\begin{equation} \label{sigmajkl_inequality}
\sigma_{j+k+l}(h^m \mathrm{Op}_h(a)) \leq \sigma_j(A) + \sigma_k(A \mathrm{Op}_h(\psi)) + \sigma_l(R).
\end{equation}
The middle term on the right is zero if $k \geq C_1 h^{-n}$, and the last term is $\leq C_2 l^{-\frac{m+1}{n}}$ by Theorem \ref{thm:eigenval_bds}.

We now start fixing the various parameters. Fix $\eps_0 > 0$ so that $V(2 \eps_0)/2 - 4 C_1 \geq 1$, and choose $h_0 > 0$ so small that for $0 < h \leq h_0$ one has $C_2 (C_1)^{-\frac{m+1}{n}} h \leq \eps_0/2$ and 
\[
\# \{ k \,;\, \sigma_k(\mathrm{Op}_h(a)) \geq \eps_0 \} \geq \frac{1}{2} V(2 \eps_0) h^{-n}.
\]
Also choose $k$ with $C_1 h^{-n} < k \leq 2 C_1 h^{-n}$, and take $l = k$. We obtain from \eqref{sigmajkl_inequality} that 
\[
\sigma_{j+2k}(h^m \mathrm{Op}_h(a)) \leq \sigma_j(A) + C_2 k^{-\frac{m+1}{n}} \leq \sigma_j(A) + C_2 (C_1)^{-\frac{m+1}{n}} h^{m+1}.
\]
It follows that for $0 < h \leq h_0$, one has 
\begin{align*}
\# \{ j \,;\, \sigma_j(A) \geq (\eps_0/2) h^m \} &\geq \# \{ j \,;\, \sigma_{j+2k}(\mathrm{Op}_h(a)) - C_2 (C_1)^{-\frac{m+1}{n}} h \geq \eps_0/2 \} \\
 &\geq \# \{ j \,;\, \sigma_{j+2k}(\mathrm{Op}_h(a)) \geq \eps_0 \} \\
 &= \# \{ j \,;\, \sigma_j(\mathrm{Op}_h(a)) \geq \eps_0 \} -2k \\
 &\geq \frac{1}{2} V(2 \eps_0) h^{-n} - 4 C_1 h^{-n} \\
 &\geq h^{-n}.
\end{align*}
This implies that $\# \{ k \,;\, \sigma_k(A) \geq \eps \} \gtrsim \eps^{-n/m}$. Writing \eqref{a_hmopha_relation} as $A = h^m \mathrm{Op}_h(a) + A \mathrm{Op}_h(\psi) - R$ and arguing similarly as above, we also obtain $\# \{ k \,;\, \sigma_k(A) \geq \eps \} \lesssim \eps^{-n/m}$. This proves that $\sigma_k(A) \sim k^{-m/n}$.
\end{proof}

\section{Carleman estimates and H{\"o}lder instability of interior UCP}
\label{sec:Carl}

In this section, we prove the optimality of H{\"o}lder stability estimates for the problem of interior unique continuation for uniformly elliptic equations with Lipschitz regular metrics.
Further, as a complementary result to the exponential decay estimates from the previous sections in the main part of the text, we show that Carleman estimates can be used to infer lower bounds on the singular values. Combined with corresponding upper bounds (which can be inferred through regularity of the underlying operator as illustrated above), this can lead to (up to constants) optimal bounds.

\subsection{Instability of unique continuation in the interior}
\label{sec:UCP_int}

In this section we discuss the instability of the unique continuation property from the interior for (Lipschitz) continuous metrics. In this context, for comparison, the positive statement reads as follows:

\begin{proposition}[Theorem 1.7 in \cite{AlessandriniRondiRossetVessella}]
\label{prop:three_spheres}
Let $a^{ij}\in C^{0,1}(B_2, \R^{n \times n}_{+})$ be a uniformly elliptic tensor field and $u: B_2 \rightarrow \R$ a solution to
\begin{align*}
\p_i a^{ij} \p_j u = 0 \mbox{ in } B_2.
\end{align*}
Then, for $0<r_1 < r < r_2 \leq 1$ there exist constants $C>0$, $\alpha \in (0,1)$ which only depend on the ellipticity constants of $a^{ij}$, the dimension $n$ and the quotients $\frac{r}{r_1}, \frac{r}{r_2}$ such that 
\begin{align}
\label{eq:three_balls}
\| u \|_{L^2(B_r)} \leq C \|u\|_{L^2(B_{r_1})}^{\alpha} \|u\|_{L^2(B_{r_2})}^{1-\alpha}.
\end{align}
\end{proposition}

As one can, for instance, note by considering an expansion into spherical harmonics, the estimate \eqref{eq:three_balls} is clearly optimal for the Laplacian $a^{ij}=\delta_{ij}$ with $\alpha = \frac{\log\left( \frac{r_2}{r} \right)}{\log\left( \frac{r_2}{r_1} \right)}$. However, as in the previous sections, one may ask whether the modulus of continuity given by $\omega(t) = C t^{\alpha}$ with $\alpha \in (0,1)$ is optimal for \emph{general} metrics $a^{ij}$. 

For simplicity, (but essentially without loss of generality), in the sequel, we only consider the case in which $r_1 = \frac{r}{2}$ and $r_2 = 2r$.
In the context of our instability arguments, one main source of interest in the stability question for unique continuation in the \emph{interior} stems from the fact that the restrictions
\begin{align*}
L^2(B_{2 r}) \ni u \mapsto (u|_{B_{\frac{r}{2}}}, u|_{B_{r}}) \in L^2(B_{\frac{r}{2}}) \times L^2(B_r) ,
\end{align*}
are both strongly compressing (even for metrics $a^{ij}$ which are only bounded), and yet the stability estimate in Proposition \ref{prop:three_spheres} is of \emph{H{\"o}lder} and \emph{not} of \emph{logarithmic} type. This is due to the fact that in the estimate \eqref{eq:three_balls} there are compensation effects between the estimates in the different balls. Hence, it will turn out that in contrast to the \emph{logarithmic} bounds in the UCP up to the boundary, in the interior uniqueness setting the \emph{H{\"o}lder} estimates are indeed optimal. In particular, proving this will (at least indirectly) require both \emph{upper and lower} singular value bounds.

In order to simplify our notation, in the sequel, we will make the following normalization assumptions:
\begin{itemize}
\item[(A1)] the metric $a^{ij}$ is uniformly elliptic with $a^{ij}(0)= \delta^{ij}$,
\item[(A2)] $a^{ij}\in C^{0,1}(\R^n, \R^{n\times n})$ with $[a^{ij}]_{C^{0,1}}\leq \mu$ for $\mu \in (0,1)$ sufficiently small.
\end{itemize}
The first condition can always be satisfied by scaling and an affine change of coordinates. The second assumption could be weakened and is mainly used in the Carleman estimate, see \cite{KochTataru01}.

Under these conditions, we obtain the following instability result:

\begin{theorem}
\label{prop:three_spheres_opt}
Let $a^{ij}$ satisfy the conditions (A1), (A2) and let $u:B_4 \rightarrow \R$ be a solution to
\begin{align}
\label{eq:elliptic}
\p_i a^{ij} \p_j u & = 0 \mbox{ in } B_4,
\end{align}
where $a^{ij}\in C^{0,1}(B_4, \R^{n\times n}_+)$ is uniformly elliptic. Suppose that for some modulus of continuity $\omega: \R_+ \rightarrow \R_+$
\begin{align}
\label{eq:stab_Carl}
\frac{\|u\|_{L^2(B_{\frac{1}{4}})}}{\|u\|_{L^2(B_{\frac{1}{2}})}} \leq C\omega \left(\frac{\|u\|_{L^2(B_{\frac{1}{8}})}}{\|u\|_{L^2(B_{\frac{1}{2}})}} \right).
\end{align}
For each $\nu > 0$ there exists a small constant $\mu>0$ such that if the condition (A2) holds with this choice of $\mu$, then $\omega(t)\geq C_{\nu,n} t^{\frac{1}{2}+\nu}$.
\end{theorem}

In order to prove these optimality estimates, using that by (A1) we have $a^{ij}(0)=\delta^{ij}$, we construct solutions $u$ to $\p_i a^{ij} \p_j u = 0$ of the form 
\begin{align*}
u_{\ell}(x) = |x|^{\ell} H_{\ell}(\frac{x}{|x|}) + R_{\ell}(x),
\end{align*} 
where $H_{\ell}(\frac{x}{|x|})$ denotes a spherical harmonic of degree $\ell$ and $R_{\ell}$ is an error contribution which decays sufficiently fast.

In order to implement the explained strategy, we first show that any spherical harmonic can be achieved as a blow-up of a suitable solution to our elliptic equation:

\begin{proposition}
\label{prop:existence}
Let $a^{ij}$ satisfies the conditions (A1), (A2) and let $\ell \in \N$ and $\delta \in (0,\frac{1}{2})$. Then, there exists $C_{\delta}>0$ such that for any $\ell \geq 1$ there is a solution $u_{\ell}(x) = |x|^{\ell} H_{\ell}(\frac{x}{|x|}) + R_{\ell}(x)$ to $\p_i a^{ij} \p_j u = 0$ in $B_{1}$ with
\begin{align*}
\|R_{\ell}\|_{L^2(B_r)} \leq C \mu r^{n + \ell + 1-\delta} \mbox{ for any } r \in (0,3/4).
\end{align*}
\end{proposition}

We prove this result by duality to a uniqueness result which follows as a consequence of a suitable $L^2$ Carleman estimate. We begin by stating this estimate. A reference for this type of result can for instance be found in \cite[Theorem 5]{KochTataru01}. Instead of the weight from \cite[Section 6]{KochTataru01} here use a (slightly convexified) radial perturbation of $\tilde{\varphi}(x):=\log|x|$.  Apart from this modification of the leading part of the Carleman weight, the construction of $\varphi$ still essentially follows as in \cite[Lemma 6.1]{KochTataru01}.

\begin{proposition}[\cite{KochTataru01}]
\label{prop:Carl_three_balls}
Let $a^{ij}$ be as in Theorem \ref{prop:three_spheres_opt}. Let $u \in C^{2}_0(B_4)$ be a solution to
\begin{align*}
\p_i a^{ij} \p_j u = f  \mbox{ in } \R^n,
\end{align*}
where $f\in L^2(\R^n)$.
Then, there exists $\varphi(x)=\psi(|x|)$ which is comparable to $\log(|x|)$,
$\tau_0>0$ and $C>0$ such that for all $\tau \geq \tau_0$ and all $f$
we have
\begin{align}
\label{eq:Carl}
\begin{split}
&\tau \|e^{\tau \varphi}|x|^{-1}  u\|_{L^2(\R^n)} +  \|e^{\tau \varphi} \nabla u\|_{L^2(\R^n)}
\leq C  \|e^{\tau \varphi}|x| f \|_{L^2(\R^n)}.
\end{split}
\end{align}
\end{proposition}

We remark that the estimates could be strengthened using convexity/concavity (depending on the sign of $\tau$). As this does not provide substantially better estimates for our application, we do not address this further.

Relying on this estimate, we now address the proof of Proposition \ref{prop:existence}.

\begin{proof}[Proof of Proposition \ref{prop:existence}]

We prove existence by duality to the uniqueness result from the Carleman estimate.

\emph{Step 1: Setting.}
We begin by defining the following functional analytic set-up: For $\Omega \subset \R^n$ and $\tau \in \R$, $\tau \geq \tau_0$, fixed, we set
\begin{align*}
\|u\|_{L^2_{\tau}(\Omega)} &:= \|e^{\tau \varphi}|x|^{-1} u\|_{L^2(\Omega)},\\
\|u\|_{H^1_{\tau}(\Omega)} &:= \|e^{\tau \varphi}|x|^{-1} u\|_{L^2(\Omega)} + \tau^{-1}\|e^{\tau \varphi} \nabla u\|_{L^2(\Omega)}.
\end{align*}
We define $H^1_{\tau,0}(\Omega):= \overline{C_0^{\infty}(\Omega\setminus \{0\})}^{H^1_{\tau}(\Omega)}$.
Both $H^1_{\tau}(\Omega)$ and $H^{1}_{\tau,0}(\Omega)$ are Banach spaces. 

We now observe that the sought for function $R_{\ell}$ is supposed to satisfy the equation
\begin{align*}
\p_i a^{ij}(x) \p_j R_{\ell}(x) = - \p_i (a^{ij}(x)-\delta^{ij})\p_i (|x|^{\ell}H_{\ell}(\frac{x}{|x|})) \mbox{ in } B_{1},
\end{align*} 
and the bound
\begin{align*}
\|R_{\ell}\|_{L^2(B_{r})} \leq C r^{n+\ell + 1-\delta} \mbox{ for } r \in (0,1) \mbox{ and } \delta>0.
\end{align*}
We use the Carleman estimate from Proposition \ref{prop:Carl_three_balls} in order to construct such a function.

\medskip

\emph{Step 2: Duality argument.}

In the functional set-up from above, the Carleman estimate from Proposition \ref{prop:Carl_three_balls} then reads
\begin{align}
\label{eq:Carl_aux}
\tau \|u\|_{L^2_{\tau}(\R^n)} +  \|u\|_{H^1_{\tau}(\R^n)}\leq C  \||x| \p_i a^{ij} \p_j u\|_{L^2_{\tau}(\R^n)}.
\end{align}

Abbreviating $L:= \p_i a^{ij} \p_j$, we now define the map
\begin{align*}
T:(L^2_{\tau}(B_{1})) \supset |x| L C_0^{\infty}(B_{1}\setminus \{0\}) & \rightarrow \R, \\ 
|x| L u &\mapsto -(\p_i u, (a^{ij}-\delta^{ij})\p_i (|x|^{\ell}H_{\ell}(\frac{x}{|x|})))_{L^2(B_{1}))}.
\end{align*}
Due to the Carleman estimate \eqref{eq:Carl_aux}, this is well-defined. Further, the map defines a bounded functional on $L^2_{\tau}(B_1(0))$:
\begin{align*}
&|(\p_i u, (a^{ij}-\delta^{ij}) \p_i (|x|^{\ell}H_{\ell}(\frac{x}{|x|}))_{L^2(B_{1}))}| \\
&\leq 
\|u\|_{H^1_{\tau}(B_{1})} \|e^{-\tau \varphi} (a^{ij}-\delta^{ij}) \p_i (|x|^{\ell}H_{\ell}(\frac{x}{|x|}))\|_{L^2(B_{1})}\\
&\leq C \|e^{\tau \varphi} |x|  L u\|_{L^2(B_1)} \|e^{-\tau \varphi} (a^{ij}-\delta^{ij}) \p_i (|x|^{\ell}H_{\ell}(\frac{x}{|x|}))\|_{L^2(B_{1})}.
\end{align*}
Hence, as a map on the vector space $U:=|x| L C_0^{\infty}(B_{1}\setminus \{0\}) \subset (L^2_{\tau}(B_1))$, the map $T$ is bounded with 
\begin{align}
\label{eq:functional_bound1}
\|T\|_{U^{\ast}} \leq C \| e^{- \tau \varphi} (a^{ij}-\delta^{ij}) \p_i (|x|^{\ell}H_{\ell}(\frac{x}{|x|})) \|_{L^2(B_{1})}.
\end{align}
Below, we will choose $\tau \geq 0$ such that the right hand side of \eqref{eq:functional_bound1} is bounded.
Now by the Hahn-Banach theorem, $T$ can be extended to a functional on $L^2_{\tau}(B_1)$ without increasing its operator norm. Thus, by the Riesz representation theorem there exists an element $\tilde{R} \in  L^2_{\tau}(B_{1})$ such that 
\begin{align*}
\|T\|_{(L^2_{\tau}(B_{1}))^{\ast }} = \|\tilde{R}\|_{L^2_{\tau}(B_{1})} ,
\end{align*}
and
\begin{align*}
T(|x|L\psi) = (\tilde{R}, |x| L\psi)_{L^2_{\tau}(B_1)} \mbox{ for all }  \psi \in C_0^{\infty}(B_1\setminus \{0\}).
\end{align*}

As a consequence, by definition of the functional $T$, for all $u\in C_0^{\infty}(B_{1}\setminus \{0\})$,
\begin{align*}
& - (\p_j u, (a^{ij}-\delta^{ij}) \p_i(|\cdot|^{\ell } H_{\ell}(\frac{x}{|x|})))_{L^2(B_{1})}
 =
T L u
= -(\p_j ( e^{2\tau} \tilde{R}),  a^{ij } \p_i u)_{L^2(B_{1})}.
\end{align*}
As a consequence, the function $R (x):= e^{2\tau \varphi(x)} \tilde{R}(x)$
solves the equation
\begin{align*}
\p_i a^{ij}(x)\p_j R(x) = \p_i (a^{ij}-\delta^{ij}) \p_i(|x|^{\ell} H_{\ell}(\frac{x}{|x|})) \mbox{ in } B_{1}.
\end{align*}
Moreover, by the estimates \eqref{eq:functional_bound1}, we have 
\begin{align*}
&\tau \|e^{-\tau \varphi}|x|^{-1} R\|_{L^2(B_{r})} \leq \| \tilde{R}\|_{L^2_{\tau}(B_{r})}  \leq \| \tilde{R}\|_{L^2_{\tau}(B_{1})} \\
& \leq C \|e^{-\tau \varphi} (a^{ij}-\delta^{ij}) \p_i (|x|^{\ell}H_{\ell}(\frac{x}{|x|})) \|_{L^2(B_{1})} .
\end{align*}
Now choosing $\tau = \ell + n -\delta$ for some $\delta>0$ yields the bound
\begin{align*}
\tau r^{-\ell-n-1+\delta} \| R\|_{L^2(B_{r})} 
& \leq C \|e^{-\tau \varphi} (a^{ij}-\delta^{ij})\p_i (|x|^{\ell}H_{\ell}(\frac{x}{|x|})) \|_{L^2(B_{1})}\\
& \leq C \mu (1+\ell),
\end{align*}
which implies the desired result for $r \in (0,1)$. 
\end{proof}

With Proposition \ref{prop:existence} in hand, the proof of Theorem \ref{prop:three_spheres_opt} is now immediate.

\begin{proof}[Proof of Theorem \ref{prop:three_spheres_opt}]
We insert the solutions from Proposition \ref{prop:existence} into equation \eqref{eq:stab_Carl}. Using that
\begin{align*}
\||\cdot |^{\ell} H(\frac{\cdot}{|\cdot|}) \|_{L^2(B_r)} = \frac{1}{\ell +n+1} r^{\ell + n +1},
\end{align*}
we have that for $\ell\in \N$ such that $C\mu \geq \frac{1}{\ell + n +1}$ it holds

\begin{align*}
C_{\ell,n,\mu} 2^{-\ell-n-1+\delta} \leq  \omega(\frac{C_{\ell,n,\mu}}{\ell + n+1}  4^{-\ell-n}),
\end{align*}
which implies that $\omega(t)\geq C_{n,\delta,\mu} t^{ \frac{\ell + 1 -\delta+n}{2(\ell + n)}}$ for all $t\geq 0$. Since $\frac{\ell + 1+n-\delta}{2(\ell + n)} \rightarrow \frac{1}{2}$ for $\ell \rightarrow 0$, for any given $\nu>0$, it is always possible to derive a lower bound of the form $\omega(t)\geq C_{\nu,n}t^{\frac{1}{2}+\nu}$, if $\mu>0$ (in condition (A2)) is chosen sufficiently small.
\end{proof}

\subsection{Carleman estimates imply lower bounds on singular values}
Next, we illustrate that quantitative propagation of smallness estimates (which can, for instance, be obtained by means of Carleman estimates) provide robust tools for deducing singular value bounds. 

\begin{proposition}[Lower bounds on the singular values]
Let $A:H^s_{\overline{W}} \rightarrow L^2(\Omega)$ be a compact, injective operator with dense range and with the singular values $\sigma_k$. Let $A^*$ denote its Hilbert space adjoint.
 Assume that for some constants $C,\mu>0$
\begin{align}
\label{eq:quant}
\|v\|_{L^2(\Omega)} \leq C e^{C \left(\frac{\|v\|_{L^2(\Omega)}}{\|v\|_{H^{-s}(\Omega)}}\right)^{\mu}}\|A^{\ast} v\|_{H^{s}_{\overline{W}}}.
\end{align}
Then, 
\begin{align*}
\sigma_j \geq C e^{-C j^{\mu s}}.
\end{align*}
\end{proposition}

\begin{proof}
Let $\{\varphi_j\}_{j\in \N}$ denote the Dirichlet eigenfunctions of the Laplacian on $\Omega$; note that they form a complete orthonormal system in $L^2(\Omega)$. Then,
\begin{align}
\label{eq:Lapl}
\frac{\|\varphi_j\|_{L^2(\Omega)}}{\|\varphi_j\|_{H^{-s}(\Omega)}} \sim j^{s}.
\end{align}
Using a max-min principle for self-adjoint operators with lower bound, 
we have that 
\begin{align*}
\sigma_j^2 &= \sup\limits_{\psi_1,\dots, \psi_j} \inf \limits_{\psi \in \spa\{\psi_1,\dots, \psi_j \}, \|\psi\|_{L^2(\Omega)}=1} (\psi, AA^{\ast} \psi)_{L^2(\Omega)}\\
&\geq  \inf\limits_{\psi \in \spa\{\varphi_1,\dots, \varphi_j \}, \|\psi\|_{L^2(\Omega)}=1} (\psi, AA^{\ast} \psi)_{L^2(\Omega)}.
\end{align*}
Here the functions $\varphi_j$ denote the eigenfunctions of the Dirichlet Laplacian. Thus, using the quantitative unique continuation result \eqref{eq:quant} together with \eqref{eq:Lapl}, we obtain
\begin{align*}
\sigma_j^2 
&\geq  \inf\limits_{\psi \in \spa\{\varphi_1,\dots, \varphi_j \}, \|\psi\|_{L^2(\Omega)}=1} (\psi, AA^{\ast} \psi)_{L^2(\Omega)}\\
&= \inf\limits_{\psi \in \spa\{\varphi_1,\dots, \varphi_j \}, \|\psi\|_{L^2(\Omega)}=1} (A^{\ast}\psi, A^{\ast} \psi)_{H^s_{\overline{W}}}\\
&\geq \inf\limits_{\psi \in \spa\{\varphi_1,\dots, \varphi_j \}, \|\psi\|_{L^2(\Omega)}=1} \|\psi\|_{L^2(\Omega)} 
e^{-C \left( \frac{\|\psi\|_{L^2(\Omega)}}{\|\psi\|_{H^{-s}(\Omega)}} \right)^{\mu}}\\
&\geq C e^{-C j^{s\mu}}.
\end{align*}
This concludes the proof.
\end{proof}

\section{Gevrey wave front sets and oscillatory integrals}
\label{sec:Gevrey}

In this section, we recall the basic definitions of Gevrey wave front sets, Gevrey pseudodifferential operators and of the asscoiated oscillatory integrals. More detailed information on these results (including the proofs) can be found in \cite{Rodino_book}. The results from this section are mainly used in Section \ref{sec:micro_smooth_1} on the microlocal smoothing properties.

\subsection{Gevrey wave front sets and smoothing operators}
\label{sec:analytic}

We will use the definitions from Section \ref{sec:Gev_def}. Let $U \subset \R^n$ be open. It is possible to endow the space $G^{\sigma}(U)$ with a Fr{\'e}chet topology. More precisely, for $f\in G^{\sigma}(U)$ and a sequence $(f_k)_{k=1}^{\infty}$ we have $f_k \rightarrow f$ if for every $K\subset U$ compact, there exists a sequence $\eps_k \rightarrow 0$, $\eps_k>0$, such that for some $C_K>0$
\begin{align*}
\sup\limits_{x\in K}|\partial^{\alpha}(f_k - f)(x)| \leq \eps_k C_K^{|\alpha|+1}|\alpha|^{\sigma |\alpha|}, \ \mbox{ for } \alpha \in (\N \cup \{0\})^n.
\end{align*}
If $\sigma > 1$ then $G^{\sigma}_c(U)$ has a natural inductive limit topology, see \cite[Section 1.4]{Rodino_book}.

As the dual object of $G^{\sigma}(U)$, for $\sigma\geq 1$ we define $\mathcal{E}'_{\sigma}(U)=(G^{\sigma}(U))'$. For $\sigma>1$, we further set $\mathcal{D}'_{\sigma}(U)=(G^{\sigma}_c(U))'$. Elements in $\mathcal{D}'_{\sigma}(U)$ are called Gevrey ultradistributions.

Now it is possible to give meaning to the $G^{\sigma} $ wave front set:

\begin{definition}
\label{defi:WF}
Let $\sigma > 1$, $u \in \mathcal{D}_{\sigma}'(U)$ and $(x_0,\xi_0)\in \R^n \times (\R^n\setminus \{0\})$. Then, $(x_0,\xi_0)$ is not in $WF_{G,\sigma}(u)$ if there exists $\varphi \in G^{\sigma}_c(U)$ with $\varphi(x_0) \neq 0$ and an open cone $\mathcal{C} \subset \R^n \setminus \{0\}$ with $\xi_0 \in \mathcal{C}$ and a constant $C>0$ such that
\begin{align*}
|\F(\varphi u)(\xi)| \leq C(CN)^N \langle \xi \rangle^{-N/\sigma} \mbox{ for all } \xi \in \mathcal{C}, \ N \in \N\cup \{0\}.
\end{align*} 
\end{definition}

With this in hand, we recall the notion of $G^{\sigma}$ symbols and their associated pseudodifferential operators (but will mainly rely on consideration for wave front sets in our arguments). The $G^{\sigma}$ symbols can be thought of as the analogue of classical $C^{\infty}$ symbols, satisfying additional quantitative estimates.

\begin{definition}[Definition 3.3.1 in \cite{Rodino_book}]
Let $U \subset \R^n$ be open and $\sigma \geq 1$. We define the symbol class $S^{m}_{\sigma}(U)$ to be the set of all $p(x,\xi)\in C^{\infty}(U \times \mR^n)$ such that for all $K \subset U$ compact there exist constants $C_K,B>0$ such that
\begin{align*}
|\partial^{\alpha}_x \partial^{\beta}_{\xi} p(x,\xi)| \leq C_K^{|\alpha|+|\beta|+1} |\alpha|^{\sigma |\alpha|} |\beta|^{ |\beta|} \langle \xi \rangle^{m-|\beta|}, \ \alpha, \beta \in (\N\cup \{0\})^n \mbox{ and } \langle \xi \rangle \geq B|\beta|^{\sigma}.
\end{align*}
\end{definition}

Given a symbol $p(x,\xi)\in S^m_{\sigma}(U)$, we define an associated pseudodifferential operator $p(x,D)$: For $u\in C_c^{\infty}(U)$ we set
\begin{align*}
p(x,D) u(x):= \int\limits_{\R^n} \int\limits_{\R^n} e^{i \xi \cdot (x-y)} p(x,\xi )u(y) d \xi  dy,
\end{align*}
where the integral is interpreted as an oscillatory integral.
We denote the set of these pseudodifferential operators by $\Psi^m_{\sigma}(U)$ and define $\Psi^{-\infty}_{\sigma}(U):= \bigcap\limits_{m \in \R } \Psi^{m}_{\sigma}(U)$. The operators in the class $\Psi^{-\infty}_{\sigma}(U)$ are $\sigma$-smoothing in the sense that they extend as maps from $\mathcal{E}'_{\sigma}(U)$ to $G^{\sigma}(U)$ (see \cite[Proposition 3.2.11]{Rodino_book}).
We remark that as in the $C^{\infty}$ setting the Schwartz kernel theorem is available which can be used to investigate linear operators and their regularity properties. In particular, $\sigma$-smoothing operators correspond to operators with kernels in $G^{\sigma}(U \times U)$ (see \cite[Chapter 1.5]{Rodino_book}).

Pseudodifferential operators in the class $\Psi^m_{\sigma}(U)$ are continuous as maps between $G_c^{\sigma}(U)$ and $G^{\sigma}(U)$ (see Theorem 3.2.3 in \cite{Rodino_book}) and also between the usual Sobolev classes (see for instance Theorem 3.1 in \cite{HuaRodino}), and they satisfy the usual composition and adjoint formulas (see for instance Theorems 3.4.12 and 3.4.13 in \cite{Rodino_book}).

\subsection{Oscillatory integrals and their wave front sets}
\label{sec:Gev_osc}

As a useful result for our applications, we observe the following estimate for the (analytic or Gevrey) wave front set of oscillatory integrals. We also refer to \cite{CZ} for a thorough discussion of oscillatory integrals in the Gevrey setting.

\begin{proposition}
\label{prop:analytic_FIO}
Let $k,n \in \N$ and let $a(x,\theta)\in S^m(\R^n, \R^k)$.
Let 
\begin{align*}
I_{\phi}(x) = \int\limits_{\R^k} e^{i\phi(x,\theta)}a(x,\theta)d\theta
\end{align*}
be an oscillatory integral with an analytic in ($x,\theta$), one-homogeneous (in $\theta$) phase $\phi$ and an analytic amplitude $a$ in $(x,\theta)$. Then,
\begin{align*}
WF_A(I_{\phi}) \subset \{(x,\phi_x'(x,\theta)): \ (x,\theta)\in \Omega \times \R^k \mbox{ and } \phi'_{\theta}(x,\theta)=0\}=:\mathcal{C}_{\phi},
\end{align*}
where $\Omega \subset \R^n$ is an open set. 
An analogous result holds true if analyticity is replaced by $\sigma$-Gevrey regularity, $\sigma>1$.
\end{proposition}

This result is standard in the smooth setting, but becomes particularly useful for our purposes in the context of the Gevrey microlocal smoothing properties of Fourier integral operators such as for instance the Radon transform. 

In the setting of the spaces $G^{\sigma}$ with $\sigma>1$ non-trivial, there exist smooth cut-off functions, e.g. $f(x)= \exp(-|x|^{-\frac{1}{\sigma-1}})$. For $\sigma=1$ only families of nearly analytic cut-off functions can be obtained:

\begin{lemma}[Lemma 1.1 in Chapter 5 in \cite{Treves}]
\label{lem:spatial_cutoff}
There exists a constant $C_{\ast}>0$ depending only on the space dimension $n$ such that for all $U \subset \R^n$ open, for all $d>0$ and $N>0$ there exists a family $\psi_N \in C^{\infty}(\R^n)$ with the property that
\begin{itemize}
\item $0\leq \psi_N \leq 1$,
\item $\psi_N = 1$ in $U$, and $\psi_N(x)=0$ if $\dist(x,U)>d$,
\item  $|D^{\alpha} \psi_N| \leq \left( \frac{C_{\ast} N}{d} \right)^{|\alpha|}$ for all $\alpha \in \Z^n$, $|\alpha | \leq N$.
\end{itemize}
\end{lemma}

With Lemma \ref{lem:spatial_cutoff} in hand, we provide the proof of Proposition \ref{prop:analytic_FIO}:

\begin{proof}[Proof of Proposition \ref{prop:analytic_FIO}]
The proof follows as, for example, in \cite[Theorem 0.5.1]{Sogge} where all cut-offs are replaced by the ones from Lemma \ref{lem:spatial_cutoff}. As we used this result in Section \ref{sec:micro_smooth_1}, we present the details of the argument (but only consider the analytic case since the Gevrey case closely mirrors the smooth setting with smooth cut-off functions being replaced by Gevrey cut-off functions).

The well-definedness of the oscillatory integral follows from the standard arguments and does not use analyticity. Hence, we only discuss the estimates for the wave front set. To this end, we seek to prove that
\begin{align}
\label{eq:ana_WF_set}
\begin{split}
|I(\xi)| 
&:= \left|\int\int e^{i(\phi(x,\theta)- x\cdot \xi)} \chi_{x_0,N}(x) a(x,\theta) d\theta dx \right| \\
&\leq (C N)^N (1+|\xi|)^{-N}
\end{split}
\end{align}
for every $N\in \N$ and $(x,\xi) \in \R^n \times \R^k$ such that $(x,\xi) \notin \mathcal{C}_{\phi}$. Here $\chi_{x_0,N}$ denotes a cut-off function as in Lemma \ref{lem:spatial_cutoff} which localizes around a point $x_0 \in \R^n$. In order to infer \eqref{eq:ana_WF_set}, we localize further by using a partition of unity consisting of the cut-off functions $\chi_{N,j}(\theta) = \chi_{N,1}(2^{-j}|\theta|)$ with $\supp(\chi_{N,1}(x)) \subset \{2^{-1}\leq |\theta| \leq 2\}$ if $j\geq 1$, and $j\in \N$ and $\supp(\chi_{N,0}) \subset B_{4}$. Here the functions $\chi_{N,1}$ are as in Lemma \ref{lem:spatial_cutoff} (which ensures that also the functions $\chi_{N,j}$ satisfy suitable Cauchy bounds for derivatives up to order $N$).
Then, 
\begin{align*}
|I(\xi)|
& \leq \sum\limits_{j=0}^{\infty} \left| \int\int e^{i(\phi(x,\theta)-x\cdot \xi)} \chi_{x_0,N}(x)\chi_{N,j}(\theta)a(x,\theta) d\theta dx \right|.
\end{align*}
Hence, in order to infer \eqref{eq:ana_WF_set}, after changing coordinates, it suffices to bound
\begin{align}
\label{eq:stationary_phase_est}
\begin{split}
&\left| \int\int e^{i(2^j\phi(x,\theta)-x\cdot \xi)} \chi_{x_0,N}(x)\chi_{N,1}(\theta) a(x, 2^j\theta) d\theta dx \right|\\
&\leq (CN)^N (2^{j} + |\xi|)^{-N+m} \mbox{ for all } N \in \N.
\end{split}
\end{align}
To this end, we use the method of stationary phase with the large parameter $2^j +|\xi|$ in combination with the  analyticity of $\phi$ and $a$: As in the non-analytic case, we observe that the phase function 
\begin{align*}
\Phi(x,\theta):= \frac{2^j\phi(x,\theta)-x \cdot \xi}{2^{j} + |\xi|}
\end{align*}
has the property that $|\nabla_{x,\theta}\Phi|\geq c >0$ if $(x,\theta) \notin \mathcal{C}_{\phi}$. In particular, if $\chi_{x_0,N}$ localizes to a sufficiently small neighbourhood of $x_0$, there is a direction $\nu \in S^{n+k-1}$ such that $|\nu\cdot \nabla_{x,\theta}\Phi | \geq c >0$ in $\supp(\chi_{x_0,N})$. As a consequence, we apply the method of stationary phase. After a change of coordinates we may assume that $\nu$ either points into the $x_1$ or the $\theta_1$ direction. Using the analyticity of $a$, $\Phi$, $\phi$ (and the fact that $a$ is a symbol of order $m$), we obtain the estimates
\begin{align}
\label{eq:analytic_1}
\begin{split}
|\nabla^{\alpha}_{x,\theta} a(x,2^j\theta)|, |\nabla^{\alpha}_{x,\theta} \Phi(x,\theta)|,
|\nabla^{\alpha}_{x,\theta}\Phi(x,\theta)| \leq (C|\alpha|)^{|\alpha|}(1+2^j|\theta|)^m \mbox{ for any } \alpha \in \N^n.
\end{split}
\end{align}
Moreover, the estimates for the cut-off functions from Lemma \ref{lem:spatial_cutoff} yield
\begin{align}
\label{eq:analytic_2}
|\nabla^{\alpha}_{\theta} \chi_{N,1}(\theta)| \leq \left(CN \right)^{N} \mbox{ for all } \alpha \in \N^n \mbox{ and }  |\alpha| \leq N.
\end{align} 
Combining \eqref{eq:analytic_1}, \eqref{eq:analytic_2} and setting $L(x,D) := \frac{1}{i (2^j+|\xi|) (\nu\cdot \nabla_{x,\theta} \Phi) } (\nu \cdot \nabla_{x,\theta} )$, we hence arrive at 
\begin{align}
\label{eq:stationary_phase_est_a}
\begin{split}
&\left| \int\int e^{i 2^j(\phi(x,\theta)-x\cdot \xi)} (L(x,D)^{\ast})^{N}(\chi_{x_0,N}(x) \chi_{N,1}(\theta) a(x, 2^j\theta)) d\theta dx \right|\\
&\leq (CN)^N (2^j + |\xi|)^{-N+m} \mbox{ for all } N \in \N.
\end{split}
\end{align}
Summing up \eqref{eq:stationary_phase_est_a} over $j\in \N$  implies the bound \eqref{eq:ana_WF_set}.
\end{proof}

\end{appendix}

\section*{Remarks on the revised version}
This is a revised version of the article ``On instability mechanisms for inverse problems'' Ars Inveniendi Analytica (2021), Paper No. 7, 93 pp by the same authors. In this revised version we have corrected several typographical errors, reworded the assumptions in Theorem 1.2 to be parallel to \cite{Mandache}, added some small additional explanations in its proof and have corrected the assumption on the closed sets in Theorems 4.1 and 4.2. to exclude singletons.

%%      ---------------------------------------------------------------------
%%      --------------------------- BIBLIOGRAPHY ----------------------------
%%      ---------------------------------------------------------------------
%% PUT HERE THE BIBLIOGRAPHY IN YOUR FAVOURITE FORMAT
%% Please check that the format of the bibliography is uniform and coherent

\bibliographystyle{alpha}

\begin{thebibliography}{CNYY09}
\bibitem[Al88]{Alessandrini}
G.\ Alessandrini, \textit{{Stable determination of conductivity by boundary measurements}}, Appl. Anal. {\bf 27} (1988), 153--172.


\bibitem[Al07]{A07}
G.\ Alessandrini.
\newblock Open issues of stability for the inverse conductivity problem.
\newblock {\em Journal of Inverse and Ill-posed Problems}, 15(5):451--460, 2007.

\bibitem[AK12]{Alessandrini_Kim}
G.\ Alessandrini, K.\ Kim, \textit{{Single-logarithmic stability for the Calder{\'o}n problem with local data}}, Journal of Inverse and Ill-Posed Problems 20.4 (2012): 389-400.

\bibitem[ARRV09]{AlessandriniRondiRossetVessella} G.\ Alessandrini, L.\ Rondi, E.\ Rosset, S.\ Vessella, \textit{{The stability for the Cauchy problem for elliptic equations}}, Inverse Problems {\bf 25} (2009), 123004.

\bibitem[AV05]{AlessandriniVessella}
G.\ Alessandrini, S.\ Vessella, Lipschitz stability for the inverse conductivity problem, Adv. Appl. Math., Vol. 35 (2005), 207-241.


\bibitem[AKK+04]{AKKLT}
M.\ Anderson, A.\ Katsuda, Y.\ Kurylev, M.\ Lassas, M.\ Taylor, \textit{{Boundary regularity for the Ricci equation, geometric convergence, and Gel'fand's inverse boundary problem}}, Inventiones Math. 158 (2004), 261--32.

\bibitem[AB18]{AnderssonBoman}
J.\ Andersson, J.\ Boman, Stability estimates for the local Radon transform, Inverse Problems 34 (2018), 034004.

\bibitem[AS20]{AS20}
Y. \ M. \ Assylbekov, P.\ Stefanov, \textit{{Sharp stability estimate for the geodesic ray transform}}, Inverse problems {\bf 36}(2), 2020.

\bibitem[ABES19]{ABES19}
P.\ Auscher, S.\ Bortz, M.\ Egert, O.\ Saari.
\newblock {\em On regularity of weak solutions to linear parabolic systems with
  measurable coefficients}, 
\newblock J. Math. Pures Appl.,
  121:216--243, 2019.
  
  

\bibitem[BNR18]{BNR}
L. Bandara, M. Nursultanov, J. Rowlett, \textit{{Eigenvalue asymptotics for weighted Laplace equations on rough Riemannian manifolds with boundary}}, arXiv:1811.08217.s  
  
  
\bibitem[BLR92]{BardosLebeauRauch}
C. \ Bardos, G. \ Lebeau, J.\ Rauch,
\textit{{Sharp sufficient conditions for the observation, control, and stabilization of waves from the boundary}}, SIAM J. Control Optim. 30(5), pp. 1024-1065, 1992.  
  
  
  
\bibitem[BH03]{BH03}
M.\ Bebendorf, W.\ Hackbusch.
\newblock Existence of $\mathcal{H}$-matrix approximants to the inverse
  {F}{E}-matrix of elliptic operators with {$L^{\infty}$}-coefficients.
\newblock {\em Numerische Mathematik}, 95(1):1--28, 2003. 

\bibitem[BD11]{BellassouedDosSantos}
M.\ Bellassoued, D.\ Dos Santos Ferreira, \textit{{Stability estimates for the anisotropic wave equation from the Dirichlet-to-Neumann map}}, Inverse Probl. Imaging {\bf 5} (2011), 745--773.


\bibitem[BDFS16]{BDHFS16}
E. \ Beretta, M. \ V. \ de Hoop, F. \ Faucher, O.\ Scherzer, 
\textit{{Inverse boundary value problem for the Helmholtz equation: quantitative conditional Lipschitz stability estimates.}} 
SIAM Journal on Mathematical Analysis 48.6 (2016): 3962-3983.



\bibitem[Be87]{Besse}
A.\ L.\ Besse, Einstein manifolds. Springer, 1987.



\bibitem[BS72]{BirmanSolomyak_nonsmooth}
M. Sh. Birman, M. Z. Solomyak, \textit{{Spectral asymptotics of nonsmooth elliptic operators. {I},
              {II}}}, Trudy Moskov. Mat. Ob\v{s}\v{c}. 27 (1972), 3--52; ibid. 28 (1973), 3--34.

\bibitem[BS77a]{BS77}
M. Sh. Birman, M. Z. Solomyak. \textit{{Estimates of singular numbers of integral operators}}, Russian Mathematical Surveys 32.1 (1977): 15.

\bibitem[BS77b]{BS_psdo1}
M. Sh. Birman, M. Z. Solomyak. \textit{{Asymptotic behavior of the spectrum of pseudodifferential operators with anisotropically homogeneous symbols}} (Russian), Vestnik Leningrad. Univ. 1977, no. 13 Mat. Meh. Astronom. vyp. 3, 13--21, 169. 

\bibitem[BS79]{BS_psdo2}
M. Sh. Birman, M. Z. Solomyak. \textit{{Asymptotic behavior of the spectrum of pseudodifferential operators with anisotropically homogeneous symbols. II}} (Russian),Vestnik Leningrad. Univ. Mat. Mekh. Astronom. 1979, vyp. 3, 5--10, 121. 


 \bibitem[BGKP19]{BGKP19}
H.\ Bolcskei, P.\ Grohs, G.\ Kutyniok, P.\ Petersen, \emph{Optimal approximation with sparsely connected deep neural networks.} SIAM Journal on Mathematics of Data Science, 1(1), 8-45, 2019.

\bibitem[BKL17]{BKL}
R.\ Bosi, Y.\ Kurylev, M.\ Lassas, \textit{{Reconstruction and stability in Gel'fand's inverse interior spectral problem}}, arXiv:1702.07937.


\bibitem[Ca80]{Calderon}
A.\ Calder{\'o}n,
\textit{{On an inverse boundary value problem}},
Seminar on Numerical Analysis and its Applications to Continuum Physics (Rio de Janeiro, 1980), pp.
65--73, Soc. Brasil. Maat., Rio de Janeiro, 1980.

\bibitem[CS90]{CS90}
B. Carl, I. Stephani,
\textit{{Entropy, compactness and the approximation of operators: Operator theoretical methods in the local theory of Banach spaces}}, 1990.

\bibitem[CDR16]{CDR}
P.\ Caro, D.\ Dos Santos Ferreira, A.\ Ruiz, Stability estimates for the Calder\'on problem with partial data, J. Diff. Eq. 260 (2016), no. 3, 2457--2489.

\bibitem[CS14]{CaroSalo}
P.\ Caro, M.\ Salo, Stability of the Calder\'on problem in admissible geometries.
Inverse Probl. Imaging 8 (2014), no. 4, 939-957.

\bibitem[CZ90]{CZ}
L. \ Cattabriga, L.\ Zanghirati,
\textit{{Fourier integral operators of infinite order on Gevrey spaces. Applications to the Cauchy problem for certain hyperbolic operators}}, Kyoto Daigaku Rigakubu Sugaku Kyoshitsu, 1990.



\bibitem[Ch90]{Chanillo}
S.\ Chanillo, \textit{{A problem in electrical prospection and a $n$-dimensional Borg-Levinson theorem}}, Proc. AMS {\bf 108} (1990), 761--767.

\bibitem[DR03]{DiCristoRondi}
M.\ Di Cristo and L.\ Rondi. Examples of exponential instability for inverse inclusion and scattering problems. Inverse Problems, 19(3):685, 2003.

\bibitem[ET08]{ET08}
D. \ E.\ Edmunds, H.\ Triebel. \textit{{Function spaces, entropy numbers, differential operators}}. Vol. 120. Cambridge University Press, 2008.

\bibitem[EHN96]{EHN}
H. Engl, M. Hanke, A. Neubauer, Regularization of inverse problems. Kluwer Academic Publishers Group, 1996.

\bibitem[EKN89]{EKN89}
H. Engl, K. Kunisch, A. Neubauer, Convergence rates for Tikhonov regularisation of non-linear ill-posed problems. Inverse problems 5(4), p.523, 1989.



\bibitem[EZ18]{EZ18}
B.\ Engquist and H.\ Zhao.
\newblock Approximate separability of the {G}reen's function of the {H}elmholtz
  equation in the high frequency limit.
\newblock {\em Communications on Pure and Applied Mathematics},
  71(11):2220--2274, 2018.


\bibitem[Fe83]{Fefferman}
C. L. \ Fefferman, \textit{{The uncertainty principle}}, Bulletin of the American Mathematical Society 9(2), pp. 129--206, 1983.

\bibitem[FCZ00]{FernandezCaraZuazua}
E. \ Fern{\'a}ndez-Cara, \ E.\ Zuazua, \textit{{The cost of approximate controllability for heat equations: the linear case}}, Advances in Differential equations 5(4-6), pp. 465-514, 2000.

\bibitem[GN21]{GN21}
E. \ Gin\'e, R. \ Nickl. \emph{Mathematical foundations of infinite-dimensional statistical models.} Cambridge University Press, 2021.

\bibitem[GRSU20]{GRSU}
T. \ Ghosh, A.\ R{\"u}land, M.\ Salo, G. \ Uhlmann, \textit{{Uniqueness and reconstruction for the fractional Caldern problem with a single measurement }},
J. Funct. Anal. 279(1), 108505, 2020.


\bibitem[Gr87]{Gramtchev}
T. V. Gramchev.
\newblock The stationary phase method in {G}evrey classes and {F}ourier integral operators on ultradistributions.
\newblock {\em Banach Center Publications}, 19(1):101--112, 1987.


\bibitem[Ha15]{Haberman_conductivity}
B.\ Haberman, \textit{{Uniqueness in Calder\'on's problem for conductivities with unbounded gradient}}, Comm. Math. Phys. {\bf 340} (2015), 639--659.

\bibitem[Ha17]{Haberman_magnetic}
B.\ Haberman, \textit{{Unique determination of a magnetic Schr\"odinger operator with unbounded magnetic potential from boundary data}}, Int. Math. Res. Not. 2018.4 (2018): 1080-1128.

\bibitem[Ha23]{Hadamard}
J.\ Hadamard, Lectures on Cauchy's problem in linear partial differential equations, Yale University Press, 1923.



\bibitem[Ho00]{H00}
T.\ Hohage, \textit{Regularization of exponentially ill-posed problems}, Numerical functional analysis and optimization 21 (3-4), pp. 439-464, 2000.




\bibitem[HU18]{HU18}
S.\ Holman, G.\ Uhlmann,
\newblock On the microlocal analysis of the geodesic X-ray transform with conjugate points.
\newblock {\em Journal of Differential Geometry}, 108(3):459--494, 2018.

\bibitem[H{\"o}71]{Hormander_FIO1}
L. \ H\"ormander, \textit{Fourier integral operators. I},
Acta mathematica, vol. 127 (1), 1971.

\bibitem[H{\"o}85]{Hoermander}
L. \ H\"ormander, \textit{The analysis of linear partial differential operators}, vols.\ I-IV. Springer-Verlag, 1983--1985.

\bibitem[HI04]{HI04}
T. \ Hrycak, V. \ Isakov, \textit{{Increased stability in the continuation of solutions to the Helmholtz equation}}, Inverse Problems 20(3), pp.697, 2004.

\bibitem[HR01]{HuaRodino}
C.\ Hua and L.\ Rodino.
\newblock Paradifferential calculus in {G}evrey classes.
\newblock {\em Journal of Mathematics of Kyoto University}, 41(1):1--31, 2001.



\bibitem[Is11]{Isaev}
M.\ I.\ Isaev, \textit{Exponential instability in the Gel'fand inverse problem on the energy intervals}, Journal of Inverse and Ill-posed Problems 19.3 (2011): 453-472.

\bibitem[Is13a]{IsaevI}
M. \ I.\ Isaev, \textit{Instabilities in the Gel'fand inverse problem at high energies}, Applicable Analysis 92.11 (2013): 2262-2274.

\bibitem[Is13b]{IsaevII}
M. \ I.\ Isaev, \textit{Exponential instability in the inverse scattering problem on the energy interval}, Functional Analysis and Its Applications 47.3 (2013): 187-194.

\bibitem[Is90]{Isakov}
V.\ Isakov.
\newblock {\em Inverse source problems}.
\newblock Number~34. American Mathematical Soc., 1990.

\bibitem[JK95]{JerisonKenig}
D.\ Jerison, C.\ E.\ Kenig, The inhomogeneous Dirichlet problem on Lipschitz domains, J. Funct. Anal. 130 (1995), 161--219.

\bibitem[Jo60]{John}
F. \ John, \textit{Continuous dependence on data for solutions of partial differential equations with a prescribed bound},
Communications on pure and applied mathematics 13.4 (1960): 551-585

\bibitem[Ka15]{Karol}
A. I. Karol', \textit{{Asymptotic behavior of the spectrum of pseudodifferential operators of variable order}}, J. Math. Sci. (N.Y.) 207 (2015), no. 2, Problems in mathematical analysis. No. 78 (Russian), 236--248.

\bibitem[Kl06]{Klibanov_parabolic}
M.\ V.\ Klibanov.
\newblock Estimates of initial conditions of parabolic equations and
  inequalities via lateral {C}auchy data.
\newblock {\em Inverse problems}, 22(2):495, 2006.


\bibitem[KT01]{KochTataru01}
H. \ Koch, D. Tataru, \textit{{Carleman estimates and unique continuation for second-order elliptic equations with nonsmooth coefficients}}, Comm. Pure Appl. Math. 54(3), pp.339-360, 2001.

\bibitem[KT09]{KochTataru}
H. \ Koch, D. Tataru, \textit{{Carleman estimates and unique continuation for second order parabolic equations with nonsmooth coefficients}}, Comm. PDE 34(4), pp.305-366, 2009.


\bibitem[KT59]{KolmogorovTikhomirov}
A.N.\ Kolmogorov, V.M.\ Tikhomirov, \textit{{$\eps$-entropy and $\eps$-capacity in functional spaces}}, Usp. Mat. Nauk {\bf 14} (1959), 3--86 (in Russian) (Engl. Transl. 1961 Am. Math. Soc. Transl. 17, 277--364).

\bibitem[K{\"o}13]{K13}
H. \ K{\"o}nig, \textit{{Eigenvalue distribution of compact operators}}, Vol. 16, Birkh{\"a}user, 2013.


\bibitem[LL18]{LaurentLeautaud}
C. \ Laurent, \ M. \ L{\'e}autaud, \textit{{Quantitative unique continuation for operators with partially analytic coefficients. Applications to approximate control for waves}}, JEMS 2018.

\bibitem[LN91]{LavineNachman}
R.~Lavine, A.~Nachman, announced in A.\ Nachman, \textit{Inverse scattering
at fixed energy}, Proceedings of the Xth Congress on Mathematical Physics, L. Schm\"udgen (Ed.), Leipzig, Germany, 1991, 434--441, Springer-Verlag. 


\bibitem[Le92]{Lebeau}
G. \ Lebeau, \textit{{Contr{\^o}le analytique I: esstimations a priori}}, Duke Math. J. 68(1), pp.1-30, 1992. 

\bibitem[LV85a]{LV85a}
H. \ A.\ Levine, S.\ Vessella, 
\textit{{Stabilization and regularization for solutions of some ill-posed problem for the wave equation}}, Math. Methods Appl. Sci 7 (1985), no.2, 202-209.


\bibitem[LV85b]{LV85b}
H. \ A.\ Levine, S.\ Vessella, 
\textit{{Stabilization and regularization for solutions of some ill-posed problems of elliptic and parabolic type}}, Rend. Circ. Mat. Palermo (2) 34 (1985), no.1, 141-160.

\bibitem[LM12a]{LionsMagenes}
J.\ L. Lions, E. Magenes. \textit{Non-homogeneous boundary value problems and applications, Vol. 1}, Springer Science and Business Media, 2012.

\bibitem[LM12b]{LionsMagenesII}
J.\ L. Lions, E. Magenes. \textit{Non-homogeneous boundary value problems and applications, Vol. 2}, Springer Science and Business Media, 2012.


\bibitem[MP85]{MP85}
R. \ Magnanini, G. \ Papi,
\textit{{An inverse problem for the Helmholtz equation}}, Inverse Problems 1 (1985), no. 4, 357-370.


\bibitem[Ma01]{Mandache}
N.\ Mandache, \textit{{Exponential instability in an inverse problem for the Schr\"odinger equation}}, Inverse Problems, 17(5):1435, 2001.

\bibitem[MP03]{MP03}
P. \ Math\'e, S.\ V.\ Pereverzev, \textit{{Geometry of linear ill-posed problems in variable Hilbert scales.}}, Inverse problems 19.3 (2003): 789.


\bibitem[Mc00]{McLean}
W. \ McLean, \textit{{Strongly elliptic systems and boundary integral equations}}, Cambridge University Press, 2000.




\bibitem[Me]{Melrose}
R.\ Melrose, \textit{{Functional Analysis}}, lecture notes, found at \url{https://math.mit.edu/~rbm/18-102-S14/FunctAnal.pdf}.

\bibitem[Me63]{M63}
N. G. Meyers,
\newblock An {$ L^{p} $}-estimate for the gradient of solutions of second order
  elliptic divergence equations.
\newblock {\em Annali della Scuola Normale Superiore di Pisa-Classe di
  Scienze}, 17(3):189--206, 1963.

\bibitem[Mi14]{Mizohata} 
S.\ Mizohata, \textit{On the Cauchy problem}, Vol. 3. Academic Press, 2014.

\bibitem[Mo14]{Montalto}
C.\ Montalto, Stable determination of a simple metric, a covector field and a potential from the hyperbolic Dirichlet-to-Neumann map, Comm. PDE 39(1):120--145, 2014.



\bibitem[MSU15]{MSU15}
F.\ Monard, P.\ Stefanov, G.\ Uhlmann, \textit{{The geodesic ray transform on Riemannian surfaces with conjugate points}}, Comm. Math. Phys. {\bf 337} (2015), no. 3, 1491--1513.



\bibitem[Na88]{Nachman}
A.\ Nachman,
\textit{{Reconstructions from boundary measurements}},
Ann. of Math. {\bf 128} (1988), 531--576.

\bibitem[Na01]{Natterer_book}
F.\ Natterer, The mathematics of computerized tomography. SIAM Classics in Applied Mathematics 32, 2001.

\bibitem[No88]{Novikov}
R.\ G.\ Novikov, \textit{{Multidimensional inverse spectral problem for the equation {$\Delta \psi +(v (x)+u (x)) \psi= 0$}}}, Functional Analysis and Its Applications 22(4), pp.263-272, 1988.


\bibitem[PSU12]{PSU_GAFA} G.P.\ Paternain, M.\ Salo, G.\ Uhlmann, {\it The attenuated ray transform for connections and Higgs fields},  Geom. Funct. Anal. {\bf 22} (2012), 1460--1489.

\bibitem[PSU20+]{PSU20+} G.P.\ Paternain, M.\ Salo, G.\ Uhlmann, {\it Geometric inverse problems in 2D}, textbook in preparation.

\bibitem[PU05]{PestovUhlmann} L.\ Pestov, G.\ Uhlmann,
{\it Two dimensional compact simple Riemannian manifolds are
boundary distance rigid},
Ann. of Math. {\bf 161} (2005), 1089--1106.

\bibitem[Ph04]{Phung}
K.-D.\ Phung, 
\textit{{Note on the cost of the approximate controllability for the heat equation with potential}}
J. Math. Anal. Appl. {\bf 295} (2004), 527--538.



\bibitem[Qu93]{Quinto}
E.\ T.\ Quinto, \textit{{Singularities of the X-ray transform and limited data tomography in $\R^{2}$ and $\R^{3}$}},
SIAM Journal on Mathematical Analysis
{\bf 24} (1993), no.5, 1215-1225.

\bibitem[Qu06]{Quinto1}
E.\ T.\ Quinto, \textit{{An introduction to X-ray tomography and Radon transforms}},
Proceedings of symposia in Applied Mathematics
{\bf 63} (2006).


\bibitem[Ro93]{Rodino_book}
L.\ Rodino.
\newblock {\em Linear partial differential operators in Gevrey spaces}.
\newblock World Scientific, 1993.

\bibitem[Ro06]{R06}
L.\ Rondi.
\newblock A remark on a paper by Alessandrini and Vessella.
\newblock {\em Advances in Applied Mathematics}, 36(1):67--69, 2006.


\bibitem[RS18]{RulandSalo_instability}
A.\ R{\"u}land, M.\ Salo, \textit{{Exponential instability in the fractional Calder\'on problem}}, Inverse Problems 34(4), p. 045003, 2018.


\bibitem[SV93]{SV93}
E. \ Scalas, G.\ A.\ Viano \textit{{$\epsilon$-entropy and $\epsilon$-capacity in the theory of ill-posed problems}}, Inverse Problems 9 (1993) 545-550.

\bibitem[Sh87]{Shubin}
M. \ A. \ Shubin, \textit{{Pseudodifferential operators and spectral theory}}, Vol. 200. No. 1. Berlin: Springer-Verlag, 1987.

\bibitem[SP61]{SP61}
D.\ Slepian, H. \ O. \ Pollak, \textit{{Prolate spheroidal wave functions, Fourier analysis and uncertainty I}}, Bell Labs Technical Journal {\bf 40} (1961), no. 1, 43--63.

\bibitem[So17]{Sogge}
C. \ D. \ Sogge, \textit{{Fourier integrals in classical analysis}}, Vol. 210, Cambridge University Press, 2017. 


\bibitem[SU04]{SU04}
P.\ Stefanov, G.\ Uhlmann, Stability estimates for the X-ray transform of tensor fields and boundary rigidity,  Duke Math. J., 123(3):445--467, 2004.

\bibitem[SU05]{SU05}
P.\ Stefanov, G.\ Uhlmann, Stable determination of generic simple metrics from the hyperbolic Dirichlet-to-Neumann map, Int. Math. Res. Not., (17):1047-1061, 2005.

\bibitem[SU08]{SU08}
P.\ Stefanov, G.\ Uhlmann, Boundary rigidity and stability for generic simple metrics, J. Amer. Math. Soc., 18(4):975--1003, 2005.


\bibitem[SU09]{SU09}
P.\ Stefanov and G.\ Uhlmann.
\newblock Linearizing non-linear inverse problems and an application to inverse backscattering.
\newblock {\em Journal of Functional Analysis}, 256(9):2842--2866, 2009.

\bibitem[SU12]{StefanovUhlmann}
P.\ Stefanov, G. Uhlmann, \textit{{The geodesic X-ray transform with fold caustics}}, Analysis \& PDE 5(2), pp. 219-260, 2012.

\bibitem[SU13]{SU13}
P.\ Stefanov and G.\ Uhlmann.
\newblock Instability of the linearized problem in multiwave tomography of recovery both the source and the speed.
\newblock {\em Inverse Problems \& Imaging}, 7(4):1367, 2013.

\bibitem[SY18]{StefanovYang}
P.\ Stefanov, Y.\ Yang, The inverse problem for the Dirichlet-to-Neumann map on Lorentzian manifolds, Analysis \& PDE, 11, no. 6, 1381--1414, 2018.


\bibitem[St70]{Stein}
E. \ Stein, \textit{{Singular integrals and differentiability properties of functions.}} Princeton university press, 1970.



\bibitem[SU87]{SylvesterUhlmann}
J.\ Sylvester, G.\ Uhlmann,
\textit{{A global uniqueness theorem for an inverse boundary value problem}}, Ann. of Math. {\bf 125} (1987), 153--169.

\bibitem[TV82]{TV82}
G.\ Talenti, S.\ Vessella, 
\textit{{A note on an ill-posed problem for the heat equation}}, J. Austral. Math. Soc. Ser. A 32 (1982), no. 3, 358-368.

\bibitem[Ta84]{Taniguchi}
K. \ Taniguchi, 
\textit{{Fourier integral operators in Gevrey class on {$R^n$} and the fundamental solution for a hyperbolic operator}}, Publications of the Research Institute for Mathematical Sciences {\bf 20}(3) (1984), 491-542.

\bibitem[Ta99]{Tataru_wave}
D.\ Tataru,
\textit{{Unique continuation for operators with partially analytic coefficients}}, Journal de math{\'e}matiques pures et appliqu{\'e}es {\bf 78}(5) (1999), 505-521.


\bibitem[Ta81]{Taylor_PsiDO}
M. \ E. \ Taylor, \textit{{Pseudodifferential operators}}, volume 34 of Princeton Mathematical Series, 1981.

\bibitem[Ta96]{Taylor}
M. \ E. \ Taylor, \textit{{Partial differential equations}}, vols. I--III, Springer, 1996.


\bibitem[Ti43]{Tikhonov}
A.\ N.\ Tikhonov. On the stability of inverse problems. C. R. (Doklady) Acad. Sci. URSS (N.S.), 39:176--179, 1943.

\bibitem[Tr80]{Treves}
F. \ Tr{\`e}ves, \textit{{Introduction to pseudodifferential and Fourier integral operators Volume 1: Pseudodifferential operators}}, Vol. 1. Springer Science and Business Media, 1980.

\bibitem[Tr97]{Triebel_fractals}
H.\ Triebel, Fractals and spectra. Birkh\"auser, 1997.

\bibitem[Uh14]{Uhlmann_survey}
G.\ Uhlmann, \textit{{Inverse problems: seeing the unseen}}, Bull. Math. Sci. {\bf 4} (2014), 209--279.

\bibitem[Ve09]{Vessella}
S. \ Vessella, \textit{{Unique continuation properties and quantitative estimates of unique continuation for parabolic equations}}, Handbook of differential equations: evolutionary equations {\bf 5} (2009), 421-500.


\bibitem[Ya09]{Yamamoto}
M.\ Yamamoto.
\newblock Carleman estimates for parabolic equations and applications.
\newblock {\em Inverse problems}, 25(12):123013, 2009.

\bibitem[ZZ19]{ZZ19}
H.\ Zhao, Y.\ Zhong.
\newblock Instability of an inverse problem for the stationary radiative
  transport near the diffusion limit.
\newblock {\em SIAM Journal on Mathematical Analysis}, 51(5):3750--3768, 2019.


\bibitem[Zw12]{Zworski}
M.\ Zworski, Semiclassical analysis. AMS, 2012.




\end{thebibliography}

\end{document}